\documentclass[aop]{imsart}

\RequirePackage{amsthm,amsmath,amsfonts,amssymb}
\RequirePackage[numbers]{natbib}
\RequirePackage[colorlinks,citecolor=blue,urlcolor=blue]{hyperref}%% uncomment this for coloring bibliography citations and linked URLs
\RequirePackage{graphicx}%% uncomment this for including figures

\usepackage{mathtools}
\usepackage{lmodern}
\usepackage{tcolorbox}
\usepackage[tickmarkheight=.5em,textwidth=\marginparwidth,textsize=small]{todonotes}
\usepackage{color}
\definecolor{RubineRed}{HTML}{E6004C}
\definecolor{NiceBlue}{HTML}{0099E6}
\definecolor{NiceGreen}{HTML}{3BB300}
\definecolor{NiceGray}{HTML}{818589}
\definecolor{RED}{HTML}{FF0000}
\definecolor{DarkRed}{HTML}{8B0000}
\DeclareRobustCommand{\neswarrow}{%
  \mathrel{\text{\ooalign{$\swarrow$\cr$\nearrow$}}}%
}
\newcommand{\setword}[2]{% %command to refer to a word in the document takes two arguments, first is the sentence and the second is the label. Use the label to refer to the sentence anywhere in the document.
  \phantomsection
  #1\def\@currentlabel{#1}\label{#2}%
}
 
\graphicspath{{resources/}} %all pictures are in this subfolder

\startlocaldefs
\theoremstyle{plain}
\newtheorem{theorem}{Theorem}
\newtheorem{proposition}[theorem]{Proposition}
\newtheorem{lemma}[theorem]{Lemma}
\newtheorem{corollary}[theorem]{Corollary}
\theoremstyle{remark}
\newtheorem{definition}[theorem]{Definition}
\newtheorem{example}[theorem]{Example}
\newtheorem{remark}[theorem]{Remark}
\endlocaldefs

\begin{document}

\begin{frontmatter}
\title{Genealogies of records of stochastic processes with stationary increments as unimodular trees}
%\title{A sample article title with some additional note\thanksref{T1}}
\runtitle{Records of stochastic processes and unimodular trees}
%\thankstext{T1}{A sample of additional note to the title.}

\begin{aug}
%%%%%%%%%%%%%%%%%%%%%%%%%%%%%%%%%%%%%%%%%%%%%%%
%% Only one address is permitted per author. %%
%% Only division, organization and e-mail is %%
%% included in the address.                  %%
%% Additional information can be included in %%
\author[A]{\fnms{François}~\snm{Baccelli}\ead[label=e1]{Francois.Baccelli@ens.fr}},
\author[A]{\fnms{Bharath}~\snm{Roy Choudhury}\ead[label=e2]{bharath.roychoudhury@gmail.com}}
%\and
%\author[B]{\fnms{???}~\snm{???}\ead[label=e3]{???@???}}
%%%%%%%%%%%%%%%%%%%%%%%%%%%%%%%%%%%%%%%%%%%%%%
%% Addresses                                %%
%%%%%%%%%%%%%%%%%%%%%%%%%%%%%%%%%%%%%%%%%%%%%%
\address[A]{INRIA/ENS - PSL, Paris \printead[presep={,\ }]{e1,e2}}

%\address[B]{???\printead[presep={,\ }]{e2,e3}}
\end{aug}

\begin{abstract}
Consider a stationary sequence $X=(X_n)$ of integer-valued random variables with mean $m \in [-\infty, \infty]$. Let $S=(S_n)$ be the stochastic process with increments $X$ and such that $S_0=0$. For each time $i$, draw an edge from $(i,S_i)$ to $(j,S_j)$, where $j>i$ is the smallest integer such that $S_j \geq S_i$, if such a $j$ exists.
This defines the record graph of $X$.

It is shown that if $X$ is ergodic, then its record graph exhibits the following phase transitions when $m$ ranges from  $-\infty$ to $\infty$.
For $m<0$, the record graph has infinitely many connected components which are all finite trees. At $m=0$, it is either a one-ended tree or a two-ended tree.
For $m>0$, it is a two-ended tree.

The distribution of the component of $0$ in the record graph is analyzed when $X$ is an i.i.d. sequence of random variables whose common distribution is supported on $\{-1,0,1,\ldots\}$, making $S$ a skip-free to the left random walk.
For this random walk, if $m<0$, then the component of $0$ is a unimodular typically re-rooted Galton-Watson Tree.
If $m=0$, then the record graph rooted at $0$ is a one-ended unimodular random tree, specifically, it is a unimodular Eternal Galton-Watson Tree.
If $m>0$, then the record graph rooted at $0$ is a unimodularised bi-variate Eternal Kesten Tree.

A unimodular random directed tree is said to be record representable if it is the component of $0$ in the record graph of some stationary sequence.
It is shown that every infinite unimodular ordered directed tree with a unique succession line is record representable.
In particular, every one-ended unimodular ordered directed tree has a unique succession line and is thus record representable.
\end{abstract}



\end{frontmatter}
%%%%%%%%%%%%%%%%%%%%%%%%%%%%%%%%%%%%%%%%%%%%%%
%% Please use \tableofcontents for articles %%
%% with 50 pages and more                   %%
%%%%%%%%%%%%%%%%%%%%%%%%%%%%%%%%%%%%%%%%%%%%%%
%\tableofcontents

%%%%%%%%%%%%%%%%%%%%%%%%%%%%%%%%%%%%%%%%%%%%%%
%%%% Main text entry area:
%-----------------------------------------------------------
%           PLAN SECTION OF THE DOCUMENT
%-----------------------------------------------------------
% \section*{Notes, todos, comments}

% \begin{tcolorbox}[colback=green!5!white, colframe = green!10!white]
%     \subsection*{Plan:}
% \begin{itemize}
%     \item Start with the stationary increments and ergodic case and show the phase transition at mean 0.
%     \item Give an example for I/F when mean =0.
%     \item Prove that for i.i.d. case the record graph at mean 0 is I/I.
%     \item Describe the record graph distribution in the three cases for i.i.d.
% \end{itemize}
%     \end{tcolorbox}

%-----------------------------------------------------------
%           INTRODUCTION SECTION
%-----------------------------------------------------------

\section{Introduction}
Consider a stationary sequence \(X=(X_n)_{n \in \mathbb{Z}}\) of integer-valued random variables with common mean \(m\) that exists and lies in \([-\infty,\infty]\).
Let \(S=(S_n)_{n \in \mathbb{Z}}\) be the stochastic process starting at \(0\) with increments \(X\) namely,
\begin{equation}\label{eq_sums_increment}
    S_0=0, \, S_n = \sum_{k=0}^{n-1}X_k,\, n >0 \text{ and } S_n = - \sum_{k=n}^{-1}X_k,\,  n<0.
\end{equation}
For each integer \(i\), the \emph{record} (epoch) \(R_X(i)\) of \(i\) is given by
\begin{equation} \label{eq_recordMap}
    R_X(i) = \begin{cases}
        \inf\{n>i:S_n \geq S_i\} \text{ if the infimum exists,}\\
        i \text{ otherwise.}
    \end{cases}
\end{equation}
Construct the \emph{record graph} \(\mathbb{Z}^R_X\) of \(X\) on the vertex set \(\mathbb{Z}\) by drawing a directed edge from each integer \(i\) to its record \(R_X(i)\) (see Fig.~\ref{fig_record_graph}).
Remove all the self loops, namely all directed edges \(i\) to \(i\), if there are any.
The record graph thus obtained is a random directed graph, with \(\mathbb{Z}\) as the vertex set and where the directed edges are random.
Let \(T\) be the connected component of \(0\) in the record graph \(\mathbb{Z}^R_X\), i.e., the subgraph induced on the subset of integers \(j\) such that there is either a directed path from \(j\) to \(0\) or from \(0\) to \(j\).
In other words, \(T\) is the connected component of \(0\) in the undirected graph of \(\mathbb{Z}^R_X\).
The focus of this paper is on the following two objects: the random directed graph \(\mathbb{Z}^R_X\) and the random rooted directed graph \([T,0]\), with \(0\) as the root.
In particular, the following three questions are addressed:
\begin{enumerate}
    \item[(Q1)] When is the record graph connected, and more generally how do the structural properties of the record graph depend on the distribution of \(X\)?
    \item[(Q2)] Are there instances of \(X\) for which the distribution of \([T,0]\) can be computed explicitly?
    \item[(Q3)] When is a random rooted graph the component of \(0\) in the record graph of a stationary sequence?
\end{enumerate}

\begin{figure}[htbp] 
  \centering % gives better spacing than \begin{center}...\end{center}
  \includegraphics[scale=0.9]{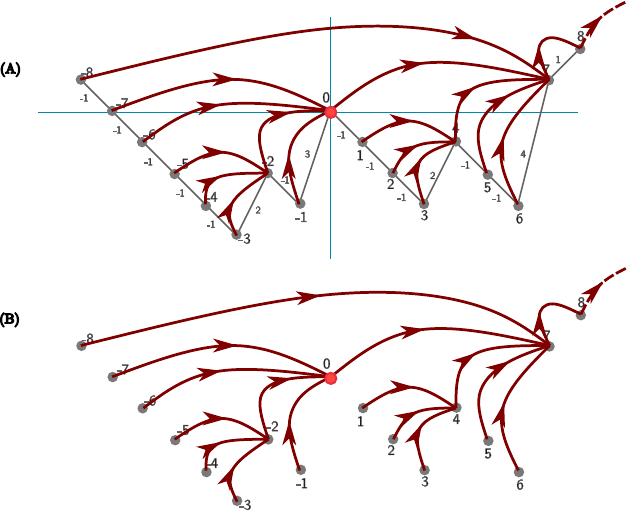}
  \caption{(A) An illustration of the trajectory \(\{(n,S_n): n \in \mathbb{Z}\}\) and the record map. (B) The component of \(0\) in the record graph.}
  \label{fig_record_graph}
  \end{figure}

%----------------------------------Short summary--------------------------
The connectivity of the record graph depends only on the mean of the increment \(X_0\).
Indeed, it is shown in the paper that the record graph is connected (in which case \(T\) is infinite) if and only if the mean is non-negative.
Since \(R_X(i) \geq i\) for all \(i \in \mathbb{Z}\) and since self loops if any were removed, the record graph does not contain any cycle.
Under the assumption that \(X\) is a stationary and ergodic, it is shown that \([T,0]\) exhibits a phase transition at \(0\) when the mean is varied in \([-\infty,\infty]\).
In the regions \([-\infty,0)\) and \((0,\infty]\), \(T\) is a finite directed tree and a two-ended directed tree, respectively.
At \(\mathbb{E}[X_0]=0\), \(T\) is either a two-ended or a one-ended directed tree.

The distribution of \([T,0]\) is explicitly computed when \((X_n)_{n \in \mathbb{Z}}\) is an i.i.d. sequence of random variables such that their common mean exists, and the support of \(X_n\) is \(\{-1,0,1,2,\cdots\}\) and non-degenerate, for all \(n \in \mathbb{Z}\).

Concerning the last question on record representation of graphs, it is shown that under a mild assumption, certain classes of ordered directed trees can be represented as the component of \(0\) in the record graph of some stationary sequence.
This includes the infinite trees associated to virtually all dynamics on ordered networks, thereby underscoring the importance of delving into the record graphs of stationary sequences.

These results are now stated more precisely after introducing some terminology from \cite{baccelliEternalFamilyTrees2018a}.

%---------------------------------Detailed summary--------------------------------
%---------------------------------------------------------------------------------
%---------------------Terminology------------------------------------------------
\subsection{Some terminology} \label{subsec_intro_terminology}
Informally, a random rooted graph is called \emph{unimodular} if for all the covariant ways of sending mass, the average total mass sent from the root to every vertex is the same as the average total mass received at the root from every vertex (see Def.~\ref{def_unimodularity}).

A \emph{vertex-shift} \(f\) is a covariant collection of maps \(\{f_G:V(G) \to V(G)\}\) that are indexed by graphs and that satisfy a measurability condition (see Def.~\ref{def_vertex_shift}).
The map \(f_G\) associated to a graph \(G\) gives a deterministic dynamic on the graph, i.e., it describes how to move from one vertex to the next vertex on the graph according to the dynamic.
One such example is the record map given in Eq.~(\ref{eq_recordMap}).
It is defined on the integer graph \((\mathbb{Z},(x_n)_{n \in \mathbb{Z}})\) in which every edge \((n,n+1)\) of the integer line has the label \(x_n\), where \((x_n)_{n \in \mathbb{Z}}\) is a realization of \((X_n)_{n \in \mathbb{Z}}\).
It describes how to move to the next record starting from every integer.
An instance of \(X=(X_n)_{n \in \mathbb{Z}}\), which is given a special emphasis in the paper, is a sequence of i.i.d. random variables satisfying the following conditions: for all \(n \in \mathbb{Z}\),
\begin{equation}\label{eq_skip_free_condition}
    X_n \in \{-1,0,1,2,\cdots\}, \ \mathbb{E}[X_n] \text{ exists, }\text{ and } 0<\mathbb{P}[X_n=-1]\leq1.
\end{equation}

The stochastic process \(S\) (as in Eq.~(\ref{eq_sums_increment})) with the increment sequence \(X\) satisfying Eq.~(\ref{eq_skip_free_condition}) is known as \emph{skip-free to the left random walk} or \emph{left-continuous random walk}. 

A \emph{Family Tree} is a (rooted) directed tree whose out-degree is at most \(1\).
A Family Tree in which every vertex has out-degree \(1\) is called an \emph{Eternal Family Tree (EFT)}.
In a Family Tree \(T\), a vertex \(u \in V(T)\) is called a child of a vertex \(v\in V(T)\) if there is a directed edge from \(u\) to \(v\).

In this case, \(v\) is called the parent of \(u\).
Similarly, for some \(n \geq 1\), a vertex \(u\) is a descendant of order \(n\) (or \(n\)-th descendant) of a vertex \(v\) if there is a directed path of length \(n\) from \(u\) to \(v\), in which case \(v\) is called the ancestor of order \(n\) (or \(n\)-th ancestor) of \(u\).

\begin{remark}\label{remark_record_graph_forest}
  Observe that the record graph of a network \((\mathbb{Z},x)\) does not contain cycles and that every vertex in the record graph has out-degree at most \(1\).
The first observation follows from the fact that if \(R_x(i)\not = i\), then \(R_x(i)>i\) and from the fact that self loops were removed.
The second observation is obvious since \(R_x\) is a function.
Thus, every component of the record graph is a (non-rooted) Family Tree.
\end{remark}

A Family Tree \(T\) is said to be \emph{ordered} if the children of every vertex \(u \in V(T)\) are totally ordered.
In this case, one obtains a total order on \(V(T)\) using the depth-first search order.

%Here is a way to construct Family Trees using vertex-shifts.
Given a vertex-shift $f$, construct a directed graph called the \emph{$f$-graph} $G^f$ associated to each graph $G$ in the following way:
the vertices of \(G^f\) are the same as those of $G$ and its directed edges are given by $\{(x,f_G(x)):x \in V(G), f_G(x)\not = x\}$.
Note that this definition of $f$-graph does not allow self-loops which differs from the definition used in  \cite{baccelliEternalFamilyTrees2018a}.

Using the vertex-shift $f$, one obtains an equivalence relation on $V(G)$ by declaring two vertices $u$ and $v$ to be equivalent if and only if there exists some positive integer $n$ such that $f^n_G(u)=f^n_G(v)$.
The equivalence class of \(u \in V(G)\) is called the \emph{foil of \(u\)} (denoted as foil\((u)\)).
Since the connected component $G^f(u)$ of every vertex $u$ in $G^f$ is the set \(\{v \in V(G): f^n_G(v)=f^m_G(u) \text{ for some } n \geq 1, m \geq 1\} \), it follows that foil\((u)\subseteq G^f(u)\).

\subsubsection{Foil classification theorem of unimodular networks} \label{par_foil_classification}
This theorem \cite[Theorem 3.9]{baccelliEternalFamilyTrees2018a} states that, given a unimodular random rooted network $[\mathbf{G}, \mathbf{o}]$ and a vertex-shift \(f\), a.s. every component of \(\mathbf{G}^f\) belongs to one of the three classes: $\mathcal{F}/\mathcal{F}$, $\mathcal{I}/\mathcal{I}$, and $\mathcal{I}/\mathcal{F}$.
The class $\mathcal{F}/\mathcal{F}$ corresponds to the component being finite and its foils being finite.

 In this case, the component is either a Family Tree or it has a unique cycle of length $n$ (for some $n>1$).
Moreover, in the former case, the component has a single foil, whereas in the latter case, it has exactly $n$ foils; where $n$ is the length of the cycle.

Components of either class \(\mathcal{I}/\mathcal{I}\) or of class \(\mathcal{I}/\mathcal{F}\) are (non-rooted) EFTs.
The class \(\mathcal{I}/\mathcal{I}\) corresponds to the component being infinite and all foils of the component being infinite.
Further, every EFT of this class is characterized by the properties that it has one end and that all of its vertices have finitely many descendants.
The class \(\mathcal{I}/\mathcal{F}\) corresponds to the component being infinite and all foils of the component being finite.
Every EFT of this class has a unique bi-infinite directed path, characterized by the set of vertices with  infinitely many descendants. The remaining vertices have finitely many descendants.

Note that, in the statement of this theorem, the relation between the number of foils and cycle length for class \(\mathcal{F}/\mathcal{F}\) is slightly modified compared to \cite{baccelliEternalFamilyTrees2018a} due to the absence of self-loops in the definition of \(f\)-graph here. 
The rest of the statement remains unchanged.

%----------------------------------Summary of phase transition------------------------
\subsection{Statement of the results} Concerning (Q1), it is shown that when \(X\) is assumed to be a stationary and ergodic sequence of random variables whose common mean exists, the record graph \(\mathbb{Z}^R_X\) is a.s. connected if and only if \(\mathbb{E}[X_0] \geq 0\).
Under the same assumptions, it is also shown that the component \([T,0]\) of \(0\) in this graph exhibits a phase transition when \(\mathbb{E}[X_0]\) is varied from \(-\infty\) to \(\infty\).
If \(\mathbb{E}[X_0] \in [-\infty,0)\), then \([T,0]\) is a finite unimodular Family Tree of class \(\mathcal{F}/\mathcal{F}\) and all the components of the record graph are finite.
If \(\mathbb{E}[X_0] \in (0,\infty]\), then \([T,0]\) is a unimodular Family Tree of class \(\mathcal{I}/\mathcal{F}\).
At \(\mathbb{E}[X_0]=0\), \([T,0]\) is a unimodular Family Tree of either class \(\mathcal{I}/\mathcal{F}\) or class \(\mathcal{I}/\mathcal{I}\).
An example of \(X\) for which \(\mathbb{E}[X_0] =0\) and \([T,0]\) is of class \(\mathcal{I}/\mathcal{I}\) is obtained when \(X\) is i.i.d. and satisfies the conditions in Eq.~(\ref{eq_skip_free_condition}).
Another example of \(X\) for which \(\mathbb{E}[X_0] =0\) but \([T,0]\) is of class \(\mathcal{I}/\mathcal{F}\) is obtained using the construction of a Loynes' sequence \cite{baccelliElementsQueueingTheory2003} (see Example \ref{example_i_f_mean_0}).

%------------------------------ Summary of i.i.d. case-------------------------------

Concerning (Q2), the distribution of \([T,0]\) is explicitly computed when \(X\) is i.i.d. and satisfies the conditions in Eq.~(\ref{eq_skip_free_condition}).
In this case, there are three phases corresponding to \(\mathbb{E}[X_0]<0\), \(\mathbb{E}[X_0]=0\) and \(\mathbb{E}[X_0]>0\).
Let \(\pi\) be the distribution of \(X_0+1\).
When \(\mathbb{E}[X_0]<0\), it is shown in Theorem~\ref{20230213191656} that the distribution of \([T,0]\) is (what is termed here as) the Typically re-rooted Galton-Watson Tree (\(TGWT(\pi)\)) with the offspring distribution \(\pi\).
It is obtained by re-rooting to a root uniformly chosen in the size-biased version of the Galton-Watson Tree (\(GW(\pi)\)), where \(\pi\) is the offspring distribution.
It is shown in Proposition~\ref{20230305162953} that the following two properties: (i) unimodularity and (ii) the independence between the offspring distribution of the root and the non-descendant part of the tree, characterize a \(TGWT\).
When \(\mathbb{E}[X_0]=0\), the distribution of \([T,0]\) is the Eternal Galton-Watson Tree (\(EGWT(\pi)\)) with the offspring distribution \(\pi\) (see Theorem~\ref{theorem:R-graph_egwt}).
This result is proved using the characterization of \(EGWT\) given in \cite{baccelliEternalFamilyTrees2018a}.
When \(\mathbb{E}[X_0]>0\), it is shown in Theorem~\ref{r_graph_positive_drift_20230203174636} that \([T,0]\) is obtained from a typical re-rooting operation of the bush of the root of (what is termed here as) the bi-variate Eternal Kesten Tree (\(EKT(\tilde{\pi},\bar{\pi})\)) with the offspring distributions \(\tilde{\pi},\bar{\pi}\).
The latter offspring distributions are related to the increment \(X_0\), the hitting probability and the Doob transform of the random walk associated to \(X\).

%------------------------------Summary of representation theorem--------------------

Concerning (Q3), it is useful to introduce the notion of succession line of an ordered Family Tree passing through a vertex of the tree using the depth-first search order, also called the Royal Line of Succession (RLS) order.
It is shown that a unimodular ordered EFT can have at most two succession lines.
If a unimodular ordered EFT \([\mathbf{T}',\mathbf{o}']\) has a unique succession line, then, using the encoding of the number of children of each vertex in the succession line, it is shown that there exists a stationary sequence \(Y\) such that the component of \(0\) in the record graph of \(Y\) is \([\mathbf{T}',\mathbf{o}']\).
In this case, \([\mathbf{T}',\mathbf{o}']\) is said to have a \emph{record representation}.
Given a vertex-shift \(f\) and a unimodular graph \([\mathbf{G},\mathbf{o}]\), the component of \(\mathbf{o}\) in the \(f\)-graph \(\mathbf{G}^f\) is a unimodular Family Tree \([\mathbf{T}'',\mathbf{o}'']\).
Assign a uniform order to the children of every vertex of \(\mathbf{T}''\) to obtain a unimodular ordered Family Tree.
If it is infinite and has a unique succession line, then the above result implies that it has a record representation.
In particular, if a vertex-shift \(f\) on a network \((\mathbb{Z},X)\) associated to a stationary sequence \(X\) satisfies the condition that the component of \(0\) is infinite and has a unique succession line, then \(f\) has a record representation on \(X\), i.e., there exists a stationary sequence \(Y\) such that the component of \(0\) in the record graph of \(Y\) has the same distribution as that of the component of \(0\) in the \(f\)-graph of \((\mathbb{Z},X)\).

%-----------------------------Literature and novelty in this work-------------------
\subsection{Literature and novelty}
The statistics of records of time series have found many applications in finance, hydrology and physics.
See for instance the references in the survey \cite{godrecheRecordStatisticsStrongly2017} for a non-exhaustive list of applications.
These studies of records are focused on statistics such as the distribution of $n$-th record, the age of a record (i.e., the time gap between two consecutive records), and the age of long-lasting records (\cite{wergenRecordsStochasticProcesses2013}, \cite{godrecheRecordStatisticsStrongly2017}). 
When the time series under consideration is an i.i.d. sequence, it is shown in \cite{stepanovCharacterizationTheoremWeak1994} that the mean of the height between two consecutive record values conditioned on the initial record value could characterize the distribution of the i.i.d. sequence.
In contrast, in the present work, the focus is on the joint local structure of the records of stochastic processes that has stationary increments, with a special emphasis on skip-free random walks. 
The structure has a natural family tree flavor with an order, suggesting to term it as `genealogy of records'.

The relation between excursions of skip-free random walks and finite Galton-Watson Trees was extensively discussed in, e.g., \cite{benniesRandomWalkApproach2000}, \cite{legallRandomTreesApplications2005b}, and \cite{jimpitmanCombinatorialStochasticProcesses2006}.
In fact, the authors of \cite{benniesRandomWalkApproach2000} encode the critical Galton-Watson Tree and show that the encoding gives a correspondence between records of an excursion of the random walk and ancestors of the root in the Galton-Watson Tree.
In the present paper, these ideas are extended to infinite trees using the succession line and the RLS order.
To the best of the authors' knowledge, this complete description of the record graph of skip-free random walks, the phase transition in the stationary and ergodic case, and the record representation of a class of unimodular Family Trees are new.

\section{Preliminaries}

\subsection{Space of networks}
%----------Introduction of spaces G_* and G_**---------------------------
The following framework for networks is adopted from \cite{aldousProcessesUnimodularRandom2007}.
A graph \(G\) is a pair \(G=(V,E)\), where \(V=V(G)\) is the vertex set of \(G\) and \(E=E(G)\) is the  edge set of \(G\).
The size of a graph is the cardinality of its vertices.
A network is a graph \(G=(V,E)\) together with a complete separable space \(\Xi\) and two maps, one from \(V\) to \(\Xi\) and the other from \(E\) to \(\Xi\).
The space \(\Xi\) is called the mark space and the images under the two functions are called marks.
A graph is locally finite if every vertex of it has a finite degree.
All the graphs considered in this paper are locally finite.
A rooted network is a pair \((G,u)\), where \(G\) is a (locally finite) connected network and \(u\) is a distinguished vertex of \(G\).
A rooted isomorphism \(\alpha:(G,u) \to (G',u')\) of two rooted networks \((G,u),\, (G',u')\) is a bijective map \(\alpha:V(G) \to V(G')\) satisfying the three conditions: (1) \((x,y) \in E(G)\) if and only if \(\left(\alpha(x),\alpha(y)\right) \in E(G')\), (2) \(\alpha(u)=u'\) and (3) marks of the vertices and the edges of \(G\) are the same as the marks of their images under \(\alpha\).
Two rooted networks are said to be equivalent if and only if there exists a rooted isomorphism between them.
The equivalence class of a rooted network \((G,u)\) is denoted by \([G,u]\) and, for brevity, the latter is also called a rooted network. 
The space of all equivalent classes of rooted networks is denoted by \(\mathcal{G}_*\).
Similarly, one defines a doubly rooted network \((G,u,v)\), doubly rooted isomorphism and an equivalence relation using the doubly rooted isomorphism.
The equivalence class of a doubly rooted network \((G,u,v)\) is denoted by \([G,u,v]\), and the set of all equivalent classes of doubly rooted networks is denoted by \(\mathcal{G}_{**}\).
For a rooted network \((G,u)\), for a doubly rooted network \((G,u,v)\), and a non-negative integer \(r\), the rooted network \((G,u)_r\) is the subgraph (rooted at \(u\)) induced by all the vertices of \(G\) that are at a graph distance of at most \(r\) from \(u\), and the doubly rooted network \((G,u,v)_r\) (rooted at \(u,v\)) is the subgraph induced by all those vertices of \(G\) that are at a graph distance of at most \(r\) from both \(u\) and \(v\).
Both the spaces \(\mathcal{G}_*,\, \mathcal{G}_{**}\) are complete separable metric spaces respectively under the local metrics \(d_*, \,d_{**}\) defined in the following: for any two rooted networks \([G,u],\, [G',u']\) and for any two doubly rooted networks \([G,u,v],[G',u',v']\), \(d_*([G,u],[G',u']) = 2^{-r_{12}}\) and \(d_{**}([G,u,v],[G',u',v'])=2^{-r'_{12}}\), where \(r_{12}\) (resp. \(r'_{12}\)) is the supremum of \(r \geq 0\) such that \((G,u)_r, \, (G',u')_r\) (resp. \((G,u,v)_r, \, (G',u',v')_r\)) are isomorphic.
Similarly, the set of all equivalent classes of rooted Family Trees (resp. doubly rooted Family Trees), denoted by \(\mathcal{T}_*\) (resp. \(\mathcal{T}_{**}\)), is a complete separable metric space with the above respective metrics.
A random rooted network \([\mathbf{G},\mathbf{o}]\) is a measurable map from a probability space to \(\mathcal{G}_*\).

\subsection{Unimodularity and vertex-shifts}
%--------------------- unimodularity, vertex-shift definitions -----------------------------------------
\begin{definition}[unimodularity]\label{def_unimodularity}
    A random network \([\mathbf{G},\mathbf{o}]\) is \emph{unimodular} if it satisfies the mass transport principle, i.e., for every measurable function \(h:\mathcal{G}_{**} \to \mathbb{R}_{\geq 0}\),
    \begin{equation}
        \mathbb{E}\left[h_+([\mathbf{G},\mathbf{o}])\right] = \mathbb{E}\left[h_-([\mathbf{G},\mathbf{o}])\right],
    \end{equation}
    where \(h_+([\mathbf{G},\mathbf{o}])= \sum_{u \in V(\mathbf{G})}h([\mathbf{G},\mathbf{o},u])\), and  \(h_-([\mathbf{G},\mathbf{o}])=\sum_{u \in V(\mathbf{G})}h([\mathbf{G},u,\mathbf{o}])\).
\end{definition}

An example of a (random) unimodular network is \([\mathbb{Z},0,X]\), where the vertex set is \(\mathbb{Z}\), edge set is \(\{(n,n+1):n \in \mathbb{Z}\}\), \(X=(X_n)_{n \in \mathbb{Z}}\) is a stationary sequence of integer-valued random variables and \(X_n\) is the mark of the edge \((n,n+1)\), for all \(n \in \mathbb{Z}\).

Dynamics on networks and some results on dynamics on unimodular networks studied in \cite{baccelliEternalFamilyTrees2018a} are now introduced. 
\begin{definition}[vertex-shift]\label{def_vertex_shift}
    A \emph{vertex-shift} \(f\) is a collection of maps \(f_G:V(G)\to V(G)\), indexed by locally-finite connected networks \(G\), satisfying the following properties:
    \begin{enumerate}
        \item  \textit{Covariance:} For any graph isomorphism $\alpha:G \rightarrow H$ of two graphs $G$ and $H$, $f$ satisfies $f_H \circ \alpha = \alpha \circ f_G$,
        \item \textit{Measurability:} The map  $[G,u,v] \mapsto \mathbf{1}_{\{f_G(u)=v\}}$ from the doubly rooted space $\mathcal{G}_{**}$ is measurable. 
      \end{enumerate}
\end{definition}

A trivial vertex-shift is the identity vertex-shift defined by \(I_G(u)=u\), for all \(u \in V(G)\) and for all \(G\).
Given a subcollection of maps that are defined on a subset of networks and that satisfy the above covariance and measurability properties, this subcollection can be extended to obtain a vertex-shift by defining it as identity on the rest of networks.
This allows one to define a vertex-shift on the support of a random rooted network.
It is easy to observe (see Subsection~\ref{subsec_record_vs}) that the record map gives a vertex-shift and hence forth it is called the \emph{record vertex-shift}, denoted by \(R\).  

\begin{example}
 Another example of a vertex-shift is the \emph{Parent vertex-shift \(F\)}.
 It is indexed by Family Trees \(T\) and defined by \(u \mapsto F_T(u)\), for all \(u \in V(T)\), where \(F_T(u)=u\) if \(u\) does not have a parent, and \(F_T(u)\) is the parent of \(u\) otherwise.
\end{example}

%\textcolor{green}{As described in Subsection~\ref{subsec_intro_terminology}, a Family Tree can be obtained from a rooted network and a vertex-shift \(f\) by taking the component of the root in the \(f\)-graph of the network.
%Under certain conditions, FB Lemma \todo[inline,color=red!20]{refer FB lemma} gives the relation between the Parent vertex-shift \(F\) on the obtained Family Tree and the vertex-shift \(f\) on the unimodular rooted network that gave the Family Tree. \textcolor{blue}{no need to talk about FB lemma in this paper since we are not concerned about R-probability.}}

%I want to \ref{label_test} here.

Consider the record vertex-shift \(R\) and a network \((\mathbb{Z},x)\), where \(x=(x_n)_{n \in \mathbb{Z}}\) is a sequence of integers.
The following convention is used in this paper.
For a vertex \(i \in \mathbb{Z}\) and any positive integer $n \geq 1$, a vertex of the form \(R^n_x(i) \not = i\), if it exists, is called the \emph{ancestor of order \(n\) of \(i\)} (also called the \(n\)-th ancestor of \(i\)).
Denote the set of \emph{$n$-th descendants} (also called the descendants of order \(n\)) of vertex $i \in \mathbb{Z}$ by $D_n(i):= \{k \in \mathbb{Z}\backslash \{i\}: R_x^n(k)=i\} = \{k<i: R_x^n(k)=i\}$ and its cardinality by $d_n(i):=\#D_n(i)$, with the convention that \(R^0_x(i)=D_0(i)=i\).
Denote the set of all descendants of $i$ by $D(i):=\{i\} \cup \bigcup_{n \geq 1}D_n(i)$ and its cardinality by $d(i)=\#D(i)$.
For any vertex \(i \in \mathbb{Z}\), its ancestor of order \(1\), if it exists, is called the \emph{parent of \(i\)}; its descendants of order \(1\) are called the \emph{children of \(i\)}; and the set \(\{k \in \mathbb{Z}\backslash \{i\}:R_x(k)=R_x(i)\}\) is called the \emph{set of siblings of \(i\)}.
The notion of ancestors and descendants is consistent with that of the Family Trees.
This is useful because the components of the record graph are Family Trees.

%For a vertex \(u\) of \(G\) and any integer $n \geq 1$, a vertex of the form \(f^n_G(u) \not = u\), if it exists is called the \emph{ancestor of order \(n\) of \(u\)} (also called the \(n\)-th ancestor of \(u\)).
% Denote the set of \emph{$n$-th descendants} (also called descendants of order \(n\)) of vertex $u \in V(G)$ by $D_n(u):= \{v \in V(G)\backslash \{u\}: f_G^n(v)=u\}$ and its cardinality by $d_n(u):=\#D_n(u)$, with the convention that \(f^0_G(u)=D_0(u)=u\).
% Denote the set of all descendants of $u$ by $D(u):=\{u\} \cup \bigcup_{n>0}D_n(u)$ and its cardinality by $d(u)=\#D(u)$.
% For any vertex \(u \in G\), its ancestor of order \(1\) is called the \emph{parent of \(u\)}, its descendants of order \(1\) are called the \emph{children of \(u\)}, and the set \(\{v \in V(G)\backslash \{u\}:f_G(v)=f_G(u)\}\) is called the \emph{set of siblings of \(u\)}.
% The notion of ancestors and descendants is consistent with that of Family Trees when the latter are seen under the Parent vertex-shift.
% Note that each vertex of an EFT has a unique ancestor of order \(n\), for all \(n \geq 0\).

A \emph{covariant} subset is a collection of subsets \(\mathfrak{S}_G\) indexed by (locally-finite) connected networks \(G\), where \(\mathfrak{S}_G\) is a subset of \(V(G)\), that satisfy the following properties; (1) covariance property: \(\alpha(\mathfrak{S}_G) = \mathfrak{S}_{\alpha(G)}\) for every network isomorphism \(\alpha\), (2) measurability: the map \([G,o] \mapsto \mathbf{1}\{o \in \mathfrak{S}_G\}\) is measurable.
A \emph{covariant partition} is a collection of partitions \(\Pi_G\) indexed by (locally-finite) connected networks \(G\), where \(\Pi_G\) is a partition of \(V(G)\) and the collection satisfies the following two conditions, (1) every network isomorphism \(\alpha\) satisfies \(\alpha \circ \Pi_{G}=\Pi_{\alpha(G)}\), (2) the subset \(\{[G,o,u]:u \in \Pi_G(o)\}\) is a measurable subset, where \(\Pi_G(o)\) is the element that contains \(o\).

Examples of covariant partitions can be constructed from any vertex-shift.
Given a vertex-shift \(f\), the collection of \(f\)-foils of \(G\) indexed by \(G\) is a covariant partition.
Similarly, the collection of subsets \(\mathfrak{S}_G = \{u \in V(G): f_G(u)=u\}\) of \(V(G)\) indexed by network \(G\) is an example of a covariant subset.

%-------------------- Vertex-shifts on unimodular networks---------------------------------------------

Vertex-shifts on unimodular networks have interesting properties.
Some of the properties which will be used in the later sections are stated below (see \cite{baccelliEternalFamilyTrees2018a} for their proofs).

\begin{theorem}[Point-stationarity] \label{thm_point_stationarity}
    Let \(f\) be  a vertex-shift and \([\mathbf{G},\mathbf{o}]\) be a unimodular measure.
    Then, the map \(\theta_f\) preserves the distribution of \([\mathbf{G},\mathbf{o}]\) if and only if \(f\) is  a.s. bijective.
\end{theorem}

\begin{proposition}\label{prop_injective_bijective}
    Let \([\mathbf{G},\mathbf{o}]\) be unimodular and \(f\) be a vertex-shift.
    \begin{itemize}
        \item If \(f\) is injective a.s., then \(f\) is bijective a.s.
        \item If \(f\) is surjective a.s., then \(f\)is bijective a.s.
    \end{itemize}
\end{proposition}

\begin{lemma}[No Infinite/Finite Inclusion]\label{lemma_no_infinite_finite}
    Let \([\mathbf{G},\mathbf{o}]\) be a unimodular network, \(\mathbf{\mathfrak{S}}\) be a covariant subset and \(\mathbf{\Pi}\) be a covariant partition.
    Almost surely, there is no infinite element \(E\) of \(\mathbf{\Pi}_{\mathbf{G}}\) such that \(E \cap \mathbf{\mathfrak{S}}_{\mathbf{G}}\) is finite and non-empty.
\end{lemma}

\begin{lemma}\label{lemma_non_empty_covariant_set}
    Let \([\mathbf{G},\mathbf{o}]\) be a unimodular network, \(\mathfrak{S}\) be a covariant subset.
    Then, \(\mathbb{P}[\mathfrak{S}_{\mathbf{G}} \text{ is non-empty}]>0\) if and only if \(\mathbb{P}[\mathbf{o} \in \mathfrak{S}_{\mathbf{G}}]>0\).
\end{lemma}

Recall that for a vertex-shift \(f\) and a network \(G\), \(G^f\) denotes the \(f\)-graph of \(G\).
For any vertex \(u \in V(G)\), denote the component of \(u\) in \(G^f\) by \(G^f(u)\).
The following lemma follows because the map \([G,o] \mapsto [G^f(o),o]\) is measurable.
\begin{lemma}\label{lemma_f_graph_unimodular}
    Let \(f\) be a vertex-shift and \([\mathbf{G}, \mathbf{o}]\) be a random rooted network.
    If \([\mathbf{G}, \mathbf{o}]\) is unimodular, then \([\mathbf{G}^f(\mathbf{o}),\mathbf{o}]\) is unimodular and conditioned on being infinite, \([\mathbf{G}^f(\mathbf{o}),\mathbf{o}]\) is a unimodular EFT.
\end{lemma}

The results that are not stated here but, that are used in the later sections are the foil classification theorem of unimodular networks (see Sec. \ref{par_foil_classification}) \cite[Theorem 3.9]{baccelliEternalFamilyTrees2018a} and the classification theorem of unimodular Family Trees \cite[Proposition 4.3, Proposition 4.4]{baccelliEternalFamilyTrees2018a}.
The latter theorem states that any unimodular (rooted) Family Tree is either of class \(\mathcal{F}/\mathcal{F}\), or of class \(\mathcal{I}/\mathcal{F}\), or of class \(\mathcal{I}/\mathcal{I}\).

\subsection{Unimodular Eternal Galton-Watson Trees} \label{subsec_EGWT}
%------------------------- EGWT------------------------------------------
 Unimodular Eternal Galton-Watson Trees (\(EGWT\)s) are instances of unimodular EFTs of class \(\mathcal{I}/\mathcal{I}\).
    They are parametrized by a non-trivial offspring distribution which has mean \(1\).
    A characterizing property of \(EGWT\)s is stated in Theorem~\ref{theorem_characterization_egwt}, see \cite{baccelliEternalFamilyTrees2018a} for more properties.
    In this work, \(EGWT\)s show up as the record graph of skip-free to the left random walks when the increments of the random walk have zero mean.

    Let $\pi$ be a probability distribution on $\mathbb{Z}_{\geq 0}$ such that mean $m(\pi)=1$ and \(\pi(1)<1\), and let $\hat{\pi}$ be its size-biased distribution defined by $\hat{\pi}(k) = k \pi(k), \forall k \geq 0$.

    The ordered unimodular Eternal Galton-Watson Tree with offspring distribution $\pi$, $EGWT(\pi)$, is a rooted ordered Eternal Family Tree $[\mathbf{T},\mathbf{o}]$ with the following properties. 
    The root $\mathbf{o}$ and all of its descendants reproduce independently with the common offspring distribution \(\pi\).
    For all $n \geq 1$, the ancestor $F^n(\mathbf{o})$ of $\mathbf{o}$ reproduces with distribution $\hat{\pi}$. 
    The descendants of $F^n(\mathbf{o})$ which are not the descendants of $F^{n-1}(\mathbf{o})$ (and not $F^{n-1}(\mathbf{o})$) reproduce independently with the common offspring distribution \(\pi\), with the convention $F^0(\mathbf{o})=\mathbf{o}$.
    All individuals (vertices) reproduce independently.
    The order among the children of every vertex is uniform.

    In particular, except for the vertices $F^n(\mathbf{o}),\, n\geq 1$, which reproduce independently with distribution $\hat{\pi}$, all the remaining vertices of the tree reproduce independently with the offspring distribution $\pi$. See Figure~\ref{fig_EGWT} for an illustration. 

    \begin{figure}[htbp] 
    \centering % gives better spacing than \begin{center}...\end{center}
    \includegraphics[scale=0.8]{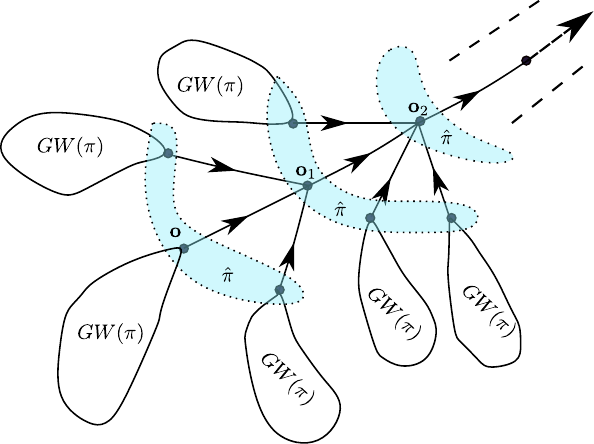}
    \caption{The Eternal Galton-Watson Tree with offspring distribution $\pi$. The Galton-Watson Tree with offspring distribution \(\pi\) is denoted by $GW(\pi)$, $\hat{\pi}$ is the size-biased distribution of $\pi$. The order among the children of every vertex is considered to be increasing from left to right.}
    \label{fig_EGWT}
    \end{figure}

    \begin{remark}\label{remark_supp_EGWT}
    Let $[\mathbf{T},\mathbf{o}]$ be a unimodular $EGWT(\pi)$, where the mean $m(\pi)=1$. Then, the descendant trees of all vertices, except for the ancestors of the root $\mathbf{o}$, are independent  critical Galton-Watson Trees with the offspring distribution $\pi$.
    Therefore, the descendant trees are finite.
    Hence, a realization of $EGWT(\pi)$ has the property that the descendant tree of every vertex is finite.
    So, \(EGWT(\pi)\) is of class \(\mathcal{I}/\mathcal{I}\).
    \end{remark}

    Similar construction exists for \(EGWT(\pi)\) when the mean \(m(\pi) \not = 1\).
    However, \(EGWT(\pi)\) is unimodular if and only if the mean \(m(\pi)\) is \(1\); and 
    the following theorem characterizes unimodular \(EGWT\)s (see \cite[Proposition 6.5, Theorem 6.6]{baccelliEternalFamilyTrees2018a}).

    \begin{theorem} \label{theorem_characterization_egwt}
        A unimodular Eternal Family Tree \([\mathbf{T},\mathbf{o}]\) is an \(EGWT\) if and only if the number of children \(d_1(\mathbf{o})\) of the root is independent of the non-descendant tree \(D^c(\mathbf{o})\), i.e., the subtree induced by \((V(\mathbf{T})\backslash D(\mathbf{o}))\cup \{\mathbf{o}\}\).
    \end{theorem}

    Note the emphasis on the Family Tree being eternal in Theorem~\ref{theorem_characterization_egwt}.
    For a similar characterization result about finite Family Trees, see Proposition~\ref{20230305162953}.

    \subsection{Typical re-rooting operation}

    The typical re-rooting operation is applied to the probability distribution of a random Family Tree whose size is finite in mean, and the operation results in another probability distribution.

    \begin{definition}[Typical re-rooting]\label{def_typical_rerooting}
        Let \([\mathbf{T},\mathbf{o}]\) be a random Family Tree that satisfies \(\mathbb{E}[\#V(\mathbf{T})]< \infty\) and \(\mathcal{P}\) be its distribution.
        The typical re-rooting of \([\mathbf{T},\mathbf{o}]\) (or its distribution \(\mathcal{P}\)) is the probability measure \(\mathcal{P}'\) defined by:
        \begin{equation}
            \mathcal{P}'[A] = \frac{1}{\mathbb{E}[\#V(\mathbf{T})]} \mathbb{E}\left[\sum_{u \in V(\mathbf{T})} \mathbf{1}_A([\mathbf{T},u])\right],
        \end{equation}
        for any measurable subset \(A\) of \(\mathcal{T}_*\).
    \end{definition}

    A random Family Tree whose distribution is \(\mathcal{P}'\) is also said to be a typically re-rooted \([\mathbf{T},\mathbf{o}]\).
    The following lemma shows that the Family Tree obtained from the typical re-rooting of any Family Tree that has finite mean size is unimodular.
    Since the unimodular Family Tree thus obtained is also finite, it is of class \(\mathcal{F}/\mathcal{F}\).
    So, this operation can be used to generate unimodular Family Trees of class \(\mathcal{F}/\mathcal{F}\).

    \begin{lemma}
        Let \([\mathbf{T},\mathbf{o}]\) be a random Family Tree that satisfies \(\mathbb{E}[\#V(\mathbf{T})]< \infty\).
        Then, the random Family Tree \([\mathbf{T}',\mathbf{o}']\) obtained by typically re-rooting \([\mathbf{T},\mathbf{o}]\) is unimodular.
    \end{lemma}

    \begin{proof}
        For any measurable function \(h:\mathcal{T}_{**} \to \mathbb{R}_{\geq 0}\),
        \begin{align*}
            \mathbb{E}\left[h_+([\mathbf{T}',\mathbf{o}'])\right] &= \frac{1}{\mathbb{E}[\#V(\mathbf{T})]}\mathbb{E}\left[\sum_{u \in V(\mathbf{T})} h_+([\mathbf{T},u])\right]\\
                &= \frac{1}{\mathbb{E}[\#V(\mathbf{T})]} \mathbb{E}\left[\sum_{u \in V(\mathbf{T})} \sum_{v \in V(\mathbf{T})}h([\mathbf{T},u,v]) \right]\\
                &= \frac{1}{\mathbb{E}[\#V(\mathbf{T})]}\mathbb{E}\left[\sum_{v \in V(\mathbf{T})} h_-([\mathbf{T},v])\right] =\mathbb{E}\left[h_-([\mathbf{T}',\mathbf{o}'])\right].
        \end{align*} 
    \end{proof}

    A construction outlined in \cite{baccelliEternalFamilyTrees2018a} offers a method for generating EFTs of class \(\mathcal{I}/\mathcal{F}\).
    In this paper, this construction is referred to as the \emph{typically re-rooted joining}, and it is defined in the following way.

    \begin{definition}[Typically re-rooted joining] \label{def_typical_reroot_joining}
        Let $([\mathbf{T}_i,\mathbf{o}_i])_{i \in \mathbb{Z}}$ be a stationary sequence of random rooted Family Trees such that \(\mathbb{E}[\#V(\mathbf{T}_i)]<\infty\) for all \(i \in \mathbb{Z}\).
        The typically re-rooted joining of $([\mathbf{T}_i,\mathbf{o}_i])_{i \in \mathbb{Z}}$ is the probability measure \(\mathcal{P}'\) given by the following: the EFT \([\mathbf{T}',\mathbf{o}']\) is first obtained by joining \((\mathbf{T}_i)_{i \in \mathbb{Z}}\) by adding directed edges \(\{(\mathbf{o}_i,\mathbf{o}_{i+1}):i \in \mathbb{Z}\}\), where \(\mathbf{o}'=\mathbf{o}_0\), and one then defines
        \begin{equation}
            \mathcal{P}'[A] = \frac{1}{\mathbb{E}[\#V(\mathbf{T}_0)]} \mathbb{E}\displaystyle\left[\sum_{v \in V(\mathbf{T}_0)} \mathbf{1}_{A}([\mathbf{T}',v]) \right],
        \end{equation}
        for every measurable subset \(A\) of \(\mathcal{T}_*\).
    \end{definition}

    The Family Tree \([\mathbf{T}_i,\mathbf{o}_i]\) is called the \(i\)-th bush of \([\mathbf{T}',\mathbf{o}']\) and \([\mathbf{T}_0,\mathbf{o}_0]\) is called the bush of the root.

    Let \([\bar{\mathbf{T}},\bar{\mathbf{o}}]\) be the typical re-rooting of \([\mathbf{T}_0,\mathbf{o}_0]\), and \(\mathbf{u}\) be the unique ancestor of \(\bar{\mathbf{o}}\) that does not have any parent.
    Let \([\mathbf{T},\bar{\mathbf{o}}]\) be the EFT obtained by joining \(\{[\mathbf{T}_i,\mathbf{o}_i]: i \in \mathbb{Z}\backslash \{0\}\} \cup \{[\bar{\mathbf{T}},\bar{\mathbf{o}}]\}\) by adding directed edges \(\{(\mathbf{o}_i,\mathbf{o}_{i+1}): i \in \mathbb{Z}\backslash \{-1,0\}\} \cup \{(\mathbf{o}_{-1},\mathbf{u}),(\mathbf{u},\mathbf{o}_1)\}\).
Observe that the distribution of \([\mathbf{T},\bar{\mathbf{o}}]\) is the same as the probability measure \(\mathcal{P}'\).

The following results show that the typically re-rooted joining of a stationary sequence is unimodular and belongs to class \(\mathcal{I}/\mathcal{F}\), and that every unimodular EFT of class \(\mathcal{I}/\mathcal{F}\) is obtained in this manner.

\begin{theorem}[\cite{baccelliEternalFamilyTrees2018a}]\label{thm:I_F_unimodularizable}
    Let $[\mathbf{T},\mathbf{o}]$ be the typically re-rooted joining of a stationary sequence $([\mathbf{T}_i,\mathbf{o}_i])_{i \in \mathbb{Z}}$, where $\mathbb{E}[\# V(\mathbf{T}_0)]< \infty$.
  Then, $[\mathbf{T},\mathbf{o}]$ is unimodular.
  \end{theorem}
  \begin{theorem}[\cite{baccelliEternalFamilyTrees2018a}]\label{thm_eft_I_f_joining}
    Let $[\mathbf{T},\mathbf{o}]$ be a unimodular EFT of class $\mathcal{I}/\mathcal{F}$ a.s., and $[\mathbf{T}',\mathbf{o}']$ be the Family Tree obtained by conditioning $[\mathbf{T},\mathbf{o}]$ on the event that $\mathbf{o}$ belongs to the bi-infinite path of $\mathbf{T}$.
    Then, $[\mathbf{T}',\mathbf{o}']$ is the joining of some stationary sequence of finite Family Trees $([\mathbf{T}_i,\mathbf{o}_i])_{i \in \mathbb{Z}}$ and $[\mathbf{T},\mathbf{o}]$ is the typically re-rooted joining of  $([\mathbf{T}_i,\mathbf{o}_i])_{i \in \mathbb{Z}}$.
  \end{theorem}

\subsection{The record vertex-shift and its properties}\label{subsec_record_vs}

Observe that the record map defined in Eq.~(\ref{eq_recordMap}) is covariant.
This property is apparent when considering any two networks \((\mathbb{Z},x)\) and \((\mathbb{Z},y)\) associated with the sequences \(x=(x_n)_{n \in \mathbb{Z}}\) and \(y=(y_n)_{n \in \mathbb{Z}}\), respectively.
The two networks are isomorphic, preserving the order on \(\mathbb{Z}\), if and only if \(y\) can be obtained by shifting \(x\), i.e., there exists an integer \(i\) such that \(T_ix = y\), where \(T_ix = (x_{n-i})_{n \in \mathbb{Z}}\).

In this case, the isomorphism \(\alpha_i:(\mathbb{Z},x)\to (\mathbb{Z},y)\) associated to the integer \(i\) is given by \(\alpha_i(n) = n+i\), for all \(n \in \mathbb{Z}\).
Therefore, the record map is covariant if and only if \(R_{T_ix}(k+i) = i+R_x(k)\), for all integers \(k,i\) and for any integer-valued sequence \(x=(x_n)_{n \in \mathbb{Z}}\).
Indeed, for any integers \(k,i\),
\begin{align*}
    R_{T_ix}(k+i) &= \begin{cases} \inf\{ j > k+i: \sum_{l=k+i}^{j-1} x_{l-i} \geq 0\} \text{ if infimum is defined},\\ 
      k+i \text{ otherwise} \end{cases}\\
    &= \begin{cases} i + \inf\{j>k: \sum_{l=k}^{j-1} x_{l}\}  \text{ if infimum is defined} \\
       i+ k \text{ otherwise}\end{cases}\\
    &= i+ R_{x}(k).
  \end{align*}

The record map also satisfies the measurability condition as the map $[\mathbb{Z},i,j,x] \mapsto 1_{R_x(i)=j}$ is a function of $(x_i,x_{i+1},\cdots,x_{j-1})$.
Therefore, the record map induces a vertex-shift which is termed as the record vertex-shift.

The record vertex-shift possesses several key properties that play a crucial role in determining the record graph.
A key property is now described.

Let $x:=(x_n)_{n \in \mathbb{Z}}$ be an arbitrary integer-valued sequence, and $(\mathbb{Z},x)$ be its associated network.
For the sequence \(x\), it is convenient to work with the partial sums \(y(k,i) = s_i-s_k= \sum_{j=k}^{i-1}x_j\), for any integers \(k<i\).
%The sequence \((y(k,i))_{k<i}\) are the partial sums in the past of \(i\). 
Define the function $L: \mathbb{R}^{\mathbb{Z}} \times \mathbb{Z}\rightarrow \mathbb{Z} \cup \{-\infty\}$ as
\begin{equation}\label{eq:L_x_defn}
  L_{x}(i) = \begin{cases}
    \inf\{j<i: y(k,i) \geq 0, \forall j \leq k < i\} \text{, if infimum exists;}\\
    i \text{ otherwise}.
    \end{cases}
\end{equation}
The integer \(L_x(i)\) is the largest integer (if it exists) up to which all the sums in the past of \(i\) are at most equal to \(s_i\).
Note that $-\infty \leq L_x(i) <i$ if and only if $-y(j,i) \leq 0$ for all integers \(j\) such that $L_x(i) \leq j<i$.
Therefore, \(L_x(i)=i\) if and only if $x_{i-1}<0$.

The following lemma shows that the set of descendants of any integer $i$ is the set of integers that lie between $L_x(i)$ and $i$ (including $L_x(i)$ and \(i\)).

Recall the notation: for a sequence \(x\) and an integer \(i \in \mathbb{Z}\), \(D_1(i):=\{j<i: R_x(j)=i\}\) is the set of children of \(i\) and \(D(i):=\{j<i: R^n_x(j)=i\) for some \(n>0\} \cup \{i\}\) is the set of descendants of \(i\).

\begin{lemma}\label{lemma:descendants}\label{lemma_descendants}
    Let $x = (x_n)_{n \in \mathbb{Z}}$ be an integer-valued sequence and $R$ be the record vertex-shift on the network \((\mathbb{Z},x)\). 
  The set of descendants of any integer $i$ is given by
    \begin{equation*}
      D(i) = \{j \in \mathbb{Z}: L_{x}(i) \leq j \leq i\}.
    \end{equation*}
  \end{lemma}
  
  \begin{proof}
  It is first proved that $ \{j \in \mathbb{Z}: L_{x}(i) \leq j \leq i\}  \subset D(i)$.
  If $L_x(i)=i$, then there is nothing to prove.
  So, assume that $L_{x}(i)<i$ (note that $L_{x}(i)$ can be $-\infty$). 
  Consider any integer $j$ such that $L_{x}(i) \leq j < i$.
  Since \(y(j,i) \geq 0\), it follows that
  $$
  j< R_{x}(j) := \inf\{k>j: y(j,k) \geq 0\} \leq i.
  $$
  \noindent By iteratively applying the same argument to \(R^l(j)\) for each \(l>0\), one finds the smallest non-negative integer \(m\) such that \(R^m_x(j)=i\).
  Such an \(m\) exists as there are only finitely many integers between \(j\) and \(i\).
  This implies that $j$ is a descendant of $i$. 
  Therefore, $\{j:L_{x}(i) \leq j \leq i\} \subseteq D(i)$.
  
  It is now proved that $D(i) =  \{j \in \mathbb{Z}: L_{x}(i) \leq j \leq i\}$, i.e., if $-\infty< j<L_x(i)$, then \(j\)  is not a descendant of \(i\).
  For any $-\infty<j < L_{x}(i)$, the \emph{claim} is that, either $R_{x}(j)< L_{x}(i)$  or $R_{x}(j)>i$.
  Assume that the claim is true.
  If \(R_x(j)<L_x(i)\), then by applying the claim again to \(R_x(j)\), one obtains that, either \(R^2_x(j)<L_x(i)\) or \(R^2_x(j)>i\).
  Iterate this process several times until one finds the largest non-negative integer \(n\) such that $R_x^n(j)<L_x(i)$.
  Such an \(n\) exists as there are only finitely many integers between \(j\) and \(L_x(i)\).
  As \(n\) is the largest integer satisfying this condition, by applying the claim to $R_x^n(j)$, one obtains that $R_x^{n+1}(j)>i$.
  This implies that none of the descendants of $R_x^n(j)$ (including $j$) is a descendant of $i$, which completes the proof.
  
  The last \emph{claim} is now proved.
  Let $k := L_{x}(i)-1$ and let \(j\) be as in the claim.
  Suppose that $R_{x}(j) \geq L_x(i)$, i.e., $R_x(j)>k$.
  Then, $y(k,i)<0$, by the definition of $L_x(i)$. 
  Since $R_x(j)>k$, one has $y(j,k) \leq 0$ (equality holds only if $j=k$).
  So, $y(j,i) = y(j,k)+y(k,i)<0$.
  Therefore, for any $L_{x}(i) \leq m\leq i$, $y(j,m) = y(j,i)-y(m,i) <0$, since $y(m,i) \geq 0$ which follows by the definition of \(L_x(i)\).
  So, $R_{x}(j)>i$ and the claim is proved.
  \end{proof}

The following lemma shows that all the integers between an integer \(i\) and its record \(R_x(i)\) are descendants of \(R_x(i)\).

\begin{lemma}[Interval property]\label{lemma:record_descendants} \label{lemma_interval_property}
    Let $x = (x_n)_{n \in \mathbb{Z}}$ be an integer-valued sequence and $R$ be the record vertex-shift on the network \((\mathbb{Z},x)\).
    For all $i \in \mathbb{Z}$, the set of descendants of $R_x(i)$ contains $\{j:i \leq j \leq R_{x}(i)\}$.
    There could be more descendants of $R_x(i)$   that are less than $i$.
  \end{lemma}
  \begin{proof}
    If \(R_x(i)=i\), then it follows from the definition of descendants that \(i = R_x(i) \in D(R_x(i))\).
    So, let us assume that \(i<R_x(i)\).
    By Lemma~\ref{lemma:descendants}, it is enough to show that $L_x(R_x(i))\leq i$.
    For any integer $m$ such that $i <m < R_x(i)$, since $y(i,m)<0$ and $y(i,R_x(i)) \geq 0$, one has
  \[
  y(m,R_x(i)) = y(i,R_x(i))- y(i,m) \geq 0.
  \]
  Thus, $L_{x}(R_{x}(i)) \leq i$.
  \end{proof}

\subsection{Royal Line of Succession (RLS) order} \label{subsec_RLS}
 Let $T$ be an ordered Family Tree, where the children of every vertex are totally ordered using $<$.
 The order on children can be extended to establish a total order on the vertices of $T$ as follows.

 Let $u,v$ be two distinct vertices of $T$.
 The vertex $v$ is said to have precedence over $u$, denoted as $v \succ u$ (or $u \prec v$), if either $F^k(u)=v$ for some \(k>0\), or there exists a common ancestor $w= F^m(u)=F^n(v)$ of $u$ and $v$ with $m,n$ being the smallest positive integers with this property, and $F^{m-1}(u)<F^{n-1}(v)$.
 In the latter case, both $F^{m-1}(u)$ and $F^{n-1}(v)$ can be compared, as they are children of $w$.
 \begin{remark} \label{remark:rls_order}
   For two vertices $u,v$ with common ancestor $w$, distinct from $u$ and $v$, $u \prec v$ if and only if $y_u \prec y_v $ for every descendant $y_u$ of $u$ and for every descendant $y_v$ of $v$.
 \end{remark}
 The resulting total order $\prec$ is termed the \emph{Royal Line of Succession (RLS) order} associated with \(<\), and it is also known as the \emph{Depth First Search order}.

The immediate successor and the immediate predecessor of any vertex in an ordered Family Tree are now defined using the RLS order.

\begin{definition}\label{defn_a_b}
  For any vertex $u$ of an ordered Family Tree $T$, let $B(u):=\{v \in V(T): v \prec u\}$ and $A(u):= \{v \in V(T): v \succ u\}$.
  Define the \emph{immediate successor} of \(u\), $b(u) = \max\{v \in B(u)\}$ if it exists, and the \emph{immediate predecessor} of \(u\), $a(u)= \min\{v \in A(v)\}$ if it exists, where $\max$ and $\min$ are taken with respect to $\prec$.
\end{definition}

The immediate predecessor of a vertex may not exist, even though the vertex is not the largest.
For instance, consider the Family Tree whose vertices are \(\{u_n:n \in \mathbb{Z}\}\) and directed edges \(\{(u_n,u_{n+1}):n \in \mathbb{Z}\}\).
Add a new vertex \(u\) to the Family Tree and a directed edge \((u,u_m)\), for some \(m \in \mathbb{Z}\), such that \(u_{m-1} \succ u\).
Then, \(A(u)=\{u_k:k \in \mathbb{Z}\}\), and therefore \(a(u)\) does not exist.

However, the following lemma shows that the immediate successor of a vertex always exists if the vertex is not the smallest.

\begin{lemma}\label{lemma:b_exists}
    Let $T$ be an ordered Family Tree, and $u \in V(T)$.
    If $u$ is not the smallest vertex, then $b(u)$ exists.
  \end{lemma}
  \begin{proof}
    By hypothesis, $B(u)$ is non-empty.
    So, either of the following two cases occurs:
    \begin{itemize}
      \item If $u$ has children, then since it could only have finitely many children, $b(u)$ is the largest (eldest) child of $u$.
      \item If $u$ does not have any child, let $w=F^m(u)$ (for some $m \geq 1$) be the smallest ancestor of $u$ that has a child smaller than $F^{m-1}(u)$, with the notation $F^0(u)=u$. Such a vertex $w$ exists because $B(u)$ is non-empty.
      Then, $b(u)$ is the largest among those siblings of $F^{m-1}(u)$ that are smaller than $F^{m-1}(u)$.
    \end{itemize}
  \end{proof}

The iterates of the immediate successor and the immediate predecessor maps \(a,b\) starting from a vertex \(u\) of an ordered Family Tree give the succession line passing through the vertex \(u\).
\begin{definition}\label{defn_succession_line}
    Let $T$ be an ordered Family Tree and \(o \in V(T)\).
    The \emph{succession line} passing through \(o\) is the sequence of vertices $U((T,o))=(u_n)_{n \in \mathbb{Z}}$ iteratively defined as $u_0=o$ and,
    \begin{align} \label{eq:succ_line}
      u_n &= \begin{cases}
        b(u_{n+1}) \text{ if it exists}\\
        u_{n+1} \text{ otherwise}
      \end{cases}
      \text{ for $n<0$},\\ \nonumber
      u_n &= \begin{cases}
        a(u_{n-1}) \text{ if it exists}\\
        u_{n-1} \text{ otherwise}
      \end{cases}
      \text{ for $n>0$}.
    \end{align}
\end{definition}

\subsection{Relation between the RLS order and the integer order on the record graph}

Let \(x=(x_n)_{n \in \mathbb{Z}}\) be an integer-valued sequence and \((\mathbb{Z},x)\) be its associated network.
The components of the record graph are ordered Family Trees, where the children of every vertex in the record graph are ordered according to the integer order.
The RLS order on each component, which depends on \(x\), gives a total order among the vertices of the component.
Since the vertices of the component are also totally ordered according to the integer order, these two orders can be compared.
The following lemma demonstrates that these two orders are, in fact, equivalent.

\begin{lemma} \label{lemma:rls_order} \label{lemma_rls_order}
    Let $x=(x_n)_{n \in \mathbb{Z}}$ be a real-valued sequence, $(\mathbb{Z},x)$ be its associated network. 
    Let $i,j$ be two integers that belong to the same connected component of the record graph. 
    Then, $i<j$ (with respect to the order on $\mathbb{Z}$) if and only if $i \prec j$.
  \end{lemma}
  \begin{proof}
    Since \(i,j\) belong to the same component of the record graph, there exists an integer \(w=R^m(i)=R^k(j)\) which is their smallest common ancestor, i.e., $m$ and $k$ are the smallest non-negative integers satisfying this property.

    ($i<j \implies i \prec j$):
    Since $R^n(i)$ is an increasing sequence for $0<n < m$, there exists a largest $n_0 \leq m$ such that $R^{n_0}(i) \leq j$.
    If \(R^{n_0}(i)=j\), then \(i \prec j\), which is the desired relation.
    So, assume that \(R^{n_0}(i)<j\) and \(R^{n_0+1}(i)>j\).
    By Lemma~\ref{lemma:record_descendants}, both $R^{n_0}(i)$ and $j$ are descendants of $R^{n_0+1}(i)$, which implies  that $R^{n_0+1}(i)=w$.
    Hence, \(R^{m_0}(j)=R^{n_0+1}(i)=w\) for some \(m_0>0\).
    Since $R^{n_0}(i)<j \leq R^{m_0-1}(j)$, one has $R^{n_0}(i)\prec R^{m_0-1}(j)$, and by Remark \ref{remark:rls_order}, one has $i \prec j$.
  
    ($i\prec j \implies i < j$):
  The case where $j$ is an ancestor of $i$ is trivial.
  So, assume that $j$ is not an ancestor of $i$.
  Then, $k>0$, and $R^{m-1}(i) <  R^{k-1}(j)$.
  Since $R^{m-1} (i)$ is not a descendant of $R^{k-1}(j)$, but $j$ is a descendant of $R^{k-1}(j)$, by Lemma~\ref{lemma:descendants}, one has $i<R^{m-1}(i)<L_x(R^{k-1}(j)) \leq j$.
  \end{proof}

\section{Phase transition of the record graph of a stationary and ergodic sequence}
The framework of stationary sequences is adopted from \cite{baccelliElementsQueueingTheory2003} and \cite{bremaudProbabilityTheoryStochastic2020}.
 Let \(X=(X_n)_{n \in \mathbb{Z}}\) be a stationary sequence of random variables whose common mean exists and \((\mathbb{Z},X)\) be its associated network.
 For all \(i \in \mathbb{Z}\), let \(\mathbb{Z}^R_X(i)\) denote the component of \(i\) in the record graph of \((\mathbb{Z},X)\).
 The foil classification theorem (see Sec. \ref{par_foil_classification}) implies that, for each \(i \in \mathbb{Z}\), a.s. \(\mathbb{Z}^R_X(i)\) is either of class \(\mathcal{F}/\mathcal{F}\) or of class \(\mathcal{I}/\mathcal{F}\) or of class \(\mathcal{I}/\mathcal{I}\).
 Recall that the record graph is a directed forest (see Remark \ref{remark_record_graph_forest}).

 The main result of this section is Theorem~\ref{thm_phase_transition_stationary} which shows that under the  assumption that \(X\) is ergodic, a.s. every component of the record graph belongs to one of the above-mentioned three classes depending on the mean of the increment \(X_0\).
 When the mean is negative, the record graph has infinitely many components and every component is of class \(\mathcal{F}/\mathcal{F}\).
 When the mean is positive, the record graph is connected, and it is of class \(\mathcal{I}/\mathcal{F}\).
 When the mean is \(0\), the record graph is connected and it is either of class \(\mathcal{I}/\mathcal{F}\) or of class \(\mathcal{I}/\mathcal{I}\).
A family of examples for the latter are the i.i.d. sequences of random variables with mean \(0\), in which case the record graphs are of class \(\mathcal{I}/\mathcal{I}\).
This result follows from Chung-Fuchs recurrence theorem for random walks, see Proposition~\ref{prop_r_graph_of_iid_II}.

Theorem~\ref{thm_phase_transition_stationary} is proved after establishing several lemmas.

The following proposition and corollary show that the event that the component of any fixed integer in the record graph belonging to one of the classes \(\{\mathcal{F}/\mathcal{F},\mathcal{I}/\mathcal{I},\mathcal{I}/\mathcal{F}\}\) has trivial probability.

\begin{proposition}\label{prop_ergodic_trivial}
    Let \(X=(X_n)_{n \in \mathbb{Z}}\) be a stationary and ergodic sequence of random variables. 
    Then, 
    \begin{align*}
        \mathbb{P}[\mathbb{Z}^R_X(0) \text{ is of class } \mathcal{I}/\mathcal{I}] \in \{0,1\}, \\
        \mathbb{P}[\mathbb{Z}^R_X(0) \text{ is of class } \mathcal{I}/\mathcal{F}] \in \{0,1\}.
    \end{align*}
\end{proposition}
\begin{corollary}\label{cor_comp_f_f}
    Under the assumptions of Proposition~\ref{prop_ergodic_trivial}, the event \(\{\mathbb{Z}^R_X(0)\)  is of class \(\mathcal{F}/\mathcal{F}\}\) has trivial probability measure.
\end{corollary}
\begin{proof}
    By the foil classification theorem (Sec. \ref{par_foil_classification}), \(\mathbb{Z}^R_X(0)\) is either of class \(\mathcal{I}/\mathcal{I}\) or \(\mathcal{I}/\mathcal{F}\) or \(\mathcal{F}/\mathcal{F}\).
    By Proposition~\ref{prop_ergodic_trivial}, the first two events have trivial probability measure which implies that the third event has trivial probability measure.
\end{proof}

Let \(T\) be the shift map that maps every sequence \(x=(x_n)_{n \in \mathbb{Z}}\) to \(Tx =(x_{n+1})_{n\in \mathbb{Z}}\).
The shift map naturally maps any rooted network of the form \([\mathbb{Z},0,x]\) to \(T([\mathbb{Z},0,x]) = [\mathbb{Z},0,Tx]\).
%--------------- proof of probability of component of 0 being I/I, I/F, F/F is trivial
\begin{proof}[Proof of Proposition~\ref{prop_ergodic_trivial}]
    For any integer \(i \in \mathbb{Z}\), consider the following events:
\begin{align*}
    E_i&:= \{[\mathbb{Z},0,x]:\mathbb{Z}^R_x(i) \text{ is of class } \mathcal{I}/\mathcal{I}\}\\
    F_i &:= \{[\mathbb{Z},0,x]:\mathbb{Z}^R_x(i) \text{ is of class } \mathcal{I}/\mathcal{F}\},
\end{align*}
and let \(T^{-1}E_{i} := \{[\mathbb{Z},0,x]:[\mathbb{Z},0,Tx]\in E_i\}\), \(T^{-1}F_{i} := \{[\mathbb{Z},0,x]:[\mathbb{Z},0,Tx]\in F_i\}\).
Then, for any \(i\in \mathbb{Z}\), one has
\begin{align}
    T^{-1}E_i  &= \{[\mathbb{Z},0,x]:\mathbb{Z}_{Tx}^R(i) \text{ is of class } \mathcal{I}/\mathcal{I}\} \nonumber\\
                      &=\{[\mathbb{Z},0,x]:\mathbb{Z}_x^R(i+1) \text{ is of class } \mathcal{I}/\mathcal{I}\} \label{eq_shift_I_I}\\
                      &=E_{i+1},\nonumber
\end{align}
where in Eq.(\ref{eq_shift_I_I}), the fact that \([\mathbb{Z},i,Tx] = [\mathbb{Z},i+1,x]\) was used.
Similarly, it can be shown that \(T^{-1}F_i = F_{i+1}\).

Observe  that if \([\mathbb{Z},0,x] \in E_{i}\) (similarly \(F_{i}\)), then, by the interval property (Lemma~\ref{lemma:record_descendants}), \(i+1\) and \(i\) are in the same connected component, i.e., \(\mathbb{Z}^R_x(i)=\mathbb{Z}^R_x(i+1)\), which implies that  \([\mathbb{Z},0,x] \in E_{i+1}\) (similarly \(F_{i+1}\)).
Thus, \(E_i \subseteq T^{-1}(E_i)\).
Similarly, \(F_i \subseteq T^{-1}(F_i)\).
Applying the following fact to \(E_0\) and \(F_0\) completes the proof: if \((\mathbb{P},T)\) is ergodic and \(A\) is a measurable subset such that \(A \subseteq T^{-1}A\), then \(\mathbb{P}[A] \in \{0,1\}\) (see \cite[Theorem 16.1.9]{bremaudProbabilityTheoryStochastic2020}).
\end{proof}

The following lemma is an application of Poincare's recurrence lemma.

\begin{lemma}\label{lemma_comp_0_fin_all_fin}
    Let \(X=(X_n)_{n \in \mathbb{Z}}\) be a stationary sequence of random variables, \((\mathbb{Z},X)\) be its associated network, and for all \(n \in \mathbb{Z}\), \(\mathbb{Z}^R_X(n)\) be the component of \(n\) in the record graph of \((\mathbb{Z},X)\).
    Then,
    \begin{align}
        \mathbb{P}[\mathbb{Z}^R_X(0) \text{ is finite}] &= \mathbb{P}[\mathbb{Z}^R_X(n) \text{ is finite, } \forall n \geq 1]
        =\mathbb{P}[\mathbb{Z}_X^R(n) \text{ is finite, } \forall n \leq -1] \label{eq_F_F_probability}.
    \end{align}
\end{lemma}
\begin{proof}
    Consider the event \(A = \{\mathbb{Z}^R_X(0) \text{ is finite}\}\).
    By Poincar\'e's recurrence lemma, \(\mathbb{P}[A] = \mathbb{P}\left[\{\mathbb{Z}_{T^nX}^R(0) \text{ is finite for infinitely many }n \geq 1\}\right]\).
    For any \(\omega \in A\), if \(\mathbb{Z}_{T^nX(\omega)}^R(0)\) is finite for infinitely many \(n\geq 1\), then \(\mathbb{Z}_{X(\omega)}^R(n)\) is finite for all \(n \geq 1\).
    This fact, proved below, together with Poincare's recurrence lemma give the first equality in Eq.~(\ref{eq_F_F_probability}).

    The second equality of Eq.~(\ref{eq_F_F_probability}) follows by taking the shift  \(\widetilde{T} := T^{-1}\) instead of \(T\), and by using the equality of the events \(\{\mathbb{Z}_{\widetilde{T}^nX}^R(0) \text{ is finite}\} = \{\mathbb{Z}_X^R(-n) \text{ is finite}\}\) for all \(n \geq 1\).

    The aforementioned fact will now be proved.

    Since \(X\) is stationary, the network \([\mathbb{Z},0,X]\) is unimodular.
    So, apply the foil classification theorem (Sec. \ref{par_foil_classification}) to the record graph \(\mathbb{Z}_X^R\) to obtain the following result:
    for any integer \(n\), its component \(\mathbb{Z}_X^R(n)\) is finite if and only if there exists an integer \(i \in V(\mathbb{Z}_X^R(n))\) that has only finitely many ancestors.
    In particular, if a vertex has finitely many ancestors then it must have finitely many descendants.

    The \emph{claim} is that  if \(\mathbb{Z}_{X(\omega)}^R(n_1)\) is finite for any positive integer \(n_1\), then \(\mathbb{Z}_{X(\omega)}^R(k)\) is finite for all \(k<n_1\).
    This follows because \(\mathbb{Z}_{X(\omega)}^R(n_1)\) is finite if and only if the sequence \((S_{n+n_1})_{n \geq 0}\) attains a maximum value \(a_{n_1}\) finitely many times, i.e., \(S_{n'+n_1}=a_{n_1}\) for some \(n' \geq 0\) and \(S_{n+n_1}<a_{n_1}\) for all \(n>n'\) (equivalently, from the discussion in the above paragraph, \(n_1\) has finitely many ancestors).
    So, if \(k \leq n'+n_1\), then the sequence \((S_{n+k})_{n \geq 0}\) attains the maximum value \(a_{n_1}\) at \(N=n'+n_1-k\) and never attains it for all \(n>N\).
    Therefore, \(k\) has finitely many ancestors.
    By the discussion in the above paragraph, \(\mathbb{Z}_{X(\omega)}^R(k)\) is finite, and the claim is proved.

    The above claim implies that if there exists a subsequence \((n_k)_{k \geq 1}\) of non-negative integers such that \(n_k \to \infty\) as \(k \to \infty\) and \(\mathbb{Z}^R_{X(\omega)}(n_k)\) is finite for all \(k \geq 1\), then \(\mathbb{Z}^R_{X(\omega)}(n)\) is finite for all \(n \geq 1\).
    Using this and the fact that \(\mathbb{Z}^R_{T^nX}(0)= \mathbb{Z}^R_{X}(n)\) for all \(n \geq 0\), it follows that the events \(\{\mathbb{Z}_{T^nX}^R(0)\) is finite for infinitely many \(n \geq 1\}\) and \(\{\mathbb{Z}_{X}^R(n) \text{ is finite for all }n \geq 1\}\) are one and the same.
\end{proof}

The following lemma gives a sufficient condition for the connectedness of the record graph.
The proof is an application of the interval property of the record vertex-shift.
%----- Lemma that states infinite ancestors implies connectedness of record graph
\begin{lemma}\label{lemma_stationary_one_component}
    Let \(X\) be a stationary sequence of random variables, and \((\mathbb{Z},X)\) be its associated network.
    If \(0\) has a.s. infinitely many ancestors in the record graph \(\mathbb{Z}^R_X\), then \(\mathbb{Z}^R_X\) is a.s. connected.
\end{lemma}
\begin{proof}
    For any \(i \in \mathbb{Z}\), let \(A_i\) be the event that \(i\) has infinitely many ancestors in \(\mathbb{Z}^R_X\).
    By the stationarity of \(X\), every event \(A_i\) occurs with probability \(1\) for each \(i\in \mathbb{Z}\) since \(A_0\) occurs with probability \(1\).
    Therefore, the event \(A = \cap_{i \in \mathbb{Z}}A_i\) occurs with probability \(1\).

    The next step consists in showing that every integer \(i<0\) belongs to the component of \(0\) in  \(\mathbb{Z}^R_X\).
    Similarly, by interchanging the role of \(i\) and \(0\), it can be shown that every integer \(i>0\) belongs to the component of \(0\).

    Let $i<0$.
    Since \(A\) occurs with probability \(1\), the sequence $(R^n_X(i))_{n \in \mathbb{Z}}$ is strictly increasing a.s..
    So, there exists a smallest (random) integer $k>0$ such that $R_X^{k-1}(i)<0 \leq R^k_X(i)$ a.s..
    By Lemma~\ref{lemma:record_descendants} (applied to $R_X^{k-1}(i)$), $0$ is a descendant of $R^k_X(i)$ a.s..
    Thus, a.s., $0$ and $i$ are in the same connected component.
\end{proof}

\begin{corollary}
    Let \(X=(X_n)_{n \in \mathbb{Z}}\) be a stationary and ergodic sequence of random variables, and \((\mathbb{Z},X)\) be its associated network.
    Then, the following dichotomy holds: either
    \begin{enumerate}
        \item the record graph is  a.s. connected, or
        \item a.s., every component of the record graph is finite.
    \end{enumerate}
\end{corollary}
\begin{proof}
    By Proposition~\ref{prop_ergodic_trivial}, the event that \(0\) has infinitely many ancestors has trivial probability.
    If a.s., \(0\) has finitely many ancestors, then by Lemma~\ref{lemma_comp_0_fin_all_fin}, every component of the record graph is finite.
    Otherwise, by Lemma~\ref{lemma_stationary_one_component}, the record graph is a.s. connected.
\end{proof}

The main result is:

\begin{theorem}[Phase transition of the record graph]\label{thm_phase_transition_stationary}
    Let \(X=(X_n)_{n \in \mathbb{Z}}\) be a stationary and ergodic sequence of random variables such that their common mean exists.
    Let \([\mathbb{Z}^R_X(0),0]\) be the component of \(0\) in the record graph \(\mathbb{Z}_X^R\) of the network \((\mathbb{Z},X)\).
    \begin{enumerate}
        \item If \(\mathbb{E}[X_0]<0\), then a.s. every component of \(\mathbb{Z}^R_X\) is of class \(\mathcal{F}/\mathcal{F}\).
        \item If \(\mathbb{E}[X_0]>0\), then \(\mathbb{Z}^R_X\) is connected, and it is of class \(\mathcal{I}/\mathcal{F}\) a.s.
         \item If \(\mathbb{E}[X_0]=0\), then a.s., \(\mathbb{Z}^R_X\) is connected, and it is either of class \(\mathcal{I}/\mathcal{F}\) or of class \(\mathcal{I}/\mathcal{I}\).
    \end{enumerate}
\end{theorem}
\begin{proof}
    {\it (Proof of 1.)}
    Let \(\mathbb{E}[X_0]<0\).
    Since, by ergodicity, \(\frac{S_n}{n}\to \mathbb{E}[X_0]\) a.s., it follows that \(S_n \to -\infty\) a.s.
    Therefore, the sequence \((S_n)_{n \geq 0}\) attains a maximum value \(a_0\) at some (random) \(n'\geq 0\) and \(S_n<a_0\) for all \(n>n'\).
    This implies that \(\mathbb{Z}^R_X(0)\) is finite a.s. and hence it is of class \(\mathcal{F}/\mathcal{F}\).
    By Lemma~\ref{lemma_comp_0_fin_all_fin}, every component of \(\mathbb{Z}^R_X\) is finite and hence it is of class \(\mathcal{F}/\mathcal{F}\).

    {\it (Proof of 2.)}
    Let \(\mathbb{E}[X_0]>0\).
    Since, by ergodicity, \(\frac{S_n}{n}\to \mathbb{E}[X_0]\) a.s., it follows that \(S_n \to \infty\) a.s. and \(S_{-n} \to -\infty\) a.s.
    The former implies that \(0\) has infinitely many ancestors.
    The latter implies that there exists \(n_0 \leq 0\) such that \((S_n)_{n \leq 0}\) attains maxima at \(n_0\), i.e., \(S_n \leq S_{n_0}\), for all \(n \leq n_0\).
    Therefore, \(L_X(n_0) = -\infty\), and by Lemma~\ref{lemma:descendants}, it follows that \(n_0\) has infinitely many descendants.
    %Therefore some ancestor of \(0\) (or \(0\)) has infinitely many descendants, since the time \(n_0 \leq 0\) at which \((S_n)_{n \leq 0}\) attains maxima has infinitely many descendants (\(L_X(n_0) = - \infty\)).
   % Therefore, \(0\) has infinitely many ancestors and by Lemma~\ref{lemma:descendants}, some ancestor of \(0\) (or \(0\)) has infinitely many descendants.
   By Lemma~\ref{lemma_stationary_one_component}, \(\mathbb{Z}^R_X\) has a single component and thus, it is of class \(\mathcal{I}/\mathcal{F}\).

    {\it (Proof of 3.)}
    Let \(\mathbb{E}[X_0]=0\) and \(A\) be the event defined as
    \begin{equation*}
        A:=\{\mathbb{Z}^R_X(0) \text{ is of class }\mathcal{F}/\mathcal{F}\}.
        % \mathbb{P}[[\mathbb{Z}^R_X(0),0] \text{ is of class }\mathcal{F}/\mathcal{F}] = 0.
    \end{equation*}
    In order to prove that \(\mathbb{P}[A]=0\), \emph{suppose} that \(\mathbb{P}[A]>0\).
    Then, by Corollary~\ref{cor_comp_f_f}, \(\mathbb{P}[A]=1\).
     Lemma~\ref{lemma_comp_0_fin_all_fin} implies that \( \mathbb{P}[\mathbb{Z}^R_X(n) \text{ is finite } \forall n \in \mathbb{Z}]=1\).
     Note that \(\mathbb{Z}^R_{X(\omega)}(n)\) is finite for all \(n  \in \mathbb{Z}\) if and only if there exists a (random) subsequence \((n'_k)_{k \in \mathbb{Z}}\) of \(\mathbb{Z}\) such that \(n'_k \to +\infty\), \(n'_{-k} \to -\infty\) as \(k \to \infty\) and for every \(k \in \mathbb{Z}\), \(S_n(\omega) < S_{n'_k}(\omega)\) for all \(n>n'_{k}\).
    In particular, a.s., \(S_n \to - \infty\) and \(S_{-n} \to \infty\) as \(n \to \infty\).
    Consider the covariant subset \(Q\) consisting of peak points 
    \begin{equation*}
        Q := \left\{k \in \mathbb{Z}: \sum_{i=k}^{k+n} X_i \leq -1 \, \forall n\geq 0\right\}.
    \end{equation*}
    The covariant set \(Q\) is a subset of integers such that the sums starting from these integers are always negative.
    The integers of the set \(Q\) can be enumerated as \(Q = (n_{i})_{i \in \mathbb{Z}}\), where \(n_i < n_{i+1}\) for all \(i \in \mathbb{Z}\) and \(n_0 > 0\) is the smallest positive integer that belongs to the set \(Q\).
    Such an enumeration is possible because \(S_n \to -\infty\) and \(S_{-n} \to \infty\); thus, \(n_i \to \infty\) as \(i \to \infty\).
    Note that the intermediate sums between two consecutive integers of \(Q\) satisfy
    \begin{equation}\label{eq_inter_sums}
        \sum_{k=n_i}^{n_{i+1}-1}X_k \leq -1, \, \forall i \in \mathbb{Z}.
    \end{equation}
    Since \(\mathbb{P}[Q \text{ is non-empty}]= \mathbb{P}[A]= 1\), the intensity \(\mathbb{P}[0 \in Q]\) of the set \(Q\) is positive (by Lemma~\ref{lemma_non_empty_covariant_set}).
    Birkhoff's pointwise ergodic theorem implies that a.s., \(\frac{S_n}{n} \to \mathbb{E}[X_0]\) as \(n \to \infty\).
    Since \(n_i \to \infty\) as \(i \to \infty\), it follows that a.s.,
    \begin{equation*}
        \lim_{i \to \infty} \frac{S_{n_i}}{n_i} = \lim_{n \to \infty} \frac{S_n}{n} = \mathbb{E}[X_0].
    \end{equation*}

    For all \(i\geq 1\),
    \begin{align}
        \frac{S_{n_i}}{n_i} &= \frac{X_0+\cdots+X_{n_0-1}+\sum_{k=0}^{i-1}(X_{n_k}+\cdots+X_{n_{k+1}-1})}{n_i} \nonumber \\
        &\leq \frac{X_0+\cdots+X_{n_0-1}+\sum_{k=0}^{i-1}(-1)}{n_i}\label{eq_upperbound}\\
        & = \frac{X_0+\cdots+X_{n_0-1}}{n_i}+ \left(\frac{-i}{n_i} \right) \nonumber,
    \end{align}
    where Eq.~(\ref{eq_upperbound}) is obtained using Eq.~(\ref{eq_inter_sums}).
    Taking the limit \(i \to \infty\) on both sides of Eq.~(\ref{eq_upperbound}), one obtains that a.s.,
    \begin{equation*}
        \mathbb{E}[X_0] \leq -\mathbb{P}[0 \in Q],
    \end{equation*}
    since \(\frac{-i}{n_i} = -\left(\frac{\left(\sum_{k=1}^{n_i}\mathbf{1}\{k \in Q\}\right)-1}{n_i}\right)\) and the latter converges a.s. to \(-\mathbb{P}[0 \in Q]\) as \(i \to \infty\) by the cross-ergodic theorem for point processes (see \cite[Section 1.6.4]{baccelliElementsQueueingTheory2003}).
    Since, the intensity \(\mathbb{P}[0 \in Q]>0\) and \(\mathbb{E}[X_0]=0\), one gets a contradiction.
    Hence,  \(\mathbb{P}[[\mathbb{Z}^R(0),0]\text{ is of class }\mathcal{F}/\mathcal{F}]=0\).
\end{proof}

Instances of \(X=(X_n)_{n \in \mathbb{Z}}\) for which \(\mathbb{E}[X_0]=0\), but the record graph is of class \(\mathcal{I}/\mathcal{I}\) occur when \(X=(X_n)_{n \in \mathbb{Z}}\) is an i.i.d. sequence of random variables.
This result is proved in the following proposition.

\begin{proposition}\label{prop_r_graph_of_iid_II}
    Let \(X=(X_n)_{n \in \mathbb{Z}}\) be an i.i.d. sequence of random variables such that their common mean exists and \((\mathbb{Z},X)\) be its associated network.
    If \(\mathbb{E}[X_0]=0\), then the record graph of \((\mathbb{Z},X)\) is of class $\mathcal{I}/\mathcal{I}$. 
\end{proposition}
\begin{proof}
    By Theorem~\ref{thm_phase_transition_stationary}, the record graph is either of class \(\mathcal{I}/\mathcal{I}\) or of class \(\mathcal{I}/\mathcal{F}\).
    So, it is sufficient to show that every vertex of the record graph has finitely many descendants, which implies that the record graph is of class \(\mathcal{I}/\mathcal{I}\).

    Let \(i\) be an integer.
    Consider the i.i.d. sequence \((X_{i-1-n})_{n \geq 0}\) of random variables.
    By the Chung-Fuchs theorem \cite{kallenbergFoundationsModernProbability2021}, the random walk \(\left(y(i-n,i)\right)_{n \geq 0}\) associated to the increment sequence \((X_{i-1-n})_{n \geq 0}\) is recurrent, where \(y(i-n,i) = \sum_{k=i-n}^{i-1}X_k\), for all \(n \geq 0\).
    Therefore, a.s., \(y(i-n_0,i)<0\) for some \(n_0 \geq 1\).
    This implies that \(L_X(i)>-\infty\).
    By Lemma~\ref{lemma:descendants}, the number of descendants of \(i\) in the record graph of \((\mathbb{Z},X)\) is finite.
    Since this is true for any integer \(i\), the event that every integer (vertex) of the record graph has finitely many descendants occurs almost surely.
\end{proof}

The following construction gives an example of a stationary and ergodic sequence of random variables whose mean is \(0\) but the record graph associated to the sequence is of class \(\mathcal{I}/\mathcal{F}\).

\begin{example}[stationary, ergodic, mean \(0\) but the record graph is of class \(\mathcal{I}/\mathcal{F}\)]\label{example_i_f_mean_0}\normalfont
    See Figure~\ref{fig_example_i_f_mean_0} for an illustration of this example.
    Consider an \(M/M/1/\infty\) queue with arrival rate \(\lambda\) and service rate \(\mu\) satisfying \(\lambda<\mu\).
    Let \((N_n)_{n \geq 1}\) be the number of customers in the queue observed at all changes of state.
    The Markov chain \((N_n)_{n \geq 1}\) has a unique stationary distribution \(\eta\) (see \cite{asmussenAppliedProbabilityQueues2003}).
    The Markov chain \((N_n)_{n \geq 1}\) starting with the stationary distribution \(\eta\) is ergodic and \(N_n = 0\) for infinitely many \(n \geq 1\).
    Note that \(N_n \geq 0\) and \(N_{n+1}-N_{n}\in \{-1,+1\}\), for all \(n \geq 1\).
    Consider the stationary version \((-N_n)_{n \in \mathbb{Z}}\) of \((-N_n)_{n\geq 1}\) (which can be done by extending the probability space, see \cite[Theorem 5.1.14]{bremaudProbabilityTheoryStochastic2020}).
    Let \(Z_n = -N_n\) and \(X_n = Z_{n+1}-Z_n\), for all \(n \in \mathbb{Z}\).
    The sequence \(X=(X_n)_{ \in \mathbb{Z}}\) is a stationary sequence of random variables taking values in \(\{-1,+1\}\) and with mean \(\mathbb{E}[X_0] = \mathbb{E}[Z_1-Z_0]=0\) (since the sequence \((Z_n)_{n \in \mathbb{Z}}\) is stationary).
    Consider the record vertex-shift on the network \((\mathbb{Z},X)\).
    Let \((S_n)_{n \in \mathbb{Z}}\) be the sums of \(X\) starting at \(0\) as in Eq.~(\ref{eq_sums_increment}).
    The relation between \(S_n\) and \(N_n\) is given by \(S_n=-N_n+N_0\), for all \(n \in \mathbb{Z}\).
    Since \(N_n\geq 0\) for all \(n\in \mathbb{Z}\) and \(N_n=0\) a.s. for infinitely many \(n \geq 1\), \(S_n \leq N_0\) for all \(n \in \mathbb{Z}\) and \(S_n = N_0\) for infinitely many \(n \geq 1\).
    Therefore, \(0\) has infinitely many ancestors (since some \(k\)-th record epoch of \(0\) satisfies \(S_{R^k(0)}=N_0\) and \(S_{R^{(k+i)}(0)}=N_0\), for all \(i \geq 1\)).
    Similarly, since \(S_n \leq N_0\) for all \(n< 0\), some ancestor of \(0\) has infinitely many descendants (by the interval property of the record vertex-shift).
    Thus, the component \([\mathbb{Z}_X^R(0),0]\) of the record graph of \((\mathbb{Z},X)\) is of class \(\mathcal{I}/\mathcal{F}\).
\end{example}

\begin{figure}[htbp]
    \centering % gives better spacing than \begin{center}...\end{center}
    \includegraphics[scale=0.90]{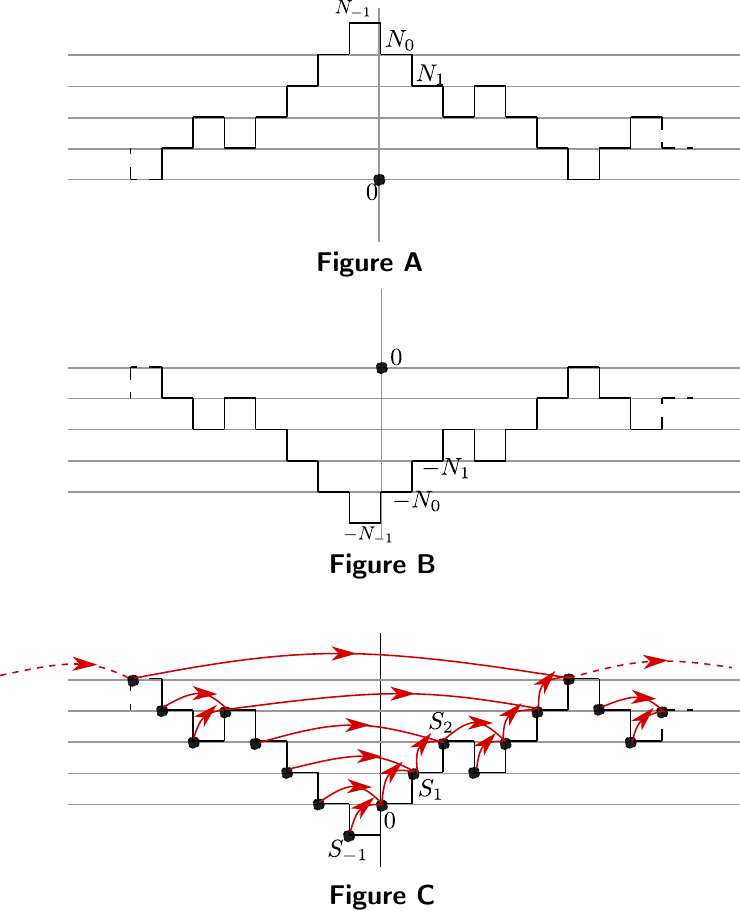}
    \caption{An Illustration of Example \ref{example_i_f_mean_0} for constructing a stationary sequence \(X=(X_n)_{n \in \mathbb{Z}}\) with \(\mathbb{E}[X_0]=0\) such that the component of \(0\) of the record graph of the network \((\mathbb{Z},X)\) is of class \(\mathcal{I}/\mathcal{F}\).
    Figure A depicts the stationary sequence \((N_n)_{n \in \mathbb{Z}}\), Figure B gives the stationary sequence \((-N_n)_{n \in \mathbb{Z}}\), and Figure C is obtained from Figure B by shifting the x-axis to \(-N_0\). In Figure C, the increment sequence of \((S_n)_{n \in \mathbb{Z}}\) is \(X\), the record graph is drawn in red and the sequence of top most arrows is the bi-infinite path in the record graph.}
    \label{fig_example_i_f_mean_0}
  \end{figure}

%-----comments-------
%  \begin{tcolorbox}[colback=orange!5!white, colframe = orange!10!white]
%     Write Lemma 116, 117 without many modifications.
%     After writing Lemma 117, write a corollary that X stationary and ergodic then by Lemma 116 and Lemma 117, we have the following dichotomy: Either (1) a.s. the record graph is connected or (2) a.s. every component of the record graph is finite.

%     The proof of phase transition comes here.

%     Say that an example for mean 0 but record graph is of I/I is when the increments satisfy skip-free condition, refer to the equation.
%     Give the second part of Prop. 40 (page 46) of the Thesis. To use Chung-Fuchs, we need finite expectation for the increments.

%     Now, give the example for mean 0 but of I/F.

%  \end{tcolorbox}
% \begin{tcolorbox}[colback=blue!5!white, colframe = blue!10!white]
%  Phase transition for stationary case 

%  Example for mean 0 I/F

%  Proof of mean 0 I/I random walk case
% \end{tcolorbox}

\section{Record vertex-shift on skip-free to the left random walks}
Let \(X=(X_n)_{n \in \mathbb{Z}}\) be an i.i.d. sequence of random variables satisfying the conditions of Eq.~(\ref{eq_skip_free_condition}) and \(S=(S_n)_{n \in \mathbb{Z}}\) be the skip-free to the left random walk starting at \(0\) with the increment sequence \(X\), as in Eq.~(\ref{eq_sums_increment}).

The main objective of this section is to describe the distribution of the component of \(0\) in the record graph of the network \((\mathbb{Z},X)\) for the three phases associated to \(\mathbb{E}[X_0]<0\), \(\mathbb{E}[X_0]=0\) and \(\mathbb{E}[X_0]>0\).
This is described in Subsection \ref{subsec_recordGraph_randomWalk}.

This section is organized as follows.
The first subsection describes the relation between the number of offsprings of a vertex in the record graph and the increment associated to it, the second subsection focuses on the properties of skip-free to the left random walks, which will be used to compute the distributions of the record graphs.
Two new families of unimodular random rooted networks are introduced in the third subsection.
Finally, the fourth subsection focuses on the aforementioned main objective.

\subsection{Relation between offspring count in the record graph and the respective increment of the sequence} \label{subsec_relation}
For a sequence \(x=(x_k)_{k \in \mathbb{Z}}\), where \(x_k \in \{-1,0,1,2,\ldots\}\) for all \(k \in \mathbb{Z}\), the number of children of every vertex \(i \in \mathbb{Z}\) in the record graph of \((\mathbb{Z},x)\) depends on the values of the three functions \(L_x\), \(l_x\) and the type function \(t_x\) at \(i\).

This subsection elaborates on these relationships, which are instrumental in characterizing the distribution of the component of \(0\) in the record graph.

Recall the notation, $y(j,k):=\sum_{l=j}^{k-1}x_l$ for all integers $j<k$ and for the sequence $x=(x_n)_{n \in \mathbb{Z}}$. 
The function \(L_x\) is defined in Eq. (\ref{eq:L_x_defn}).
\begin{definition}[Type function]
  The {\em type function} associated to \(x\) is the map $t_x:\mathbb{Z} \to \mathbb{Z}$ that assigns to each integer $i \in \mathbb{Z}$, the integer $t_x(i)\in \{-1,0,1,\ldots\}$ defined by 
\begin{equation}
  t_x(i)= \inf\{n \geq -1: y(m,i) = n \text{ for some } m<i\}.
\end{equation}
\end{definition}
The integer \(t_x(i)\) is called the {\em type} of \(i\).
The type \(t_x(i)\) of an integer \(i\) can also be written as
\begin{equation} \label{eq_type_inf}
  t_x(i)  = \inf\{y(m,i)\vee -1: m<i\},
\end{equation}
where \(\vee\) is the max function.
The type function is analogous to the Loynes' construction of the work load process \cite{baccelliElementsQueueingTheory2003}.

The function $l_x:\mathbb{Z}\to \mathbb{Z}$ is the map that assigns every integer $i$ to
\begin{equation}\label{eqn:def_l_x}
    l_x(i)= \sup\{m<i:y(m,i)=t_x(i)\}.
\end{equation}
The relation between the functions \(L_x\) and \(l_x\) is as follows:
first note that for every \(i \in \mathbb{Z}\),  \(-\infty<l_x(i)<i\) holds since \(t_x(i)=\min\{y(m,i)\vee -1: m<i\}\) and \(l_x(i)\) is the largest \(m\) for which \(y(m,i)\) attains \(t_x(i)\), when \(m\) is varied in \(\mathbb{Z}_{<i}\). 
 If \(t_x(i)=-1\), then \(l_x(i)=L_x(i)-1\).
 For any integer \(i\), \(t_x(i) \geq 0\) if and only if \(L_x(i)=-\infty\).

The following lemma relates the offspring count of any vertex \(i\) in the record graph of \((\mathbb{Z},x)\) and its associated increment \(x_{i-1}\) when \(L_x(i)< - \infty\).
Note that, skip-free to the left condition is necessary for this lemma to hold.

\begin{lemma}\label{lemma_offspring_count}
    Let $x = (x_k)_{k \in \mathbb{Z}}$ be a sequence such that $x_k \in \{-1,0,1,2,\ldots\}$ for all \(k \in \mathbb{Z}\).
    Let $(\mathbb{Z},x)$ be the network associated to \(x\), and $T$ be the record graph of \((\mathbb{Z},x)\).
    Let $i$ be an integer such that $L_x(i)>-\infty$ (see Eq.~(\ref{eq:L_x_defn})).
    Then,
    \begin{enumerate}
      \item The number of children $d_1(i)$ of $i$ in $T$ is given by $d_1(i) = x_{i-1}+1$.
      \item Let $d_1(i)=n>0$ and $i_n<i_{n-1}<\cdots<i_1$ be the positions on $\mathbb{Z}$ of the \(n\) children of $i$. Then, the position $i_m$ of the $m$-th child satisfies the relation $m = x_{i-1}+1-y(i_m,i)$.
      \item Let $d_1(i) = n>0$, and $m \in \{1,2,\cdots,n\}$. Then, an integer $i'<i$ is the $m$-th child of $i$ if and only if $i'$ is the largest integer among $\{L_x(i), \cdots,i-1\}$ that satisfies $y(i',i) = x_{i-1}+1-m$.
    \end{enumerate} 
  \end{lemma}
  \begin{proof}
    If $x_{i-1}=-1$, then $L_x(i) = i$. So, by Lemma~\ref{lemma:descendants}, $i$ has no descendant.
    Then, \((1.)\) holds as \(d_1(i) = x_{i-1}+1=0\), whereas the other statements are empty.
  
    So, assume that $x_{i-1}\geq 0$.
    Clearly $i-1$ is then a child of $i$ as \(i\) is the record of \(i-1\).
    Therefore, \(d_1(i)\geq 1\).
    It is shown below that the sum of increments between any two successive children is $-1$.

    Let $i_n<i_{n-1}<\cdots <i_1$ be the positions on $\mathbb{Z}$ of the children of $i$, with $i_1=i-1$.
  Consider a child $i_m$ with $m<n$.
  If $i_{m+1} = i_m-1$, then $x_{i_m-1}=-1$. 
  Therefore, $y(i_{m+1},i_m)=-1$. 
  Consider now the case where $i_{m+1}<i_m-1$, and let $i'$ be an integer with $i_{m+1}<i'<i_{m}$.
    As $i'$ is not a child of $i$, and $i_m$ is a child of $i$, by Lemma~\ref{lemma:record_descendants}, one has $R_x(i')\leq i_m$, and there exists a smallest $k>0$ such that $R_x^k(i')=i_m$.
    Therefore, $y(i',i_m) = \sum_{l=0}^{k-1} y(R_x^l(i'),R_x^{l+1}(i')) \geq 0$. 
  
    In particular, $y(i_{m+1}+1,i_m) \geq 0$. 
  But, $y(i_{m+1},i_m)<0$, because the record of $i_{m+1}$ is \(i\) (not $i_m$).
  Therefore, the only possibility is that $x_{i_{m+1}}=-1$.
  This also implies that $y(i_{m+1}+1,i_m)=0$.
  Indeed, if $y(i_{m+1}+1,i_m)\geq 1$, then $y(i_{m+1},i_m)=x_{i_{m+1}}+y(i_{m+1}+1,i_m)=-1+y(i_{m+1}+1,i_m) \geq 0$, which gives a contradiction that \(i_m\) is the record of \(i_{m+1}\).
  Hence,  $y(i_{m+1},i_m)=x_{i_{m+1}}+y(i_{m+1}+1,i_m)=-1$.
  Thus, it was shown that the sum of increments between two successive children is \(-1\). 
  This is now used to prove the three statements of the lemma.
  
  The proof of the \emph{second statement} follows because, for any child $i_m$ of $i$,  one has 
  \begin{equation} \label{eqn:cond_child}
    y(i_m,i)=y(i_m,i_1)+x_{i-1}=y(i_m,i_{m-1})+\cdots+y(i_2,i_1)+x_{i-1}=-(m-1)+x_{i-1}.
  \end{equation}
  The proof of the third statement is as follows.
  Since \(i_m\) is a child of \(i\), \(y(i_m,i)\geq 0\), for all  \(1 \leq m \leq n\).
  Further, for any $i_m<i' \leq i-1$, $y(i',i) = y(i',i_j)+y(i_j,i)$ holds, with $i_j\geq i'$ being the unique smallest integer which is a child of $i$.
  Such an \(i_j\) always exists since \(i-1\) is a child of \(i\).
  Moreover, \(j<m\).
  Since \(y(i',i_j) \geq 0\), one has
  \[y(i',i)=y(i',i_j)+x_{i-1}+1-j \geq x_{i-1}+1-j > x_{i-1}+1-m,\]
  where Eq.~(\ref{eqn:cond_child}) was used to get the first equality.
  This proves the \emph{forward implication of the third statement}: if $h$ is the $m$-th child of $i$, then $h<i$ is the largest integer that satisfies $y(h,i)=x_{i-1}+1-m$.
  
  As for the \emph{backward implication of the third statement}, if $i'<i$ is the largest integer satisfying $y(i',i)=x_{i-1}+1-m \geq 0$ for some $m \in \{1,2, \cdots, n\}$, then $y(i',j)= y(i',i)-y(j,i)<0$, for all $i'<j \leq i-1$.
  So, $i'$ is a child of $i$. 
  But, by Eq.~(\ref{eqn:cond_child}), there is only one child that satisfies this condition, namely, the $m$-th child of $i$.
  This proves the third statement.
  
  The next step consists in proving that $y(L_x(i),i)=0$.
  Indeed, from the definition of $L_x(i)$, it follows that $y(L_{x}(i)-1, i) = x_{L_x(i)-1}+y(L_x(i),i)<0$.
  But $y(L_x(i),i) \geq 0$, implying that $x_{L_{x}(i)-1}=-1$ (the only possibility for \(x_{L_{x}(i)-1}\)).
  Hence, $y(L_x(i),i)=0$.
  This implies that $i_n = L_x(i)$ since no integer smaller than \(L_x(i)\) can be a descendant of \(i\) by Lemma~\ref{lemma:descendants}.
  To prove the first statement, observe that 
  \begin{equation*}
    y(i_n,i) = x_{i-1}+\sum_{k=1}^{n-1} y(i_k,i_{k+1}) = x_{i-1}-(n-1).
  \end{equation*}
  Therefore, $y(i_n,i)=0=x_{i-1}+1-n$ which gives $d_1(i) = n =x_{i-1}+1$, proving the \emph{first statement}.
  \end{proof}
  
  \begin{remark}\label{remark_smallest_child_type_negative}
    From the last part of the proof of Lemma~\ref{lemma_offspring_count}, it follows that, if \(L_x(i)>-\infty\), then \(L_x(i)\) is the smallest child of \(i\).
  \end{remark}

  The following lemma describes the smallest child of a vertex in the record graph.

\begin{lemma}\label{defn_l_20230404185835}\label{smallest_child_20230404185835}
    Let $x = (x_k)_{k \in \mathbb{Z}}$ be a sequence such that $x_k \in \{-1,0,1,2,\ldots\}$ for all \(k \in \mathbb{Z}\).
    Let $(\mathbb{Z},x)$ be the network associated to \(x\), and $T$ be the record graph of \((\mathbb{Z},x)\).
    For any integer \(i\), the following dichotomy holds:
  \begin{itemize}
    \item if \(t_x(i)=-1\), then \(L_x(i)\) is the smallest child of \(i\) in \(T\),
    \item if \(t_x(i)\geq 0\), then \(l_x(i)\) is the smallest child of \(i\) in \(T\).
  \end{itemize}
\end{lemma}

\begin{proof}
  If \(t_x(i)=-1\), then \(L_x(i)>-\infty\).
  So, by Remark \ref{remark_smallest_child_type_negative}, \(L_x(i)\) is the smallest child of \(i\) in \(T\).

  So, assume that \(t_x(i)\geq 0\).
  In this case, $i$ is the record of $l_x(i)$, i.e., $i = R(l_x(i))$.
  This follows because, for any $l_x(i)<m<i$, one has $y(m,i)>t_x(i)$ by the definition of $t_x(i)$ and $l_x(i)$.
  Since $y(l_x(i),i)=t_x(i)$, it follows that $y(l_x(i),m) = y(l_x(i),i)-y(m,i)<0$.
  Thus, $R(l_x(i))=i$.
  Further, if $t_x(i)\geq 0$, then $y(m,l_x(i)) \geq 0$, for all $m<l_x(i)$.
  Indeed, for \(m<l_x(i)\), since $y(m,i) \geq t_x(i)$, one has $y(m,l_x(i)) = y(m,i)-y(l_x(i),i) \geq t_x(i)-t_x(i)=0$.
  So, $R(m) \leq l_x(i)$ for all $m<l_x(i)$.
  This implies that $l_x(i)$ is the smallest among the children of $i$ as none of the integers smaller than $l_x(i)$ are children of $i$.
\end{proof}

%!2.41! Relation between the type of a vertex and number of its children (20230405112810)
The following lemma describes the relation between the number of children of any integer \(i\), its associated increment \(x_i\) and its type \(t_x(i)\) when \(t_x(i) \geq 0\).
When \(t_x(i)=-1\), the type does not play any role in determining this relation.

\begin{lemma}\label{20230405112810}
  Let $i \in \mathbb{Z}$ be an integer.
  The following dichotomy holds:
  \begin{itemize}
    \item If \(t_x(i)=-1\), then the number of children of $i$ in $T$ is given by \(d_1(i,T)=x_{i-1}+1\).
    \item If \(t_x(i) \geq 0\), then the number of children of $i$ in $T$ is given by $d_1(i,T) = x_{i-1}+1-t_x(i)$.
  \end{itemize}
\end{lemma}

\begin{figure}[h]
    \begin{center}
      \includegraphics[scale=0.92]{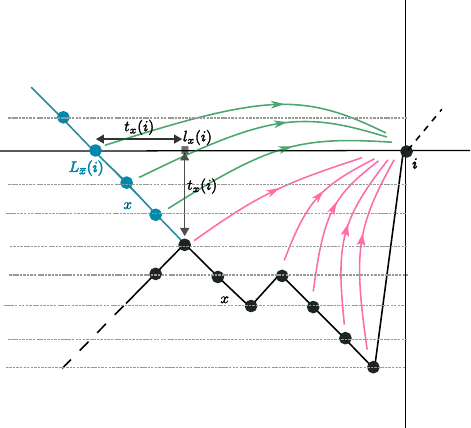}
    \end{center}
    \caption{The trajectories of $x=(x_n)_{n \in \mathbb{Z}}$ (in black) and $\bar{x}=(\bar{x}_n)_{n \in \mathbb{Z}}$ (in blue). The children of $i$ in $T$ are drawn in red, and the additional children of $i$ in $\bar{T}$ are drawn in green. See Lemma~\ref{20230405112810}.}
    \label{fig:relation_children_type_20230405112810}
\end{figure}
\begin{proof}
  If \(t_x(i)=-1\), then \(L_x(i) > - \infty\) (see the paragraph after Eq.~(\ref{eqn:def_l_x})), which is the condition needed to apply Lemma~\ref{lemma_offspring_count}.
  The first part of the statement follows from part 1 of Lemma~\ref{lemma_offspring_count}.

  The second part of the statement will now be proved.
  To follow the proof, see Fig. \ref{fig:relation_children_type_20230405112810}. Define a new sequence $\bar{x} = (\bar{x}_n)_{n \in \mathbb{Z}}$ by
  \[\bar{x}_n = \begin{cases}
    -1, \forall n<l_x(i)\\
    x_n, \forall n \geq l_x(i). 
  \end{cases}\]
  Let $\bar{y}(j,k)=\sum_{l=j}^{k-1}\bar{x}_l$ for all integers $j<k$.
  Then, $y(j,k)=\bar{y}(j,k)$ for all $ l_x(i) \leq j < k$.
  Thus, 
  \begin{equation}\label{eq:undisturbed_x_20230405112810}
      \bar{y}(k,i) = y(k,i) \geq t_x(i)>0 \quad \forall l_x(i)\leq k < i,
  \end{equation}
  and 
  \[\bar{y}(k,i) = k- l_x(i)+t_x(i) \quad \forall k<l_x(i).\]
  Therefore, (see Eq.(\ref{eq:L_x_defn}) for the definition of \(L\)) 
  \begin{equation}\label{eq:rel_l}
    L_{\bar{x}}(i) = l_x(i) - t_x(i)> - \infty.
  \end{equation}
  Let $\bar{T}$ denote the record graph of the network $(\mathbb{Z},\bar{x})$, $D_1(i,\bar{T})$ denote the set of children of $i$ in $\bar{T}$ and let $d_1(i,\bar{T}) = \#D_1(i,\bar{T})$.
  Then, by the first part of Lemma~\ref{lemma_offspring_count}, it follows that $d_1(i,\bar{T})=\bar{x}_{i-1}+1$.
  Since $l_x(i)<i$, one has $x_{i-1} = \bar{x}_{i-1}$ which implies that $d_1(i,\bar{T}) = x_{i-1}+1$.

  Note that $\bar{y}(l_x(i),i)=t_x(i)$ and $\bar{y}(j,i)>t_x(i)$ for all $l_x(i)<j<i$.
  Therefore, by the third part of Lemma~\ref{lemma_offspring_count}, $l_x(i)$ is the $x_{i-1}+1-t_x(i)$-th child of $i$ in $\bar{T}$.
  Since $\bar{x}_j=-1$ for all $j<l_x(i)$, by (3) of Lemma~\ref{lemma_offspring_count}, one gets that
  \[\{j \in \mathbb{Z}: L_{\bar{x}}(i) \leq j < l_x(i)\} \subseteq D_1(i,\bar{T}).  \]

  By Eq. (\ref{eq:rel_l}), $\#\{j \in \mathbb{Z}:L_{\bar{x}}(i)  \leq j < l_x(i)\}=l_x(i)-L_{\bar{x}}(i)=t_x(i)$.
  Observe that $D_1(i,\bar{T})\cap \{j \in \mathbb{Z}: l_x(i)\leq j < i\} = D_1(i,T)$ by Eq. (\ref{eq:undisturbed_x_20230405112810}).
  Finally, Lemma~\ref{defn_l_20230404185835} implies that $l_x(i)$ is the smallest child of $i$ in $T$.
  Thus,
  \begin{align*}
      D_1(i,T) &= D_1(i,\bar{T})\backslash \{j \in \mathbb{Z}: L_{\bar{x}} \leq j<l_x(i)\},
  \end{align*}
  which gives
  \[d_1(i,T)=d_1(i,\bar{T})- t_x(i)=x_{i-1}+1-t_x(i).\]
\end{proof}

\subsection{Properties of skip-free to the left random walks} \label{subsec_properties_skip_free}
 The lemmas in this subsection can be found in \cite[Chapter 5]{nicolascurienRandomWalksGraphs}.
They are presented here utilizing the notations established in this paper for the sake of completeness and to maintain consistency.
Before delving into the proofs of the lemmas, the following notations are introduced.

For any \(j \in \mathbb{Z}\), let \(\eta_j\) be the hitting time of the random walk \(S\) defined as 
\begin{equation}\label{eq_eta_j}
    \eta_j:=\begin{cases}
      \inf\{n\geq 0: S_n = j\} \text{ if } S_n=j \text{ for some } n\geq 0\\
      \infty \text{ otherwise.}
    \end{cases}
\end{equation}
Let \(\tau\) be the \emph{weak upper record} (epoch) of the random walk \(S\) defined as 
\begin{equation}\label{eq_tau_j_weak_record}
    \tau:=\begin{cases}
        \infty \text{ if } S_n<0 \  \forall n>0,\\
        \inf\{n>0: S_n \geq 0\} \text{ otherwise,}
    \end{cases}
\end{equation}
$S_{\tau}$ be the \emph{weak upper record height}, and $X_{\tau-1}$ be its last increment (both are defined to be arbitrary when \(\tau = \infty\)).
For an event \(A\), let $\mathbb{P}_k[A]$ denote the probability of the event $A$ when the random walk starts at $k \in \mathbb{Z}$.
Let $\mathbb{P}:=\mathbb{P}_0$, and $c := \mathbb{P}[\eta_{-1} < \infty]$.
The last condition of Eq.~(\ref{eq_skip_free_condition}) implies that \(0 <c \leq 1\).

\begin{lemma}  \label{20230118184651}\label{hitting_time_20230118184651}
    For $0 \leq j \leq k$, $\mathbb{P}[\eta_{j-k}< \infty]= \mathbb{P}_k[\eta_j < \infty]= c^{k-j}$.
\end{lemma}
\begin{proof}
    The first equality in the statement is obvious since the probability is left invariant by a shift of the starting point of the random walk.
    The second equality follows because of the skip free property which prohibits the random walk from taking jumps smaller than $-1$.
    It is proved by induction on $k-j$.

    If $j=k$, then $\mathbb{P}_j[\eta_j<\infty]=1$.
    So, assume that $0 \leq j <k$.
    By induction, \(\mathbb{P}_{k'}[\eta_{j'}< \infty] = c^{k'-j'}\) for all $0 \leq j' \leq k'$ such that $k'-j'<k-j$. 
    Then,
    \begin{align*}
        \mathbb{P}_k[\eta_j< \infty] = \mathbb{P}_k[\eta_{k-1}< \infty]\mathbb{P}_{k-1}[\eta_j < \infty] = \mathbb{P}[\eta_{-1}< \infty]c^{k-1-j} = c^{k-j},
    \end{align*}
    where the second equation is obtained by applying the inductive statement to $k'=k-1$, and $j'=j$.
\end{proof}

\begin{lemma}[\cite{benniesRandomWalkApproach2000}]\label{20230119141633}
    For all integers $j,k$ such that $0 \leq j \leq k$, 
   \[\mathbb{P}[\tau<\infty,S_{\tau}=j,X_{\tau-1}=k] = \mathbb{P}[X_0=k] c^{k-j}.\]
\end{lemma}
\begin{proof}
    Note the equality of the following events:
    \[\{S_{\tau}=X_{\tau-1},\tau<\infty\} = \{\tau = 1\} = \{S_1 \geq 0\}.\]
    Therefore, for $0 \leq j=k$, one has $\mathbb{P}[S_{\tau}=j,X_{\tau-1}=j,\tau< \infty] = \mathbb{P}[S_1 = j]$.

    So, for $0 \leq j <k$,
    \begin{align*}
        \mathbb{P}[\tau<\infty,S_{\tau}=j,X_{\tau-1}=k] &= \sum_{n=2}^{\infty}\mathbb{P}[\tau=n,S_n=j,X_{n-1}=k]\\
        &=\sum_{n=2}^{\infty} \mathbb{P}[S_1<0,\cdots,S_{n-1}<0,S_n=j,S_n-S_{n-1}=k].
    \end{align*}
The sum in the first equation starts from $n=2$ since it is assumed that $j<k$ (\(\{\tau=1\}\) occurs if and only if \(j=k\), which is already covered).
By the duality principle (also known as reflection principle), the last expression is equal to
\[\sum_{n=2}^{\infty}\mathbb{P}[S_n-S_{n-1}<0, S_n - S_{n-2}<0, \cdots,S_n-S_1<0, S_n=j,S_n-(S_n-S_1)=k].\]
So, one gets:
\begin{align*}
  \mathbb{P}[\tau<\infty,S_{\tau}=j,X_{\tau-1}=k]  &= \sum_{n=2}^{\infty}\mathbb{P}[S_n=j, S_{n-1}>j,S_{n-2}>j, \cdots,S_1>j,S_1=k]\\
  &=\sum_{n=2}^{\infty} \mathbb{P}[S_1=k] \mathbb{P}_k[S_1>j,\cdots,S_{n-2}>j,S_{n-1}=j]\\
        &=\mathbb{P}[S_1=k] \sum_{n=2}^{\infty}\mathbb{P}_k[\eta_{j}=n-1]\\
        &= \mathbb{P}[S_1=k] \mathbb{P}_k[0<\eta_j< \infty]\\
        &= \mathbb{P}[S_1=k]c^{k-j} = \mathbb{P}[X_0=k]c^{k-j}.
\end{align*}
The second equation follows by the Markov property.
The second to the last equation follows from Lemma~\ref{20230118184651} and from the fact that $\mathbb{P}[\eta_{j-k}=0]=\delta_j(k)$, where $\delta$ is the Dirac function.
\end{proof}

\begin{remark}
    If the random walk has negative drift, i.e., $\mathbb{E}[X_0]<0$, then $S_n \to -\infty$ a.s. as $n \to \infty$.
    Therefore, $c=1$, which implies that $\mathbb{P}[\tau<\infty,S_{\tau}=j,X_{\tau-1}=k] = \mathbb{P}[X_0=k]$, and 
    \[\mathbb{P}[\tau< \infty, X_{\tau-1}=k] = \sum_{j=0}^{k}\mathbb{P}[\tau<\infty,S_{\tau}=j,X_{\tau-1}=k] = (k+1) \mathbb{P}[X_0=k]. \]
\end{remark}

\begin{remark}
    If the random walk has positive drift, then $S_n \to \infty$ a.s., as $n \to \infty$.
    So, \(\tau<\infty\) a.s..
    Therefore, $\mathbb{P}[\tau< \infty, S_{\tau}=j, X_{\tau-1}=k]= \mathbb{P}[S_{\tau}=j,X_{\tau-1}=k]$ for $0 \leq j \leq k$.
\end{remark}

Assume that \(\mathbb{E}[X_0]>0\).
In this case, the following lemma shows that the random walk conditioned to hit \(-1\) is also a skip-free to the left random walk with a different increment distribution.

Let \(p=(p_k)_{k \geq -1}\) be the distribution of the increment \(X_0\), i.e., \(p_k := \mathbb{P}[X_0 = k]\), for all \(k \in \mathbb{Z}_{\geq -1}\).
Consider the harmonic function \(h\) (with respect to the random walk \(S\)) on \(\mathbb{Z}_{\geq 0}\) defined by \(h(i) = \mathbb{P}_i[\eta_{-1}<\infty]\), for all \(i \in \mathbb{Z}_{\geq 0}\).
In particular, \(h(0) = \mathbb{P}_0[\eta_{-1}<\infty]=c\).
Since \(h\) is harmonic, \(h(0) = \sum_{k=-1}^{\infty}p_kh(k) = \sum_{k=-1}^{\infty}p_k c^{k+1}\), which follows from Lemma~\ref{hitting_time_20230118184651}.
Therefore, 
\begin{equation} \label{eq_pi_tilde_distribution}
  \sum_{k=-1}^{\infty} p_k c^k = \sum_{k=-1}^{\infty} \mathbb{P}[X_0=k] c^k = 1.
\end{equation}

Let \(\hat{p} = (\hat{p}_k)_{k \geq -1}\) be the Doob \(h\)-transform of \(p\), where \(\hat{p}_k:=p_kc^k\).
Let 
\begin{equation} \label{eq_conditioned_RW}
    \hat{X}=(\hat{X}_n)_{n \in \mathbb{Z}}
\end{equation}
be an i.i.d. sequence of random variables whose common distribution is \(\hat{p}\), and \((\hat{S}_n)_{n \in \mathbb{Z}}\) be the skip-free to the left random walk whose increment sequence is \((\hat{X}_n)_{n \in \mathbb{Z}}\), defined as in Eq.~(\ref{eq_sums_increment}).

\begin{lemma}\label{20230113184856}
    Let \(\mathbb{E}[X_0]>0\), and \(\hat{\eta}_{-1}\) be the hitting time as defined in Eq.~(\ref{eq_eta_j}) for the process \((\hat{S}_n)_{n \in \mathbb{Z}}\).
    Then, the stopped random walk \(\hat{S} = (\hat{S}_{n \wedge \hat{\eta}_{-1}})_{n \geq 0}\) has the same law as the stopped random walk \((S_{n\wedge \eta_{-1}})_{n \geq 0}\) conditioned on \(\eta_{-1}< \infty\).
\end{lemma}
\begin{proof}
    Let $(x_0,x_1,\cdots,x_n)$ be arbitrary integers that satisfy $x_k \geq -1, \forall k \leq n$, $x_1+\cdots+x_m \geq 0, \forall m<n$, and $x_1+\cdots+x_n \geq -1$.
    Then, for all such \(n\) and \(x\),
  
    % \begin{align*}
    %   &\mathbb{P}[S_0=0,S_1=x_1, S_2-S_1=x_2,\cdots,S_n-S_{n-1}=x_n|\eta_{-1}< \infty] \\
    %   &= \frac{\mathbb{P}[S_1=x_1]\mathbb{P}_{x_1}[S_1 = x_2,S_2-S_1=x_3,\cdots, S_{n-1}-S_{n-2}=x_{n},\eta_{-1}<\infty]}{\mathbb{P}[\eta_{-1}< \infty]}\\
    %   &= \frac{\left(\prod_{i=1}^n \mathbb{P}[S_1=x_i]\right) \mathbb{P}_{x_1+\cdots+x_n}[\eta_{-1}< \infty]}{\mathbb{P}[\eta_{-1}< \infty]}= \frac{\left(\prod_{i=1}^n \mathbb{P}[S_1=x_i]\right) c^{x_1+x_2+\cdots+x_n+1}}{c}\\
    %  &= \prod_{i=1}^n (\mathbb{P}[X_0=x_i]c^{x_i}).
    % \end{align*}
   \begin{align*}
      &\mathbb{P}[S_0=0,S_1=x_1, S_2-S_1=x_2,\cdots,S_{(n \wedge \eta_{-1})}-S_{(n \wedge \eta_{-1}-1)}=x_n, \eta_{-1} \geq n|\eta_{-1}< \infty] \\
      &= \frac{\mathbb{P}[S_1=x_1]\mathbb{P}_{x_1}[S_1 = x_2,S_2-S_1=x_3,\cdots, S_{n-1}-S_{n-2}=x_{n},\eta_{-1}<\infty]}{\mathbb{P}[\eta_{-1}< \infty]}\\
      &= \frac{\left(\prod_{i=1}^n \mathbb{P}[S_1=x_i]\right) \mathbb{P}_{x_1+\cdots+x_n}[\eta_{-1}< \infty]}{\mathbb{P}[\eta_{-1}< \infty]}= \frac{\left(\prod_{i=1}^n \mathbb{P}[S_1=x_i]\right) c^{x_1+x_2+\cdots+x_n+1}}{c}\\
     &= \prod_{i=1}^n (\mathbb{P}[X_0=x_i]c^{x_i})\\
     &=\mathbb{P}[\hat{S}_1=x_1,\hat{S}_2-\hat{S}_1 = x_2, \cdots, \hat{S}_{(n \wedge \eta_{-1})}- \hat{S}_{(n \wedge \eta_{-1}-1)}=x_n, \hat{\eta}_{-1} \geq n].
    \end{align*}
      Thus, the conditioned random walk is a stopped random walk, stopped at the hitting time \(\hat{\eta}_{-1}\), whose increments have the common distribution \(\hat{p}\).
\end{proof}

The following lemma shows that the random walk \((\hat{S}_n)_{n \geq 0}\) has negative drift.
  
  \begin{lemma}\label{20230116135048}
    Assume \(\mathbb{E}[X_0]>0\). Then, \(\sum_{k=-1}^{\infty}k\hat{p}_k <0\), i.e., the mean of the increment \(\hat{X}_0\) is negative.
  \end{lemma}
  \begin{proof}
  %We show that $\mathbb{E}[\hat{X}_0]= \sum_{k=-1}^{\infty}kc^kp_k<0$.
  
    Consider the function $\phi(x)=\sum_{k=0}^{\infty}(k+1)x^kp_k -1$ on the interval $[0,1]$.
    Using Eq. (\ref{eq_pi_tilde_distribution}) in the following, one obtains 
    \begin{align*}
      \phi(c) &= \sum_{k=0}^{\infty} (k+1)c^k p_k- 1 = \sum_{k=0}^{\infty} kc^k p_k +\left(\sum_{k=0}^{\infty}c^kp_k \right)- 1\\
      &= \sum_{k=0}^{\infty}kc^kp_k + 1- \frac{p_{-1}}{c} -1= \sum_{k=-1}^{\infty}kc^kp_k= \mathbb{E}[\hat{X}_0].
    \end{align*}
    The function $\phi$ has the following properties on $[0,1]$:
    \begin{itemize}
      \item $\phi(0)<0$: Since $\phi(0)=p_0-1<0$.
      \item $\phi(1)>0$: Since $\phi(1)= \sum_{k=0}^{\infty}(k+1)p_k-1 = \sum_{k=0}^{\infty}kp_k + 1-p_{-1} - 1 = \mathbb{E}[X_0]>0$.
      \item $\phi$ is strictly increasing since $\phi'(x)>0, \forall x \in (0,1)$.
    \end{itemize}
    Therefore, there exists a unique $c' \in (0,1)$ such that $\phi(c')=0$.
    The proof is complete if it is shown that $c<c'$, as this implies that $\phi(c)<\phi(c')=0$ by monotonicity of $\phi$, and from the fact that $\phi(c)=\mathbb{E}[\hat{X}_0]$.
  
    It is now shown that $c<c'$.
    Consider the function $\psi(x) = \sum_{k=0}^{\infty}x^{k+1}p_k + p_{-1} - x$ on the interval $[0,1]$.
    Observe that $\psi'(x) = \sum_{k=0}^{\infty}(k+1)x^kp_k - 1 = \phi(x)$,
    \[\psi(c)=\sum_{k=0}^{\infty}c^{k+1}p_k+p_{-1}-c = c \left(\sum_{k=0}^{\infty}c^k p_k + \frac{p_{-1}}{c} -1\right)=0,\]
    and $\psi(1) = \sum_{k=0}^{\infty}p_k+p_{-1}-1=0$.
    As $\psi(c)=\psi(1)=0$, there exists a $d \in (c,1)$ such that $\psi'(d)=0$.
    Since $\phi(d)=\psi'(d)=0$, one has $d=c'$.
    Thus, $c<c'$. 
  \end{proof}

  \subsection{Instances of unimodular Family Trees} \label{sec_examples}
  This subsection introduces two parametric families of unimodular Family Trees, namely, the typically rooted Galton-Watson Tree; which is of class \(\mathcal{F}/\mathcal{F}\), and the unimodularised marked ECS ordered bi-variate Eternal Kesten tree; which is of class \(\mathcal{I}/\mathcal{F}\).
  An instance of parametric family of unimodular Family Trees of class \(\mathcal{I}/\mathcal{I}\) called the unimodular Eternal Galton-Watson Trees parametrized by the offspring distribution \(\pi\) has already been described in Subsection \ref{subsec_EGWT}, where the mean \(m(\pi)=1\) and \(\pi(1)<1\).
  All the above Family Trees appear as the component of \(0\) in the record graph depending upon the mean of the increment.
  This will be proven in the next subsection.

\subsubsection{Typically rooted Galton-Watson Tree (TGWT)}
Let $\pi$ be a probability distribution on $\{0,1,2,3,\cdots\}$ such that its mean $0\leq m(\pi)<1$ and \(\hat{\pi}\) be the size-biased distribution of \(\pi\), given by \(\hat{\pi}(k) = \frac{k \pi(k)}{m(\pi)}\) for all \(k \in \{0,1,2,\cdots\}\).
The {\em Typically rooted Galton-Watson Tree} (\(TGWT(\pi)\)) with offspring distribution $\pi$ is a Family Tree $[\mathbf{T},\mathbf{o}]$ defined in the following way.
\begin{itemize}
    \item Add a parent to the root $\mathbf{o}$ with probability $m(\pi)$. Independently iterate the same to the parent of $\mathbf{o}$ and to all of its ancestors. So, the number of ancestors of the root $\mathbf{o}$ has a geometric distribution with the success probability $1-m(\pi)$.
    \item Let $Z$ be a random variable that has distribution $\hat{\pi}$ (the size-biased distribution of $\pi$). To each of the ancestors of $\mathbf{o}$, attach a random number of children independently with distribution the same as that of $Z-1$. Assign uniform order among the children of every ancestor of \(\mathbf{o}\).
    \item To each of these new children and to the root $\mathbf{o}$, attach independently ordered Galton-Watson Trees with offspring distribution $\pi$.
\end{itemize}
See Figure~\ref{figure_sgwt_2.31} for an illustration.

\begin{remark}
    The \(TGWT(\pi)\) is a finite unimodular ordered Family Tree (see Proposition~\ref{20230305162953}).
\end{remark}

\begin{figure}[h]
\begin{center}
      \includegraphics[scale=0.8]{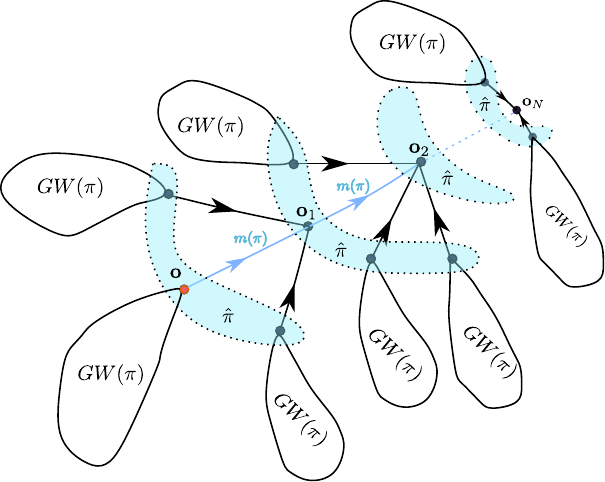}    
\end{center}
    \caption{Typically rooted Galton-Watson Tree with offspring distribution $\pi$ (\(\mathbf{o}\) is the root), \(m(\pi)\) is the probability that the root \(\mathbf{o}\) has a parent, \(\hat{\pi}\) is the offspring distribution of the parent of \(\mathbf{o}\), \(N\) is a geometric random variable with the success probability \(1-m(\pi)\).}
    \label{figure_sgwt_2.31}
\end{figure}

The nomenclature ``typically rooted'' in the typically rooted Galton-Watson Tree suggests that it is obtained by re-rooting to a typical vertex of a Galton-Watson Tree as in Def.~\ref{def_typical_rerooting}.
This is indeed true and is proved in Proposition~\ref{prop_tgwt_is_typical}.
A characterizing condition for \(TGWT\) is given in Proposition~\ref{20230305162953}.

\paragraph*{Bibliographic comment} A construction for the modified subcritical Galton-Watson Tree \(\mathfrak{T}\), a random tree that resembles \(TGWT\), is given by Jonsson and Stefánsson \cite{jonssonCondensationNongenericTrees2011} (for a simpler construction, see the subcritical case of \cite[Section 5]{jansonSimplyGeneratedTrees2012}).
In their construction, a vertex can have infinite degree with positive probability and the tree is undirected.
The special vertices of \(\mathfrak{T}\), which belong to the spine, have the offspring distribution same as that of the ancestors of the root of \(TGWT\), except for the last special vertex of \(\mathfrak{T}\) which may have infinite children.
Note that the undirected tree obtained by forgetting the directions of edges of \(TGWT\) is not the same as \(\mathfrak{T}\). 

\subsubsection{Bi-variate Eternal Kesten Tree}\label{subsec_bi_variate_ekt}
The unimodularised version of this tree is an instance of a unimodular tree belonging to the class \(\mathcal{I}/\mathcal{F}\).
Its construction is given using the typically re-rooted joining operation (as in Def.~\ref{def_typical_reroot_joining}).
See Figure~\ref{fig_bi_variate_EKT_alpha_beta} for an illustration.

The next paragraph introduces the bi-variate Eternal Kesten Tree, denoted as \(EKT(\alpha,\beta)\), with the offspring distributions \(\alpha\) and \(\beta\).
This tree is unordered, unlabeled and parametrized by \(\alpha\) and \(\beta\).
Although \(EKT(\alpha,\beta)\) is not unimodular, a unimodularised version can be obtained from it using the typically re-rooted joining operation.
To use this operation, it is necessary to have \(m(\beta)<1\).
However, to define \(EKT(\alpha,\beta)\), it is sufficient to have \(m(\alpha)<\infty\) and \(m(\beta)\leq 1\).

Let $\alpha, \beta$ be two probability distributions on $\{0,1,2,\cdots\}$ such that $m(\alpha)<\infty$ and $m(\beta) \leq 1$.
A {\em bi-variate Eternal Kesten Tree} with offspring distributions $\alpha,\beta$ is a random Family Tree $[\mathbf{T}',\mathbf{o}']$ consisting of a unique bi-infinite $F$-path $(\mathbf{o}_n)_{n \in \mathbb{Z}}$, where $\mathbf{o}_0=\mathbf{o}'$, with the following property:
 the sequence of Family Trees $([D(\mathbf{o}_n)\backslash D(\mathbf{o}_{n-1}), \mathbf{o}_n])_{n \in \mathbb{Z}}$ is i.i.d. with the following common distribution: the offspring distribution of the root is \(\alpha\), the descendant trees of the children of the root are i.i.d. Galton-Watson Trees with offspring distribution $\beta$ (denoted as \(GW(\beta)\)) and they are independent of the number of children of the root.
The Family Trees \(\{[D(\mathbf{o}_n)\backslash D(\mathbf{o}_{n-1}), \mathbf{o}_n]:n \in \mathbb{Z}\}\) are called the {\em bushes} of \(EKT(\alpha,\beta)\) and the Family Tree \([D(\mathbf{o}_0)\backslash D(\mathbf{o}_{-1}), \mathbf{o}_0]\) is called the {\em bush of the root}.

Observe that, if \(\beta\) has mean \(1\) and \(\alpha = \hat{\beta}-1\), where \(\hat{\beta}\) is the size-biased distribution of \(\beta\), then the descendant tree of the root \([D(\mathbf{o}'),\mathbf{o}']\) is the usual Kesten tree (see \cite[Section 2.3]{abrahamIntroductionGaltonWatsonTrees2015}).

Note that $EKT(\alpha,\beta)$ is the joining of the i.i.d. sequence $([\mathbf{T}_i,\mathbf{o}_i])_{i \in \mathbb{Z}}$, where $T_i = D(\mathbf{o}_i)\backslash D(\mathbf{o}_{i-1})$.
Keep in mind that $d_1(\mathbf{o}_i,\mathbf{T}_i)= d_1(\mathbf{o}_i,\mathbf{T}')-1$, where \(d_1(\mathbf{o}_i,\mathbf{T}_i)\) denotes the number of children of \(\mathbf{o}_i\) in \(\mathbf{T}_i\).

Furthermore, assume that \(m(\beta)<1\).
Then, \(GW(\beta)\) is subcritical, which implies that \(\mathbb{E}[\#V(\mathbf{T}_0)]< \infty\).
So, in this case, the necessary condition needed to apply typically re-rooted joining operation to the i.i.d. Family Trees \(([\mathbf{T}_i,\mathbf{o}_i])_{i \in \mathbb{Z}}\) is satisfied.
Let \([\mathbf{T},\mathbf{o}]\) be the typically re-rooted joining of \(([\mathbf{T}_i,\mathbf{o}_i])_{i \in \mathbb{Z}}\) (as in Def.~\ref{def_typical_reroot_joining}).
Then, by Theorem~\ref{thm:I_F_unimodularizable}, the EFT \([\mathbf{T},\mathbf{o}]\) is unimodular.
The distribution of \([\mathbf{T},\mathbf{o}]\) is called {\em the unimodularised bi-variate \(EKT(\alpha,\beta)\)}.
Clearly, the unimodularised \(EKT(\alpha,\beta)\) is of class \(\mathcal{I}/\mathcal{F}\) as it has a (unique) bi-infinite \(F\)-path.

An \emph{Every Child Succeeding (ECS) order} on \([\mathbf{T}',\mathbf{o}']\) is obtained by declaring that $\mathbf{o}_n$ is the smallest among $D_1(\mathbf{o}_{n+1})$ and using the uniform order on the remaining children for all $n \in \mathbb{Z}$.
The name ``ECS'' order comes from the fact that the succession line (see Def.~\ref{defn_succession_line} ) starting from the root reaches all the vertices on the bi-infinite path
(i.e., the succession line is an order preserving bijection from \(\mathbb{Z}\) to the vertices of the tree).
The unimodularised ECS ordered \(EKT(\alpha,\beta)\) is an example of a unimodular ordered EFT of class \(\mathcal{I}/\mathcal{F}\).  
In the next subsection, it is shown that the unimodularised ECS ordered \(EKT(\alpha,\beta)\) is the component of \(0\) in the record graph of the network \((\mathbb{Z},X)\), where the sequence \(X=(X_n)_{n \in \mathbb{Z}}\) is an i.i.d. sequence satisfying the conditions in Eq.~\ref{eq_skip_free_condition} and \(\mathbb{E}[X_0]>0\).
In this case, the distributions \(\alpha, \beta\) are related to the distribution of \(X_0\).

\begin{figure}[h]
    \begin{center}
          \includegraphics[scale=1.6]{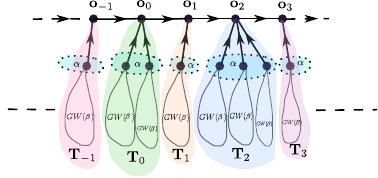}    
    \end{center}
        \caption{An illustration of bi-variate ECS ordered EKT with distributions $\alpha,\beta$ (\(\mathbf{o}_0\) is the root), \(\alpha\) is the offspring distribution of \(\mathbf{o}_0\) in \(\mathbf{T}_0\), \(\beta\) is the offspring distribution of every child of \(\mathbf{o}_0\) in \(\mathbf{T}_0\).}
        \label{fig_bi_variate_EKT_alpha_beta}
    \end{figure}

%\textcolor{RubineRed}{Why is the ECS ordered \(EKT(\alpha,\beta)\) after applying typical re-rooted joining unimodular? What are the other ways of joining an ordered stationary sequence to get a unimodular tree?}
%\textcolor{NiceBlue}{ (Prove this) Let \(([\mathbf{T}_n,\mathbf{o}_n,Z_n])_{n \in \mathbb{Z}}\) be a stationary sequence of marked trees such that \(\mathbb{E}[\mathbf{T}_0]< \infty\) and \(Z_n \leq d_1(\mathbf{o}_{n+1})\), for all \(n \in \mathbb{Z}\). Then, the EFT obtained by joining \(\mathbf{T}_n\) and \(\mathbf{T}_{n+1}\) in such a way that \(\mathbf{o}_n\) is the \(Z_n\)-th child of \(\mathbf{o}_{n+1}\), for all \(n \in \mathbb{Z}\) and size-biasing on \(\mathbf{T}_0\) is unimodular.}

%

Note that, to obtain a unimodular EFT, it is sufficient to follow a stationary rule of assigning order to the vertices in the typically re-rooted joining of a stationary sequence of Family Trees, as explained in the following.
Let \(([\mathbf{T}_n,\mathbf{o}_n,Z_n])_{n \in \mathbb{Z}}\) be stationary sequence of marked ordered Family Trees, where \(Z_n\) is a random variable such that \(Z_n \leq d_1(\mathbf{o}_{n+1})\), for all \(n \in \mathbb{Z}\).
Let \([\mathbf{T}',\mathbf{o}']\) be the ordered EFT obtained by the joining of \(([\mathbf{T}_n,\mathbf{o}_n,Z_n])_{n \in \mathbb{Z}}\), where the order on \(\mathbf{T}'\) is given by the following: for all \(n \in \mathbb{Z}\), the order of \(\mathbf{o}_n\) among the children of \(\mathbf{o}_{n+1}\) is \(Z_n\).
Let \([\mathbf{T},\mathbf{o}]\) denote the EFT obtained by re-rooting to the typical root of \(\mathbf{T}_0\) in \(\mathbf{T}'\).
Then, \([\mathbf{T},\mathbf{o}]\) is unimodular.
This follows because, the EFT \([\mathbf{T}',\mathbf{o}']\) is invariant under the vertex-shift \(f\), defined on the set of EFTs \(T\) that has a unique bi-infinite path \(S_T\) in the following way:
\begin{equation*}
    f_T(u)= \begin{cases}
        F_T(u) \text{ if } u \in S_T \\
        u \text{ otherwise,}
    \end{cases} 
\end{equation*}
where \(F_T(u)\) is the parent of \(u\).

\subsection{Distribution of the component in the record graph} \label{subsec_recordGraph_randomWalk}

This subsection focuses on the distribution of the component of \(0\) in the record graph of \((\mathbb{Z},X)\), where \(X= (X_n)_{n \in \mathbb{Z}}\) is any i.i.d. sequence satisfying conditions of Eq.~(\ref{eq_skip_free_condition}).
 This subsection is divided into three parts, each corresponds to the three phases \(\mathbb{E}[X_0]<0, \mathbb{E}[X_0]=0\) and \(\mathbb{E}[X_0]>0\) respectively.

The following proposition shows that when \(\mathbb{E}[X_0] \leq 0\), the descendant tree of \(0\) in the record graph is given by the excursion set associated to \((S_n)_{n \leq 0}\).
The case when \(\mathbb{E}[X_0]=0\) has also been addressed in \cite{benniesRandomWalkApproach2000}.

An \(n+1\)-tuple of integers of the form \((0,s_1,\cdots,s_{n-1},1)\) with \(s_j \leq 0\) for all \( 1 \leq j \leq n-1 \) is called an {\em excursion set}.
Let $\delta=\max\{n<0: S_n\geq 1\} = \max\{n<0: S_n = 1\}$ (equality follows from the skip-free property).

\begin{proposition}\label{20230128172431}
    Let \(X=(X_n)_{n \in \mathbb{Z}}\) be an i.i.d. sequence satisfying the conditions in Eq.~(\ref{eq_skip_free_condition}), \((\mathbb{Z},X)\) be its associated network, and \([D(0),0]\) be the descendant tree of \(0\) in the record graph of \((\mathbb{Z},X)\).
    If $\mathbb{E}[X_{0}]\leq 0$, then \([D(0),0]\) is the ordered Galton-Watson Tree with the offspring distribution $\pi \overset{\mathcal{D}}{=} X_{0}+1$.
  \end{proposition}
  \begin{proof}
  Since the random walk $(S_{-n})_{n \geq 0}$ drifts to $\infty$, one has $\delta>-\infty$.
  It follows from the definition of $L_X(0)$ (see Eq.~(\ref{eq:L_x_defn})) that $L_X(0)= \delta-1$ since $y(\delta,0) = -S_{\delta} = -1$ and for any $\delta-1\leq j<0$, $y(j,0)=-S_j>-1$.
  Thus, $L_X(0)> - \infty$ a.s., which is the required condition in order to apply Lemma~\ref{lemma_offspring_count}.
  The first part of Lemma~\ref{lemma_offspring_count} implies that $d_1(0)=X_{-1}+1 \overset{\mathcal{D}}{=} \pi$.
  
  It is now shown that, conditioned on the number of children of $0$, the subtrees of these children are jointly independent Galton-Watson Trees with offspring distribution $\pi$.
  
  For each $m \geq 1$, let $i_m = \max\{n<0:y(i_m,0)=X_{-1}+1-m\}$ and let \(i_0=0\).
  By the third part of Lemma~\ref{lemma_offspring_count}, $i_1>i_2>\cdots>i_{X_{-1}+1}$ are the children of $0$, and $i_{X_{-1}+1}=L_X(0)$.
  The fact that, for each $m\geq 1$, $i_m$ is a stopping time for the random walk $(S_{-n})_{n \geq 0}$, and the third part of Lemma~\ref{lemma_offspring_count}, together imply the following:
  the excursions $((S_{i_k}-S_{i_{k}},S_{i_k-1}-S_{i_k},\cdots,S_{i_{k+1}}-S_{i_{k}}))_{k=1}^{X_{-1}}$, with the notation $S_{i_0}=0$, are independent and identically distributed as $(S_0,S_{-1},\cdots,S_{\delta})$.
  
  By Lemma~\ref{lemma:descendants}, for any $1 \leq k \leq X_{-1}+1$, the descendant tree of $i_k$ is completely determined by 
  \[(S_{i_k}-S_{i_{k-1}},S_{i_k-1}-S_{i_{k-1}},\cdots,S_{i_{k+1}}-S_{i_{k-1}}) \overset{\mathcal{D}}{=}(S_0,S_{-1},\cdots,S_{\delta}). \]
  This completes the proof.
  \end{proof}
  
  \begin{remark}
    Consider the i.i.d. sequence \(\hat{X}=(\hat{X}_n)_{n \in \mathbb{Z}}\) defined in Eq.~(\ref{eq_conditioned_RW}).
  Since \(\mathbb{E}[\hat{X}_0]<0\) by Lemma~\ref{20230116135048}, the assumption of Proposition~\ref{20230128172431} is satisfied.
  This implies that the descendant tree of \(0\) in the record graph of \([\mathbb{Z},0,\hat{X}]\) is a Galton-Watson Tree with offspring distribution the same as that of \(\hat{X}_0+1\).
  This fact will be used in the last part of Step 2 of Theorem~\ref{r_graph_positive_drift_20230203174636}.
  \end{remark}

  \subsubsection{Negative mean}
   This part of the subsection focuses on the case when the mean of the increment \(X_0\) is negative, and it has three key objectives.

  The first objective is to establish that the distribution of the component of \(0\) in the record graph of \((\mathbb{Z},X)\) is the \(TGWT\) with the offspring distribution the same as that of \(X_0+1\).
  This is proved in Theorem~\ref{20230213191656}.

  The second objective is to demonstrate that every \(TGWT(\alpha)\) is obtained by typically re-rooting \(GW(\alpha)\), where \(\alpha\) is any offspring distribution that has mean \(m(\alpha) <1\).
  To prove this, consider \([\mathbf{T},0]\), the component of \(0\) in the record graph of \((\mathbb{Z},X)\), where the increments follow the common distribution \(X_0+1 \overset{\mathcal{D}}{=}\alpha\).
  Initially, it is shown that the conditioned Family Tree \([\mathbf{T},0]|\{F(0)=0\}\), conditioned on the event that the root \(0\) of \(\mathbf{T}\) has no parent, follows the distribution \(GW(\alpha)\) (see Lemma~\ref{lemma_conditioned_is_GW}).
  Then, it is shown that \([\mathbf{T},0]\) is obtained by typically re-rooting \([\mathbf{T},0]|\{F(0)=0\}\) (see Proposition~\ref{prop_tgwt_is_typical}).

  The third objective is to give a characterization of $TGWT$, which is detailed in Proposition~\ref{20230305162953}.

\begin{theorem}\label{20230213191656}
  Let \(X=(X_n)_{n \in \mathbb{Z}}\) be an i.i.d. sequence satisfying the conditions in Eq.~(\ref{eq_skip_free_condition}), \((\mathbb{Z},X)\) be its associated network, and $[\mathbf{T},o]$ be the connected component of $0$  (rooted at $o=0$) in the record graph of $(\mathbb{Z},X)$.
  If \(\mathbb{E}[X_0]<0\), then $[\mathbf{T},o]$ is the Typically rooted Galton-Watson Tree with offspring distribution $\pi$, where \(\pi\) is the distribution of \(X_0+1\).
\end{theorem}
\begin{proof}
  The theorem is proved in several steps.

  \emph{Step 1} shows that the descendant tree of the root $o$ is a Galton-Watson Tree with offspring distribution $\pi$.
  \emph{Step 2} shows the probability that $o$ has a parent is equal to the mean \(m(\pi)\).
  One way to compute it is by using Kemperman's formula (see \cite[Proposition 3.7]{bhattacharyaRandomWalkBrownian2021}).
  Alternatively, one could use the unimodularity of $[\mathbf{T},o]$.
  The latter is used in this step.
  \emph{Step 3} shows that the offspring distribution of $F(o)$ conditioned on the event that $o$ has a parent is \(\hat{\pi}\) (the size-biased distribution of \(\pi\)).
 \emph{Step 4} shows that the distribution of the order of $o$ is uniform among its siblings conditioned on the event that $o$ has a parent.
\emph{Step 5} shows that the descendant trees of the siblings of $o$ conditioned on the event that $o$ has a parent are independent Galton-Watson Trees with offspring distribution $\pi$.
  
  The final step (\emph{Step 6}) shows that the offspring distribution of $F^2(o)$ conditioned on the event that $o$ has an ancestor of order \(2\) is the same as the offspring distribution of $F(o)$ conditioned on the event that $o$ has a parent.
  Similarly, it is shown that the descendant trees of siblings of $F(o)$ and its order among its siblings conditioned on the event that $o$ has an ancestor of order \(2\) have the same distribution as those of $o$ conditioned on the event that $o$ has a parent.

  \noindent {\em Step 1:} 
  Since the random walk $(S_{-n})_{n \geq 0}$ has positive drift, $-\infty<L_X(0)$.
  By Proposition~\ref{20230128172431}, the descendant tree of $o$ is GW($\pi$).

  \noindent {\em Step 2:}
  Since $[\mathbf{T},o]$ is the connected component of $0$ in the $R$-graph of a unimodular network $[\mathbb{Z},0,X]$, it is also unimodular by Lemma~\ref{lemma_f_graph_unimodular}.
  Therefore, 
  \begin{align*}
      \mathbb{P}[o \text{ has a parent}] &= \mathbb{E}\left[\sum_{u \in V(\mathbf{T})}\mathbf{1}\{F(o)=u\} \right]=\mathbb{E}\left[\sum_{u \in V(\mathbf{T})}\mathbf{1}\{F(u)=o\} \right]\\
      &= \mathbb{E}[d_1(o)]=m(\pi)=\mathbb{E}[X_{0}]+1,
  \end{align*}
  where \(m(\pi)\) is the mean of \(\pi\).
In particular, 
\[\mathbb{P}[S_n<0\quad \forall n>0] = \mathbb{P}[o \text{ does not have parent}] = -\mathbb{E}[X_{0}].\]

\noindent {\em Step 3:} 
Note that $o$ has a parent if and only if $\tau<\infty$ (see the paragraph above Lemma~\ref{20230119141633} for the definition of \(\tau\)).
So, $\mathbb{P}[\tau< \infty] = m(\pi)$.
Since $(S_{-n})_{n \geq 0}$ drifts to $+\infty$, part (1) of Lemma~\ref{lemma_offspring_count} implies the equality of the two events $\{d_1(F(o))=n,\tau<\infty\} = \{X_{\tau-1}=n-1,\tau< \infty\}$, for any $n>0$.
Therefore, for any $n>0$,
\begin{align}
  \mathbb{P}[d_1(F(o))=n,\tau<\infty] &= \sum_{j=0}^{n-1}\mathbb{P}[\tau<\infty,X_{\tau-1}=n-1,S_{\tau}=j] \notag\\
  &= \sum_{j=0}^{n-1}\mathbb{P}[S_1=n-1]c^{j-n+1} =n\pi(n). \label{eq_1_2.25} 
\end{align} 
By Lemma~\ref{20230119141633}, Eq.~(\ref{eq_1_2.25}) follows from the previous equation.
The last equation follows since $c=\mathbb{P}[S_n=-1 \text{ for some } n>0]=1$ as the random walk $(S_n)_{n \geq 0}$ drifts to $-\infty$.
Therefore,
\[\mathbb{P}[d_1(F(o))=n|\tau< \infty] = \frac{n\pi(n)}{m(\pi)}= \hat{\pi}(n),\]
for all $n>0$.

\noindent {\em Step 4:}
Let $c_k(u)$ denote $k$-th child of $u$ for any vertex $u$ and positive integer $k$.
Note that the equality of the two events $\{\tau< \infty,c_k(F(o))=o,d_1(o)=n\}= \{\tau<\infty, X_{\tau-1}=n-1,S_{\tau}=n-k\}$ for any $0<k\leq n$ follows from part (3) of Lemma~\ref{lemma_offspring_count} (with \(i_m\) as \(0\) and \(i\) as \(F(0)\)).
Therefore, using Lemma~\ref{20230119141633} and $c=1$, one gets that for any $0<k \leq n$,
\[\mathbb{P}[\tau< \infty,c_k(F(o))=o,d_1(o)=n] = \mathbb{P}[X_0 = n-1].\]
Thus, for any $0<k \leq n$
\[\mathbb{P}[c_k(F(o))=o|d_1(o)=n,\tau<\infty] = \frac{\mathbb{P}[X_0 = n-1]}{n \mathbb{P}[X_0 = n-1]}=\frac{1}{n},\]
which is the uniform order among the siblings of $o$.

\noindent {\em Step 5:}
On the event $\{\tau<\infty, d_1(F(o))=n, c_k(F(o))=o\}$, let $i_n<i_{n-1}<\cdots<i_1$ be the positions of the children of $F(o)$.
Then, conditioned on $\{\tau<\infty, d_1(F(o))=n, c_k(F(o))=o\}$, for each $1\leq j \leq n-1$, the part of the random walk $(0, S_{i_j-1}-S_{i_j},S_{i_j-2}-S_{i_j},\cdots, S_{i_{j+1}}-S_{i_j})$ is an excursion set by Lemma~\ref{lemma_offspring_count} (part 3).
Further, these excursion sets are independent of one another because the times $i_j$ (for $1 \leq j \leq n$) are stopping times.
Indeed, for all $1 \leq j \leq n-1$, $S_{i_{j+1}}-S_{i_j}=1$ and $S_l<S_{i_{j+1}}$ for all $i_{j+1}<l \leq i_j$.
Therefore, for each $1 \leq j \leq n-1$, $(0, S_{i_j-1}-S_{i_j},S_{i_j-2}-S_{i_j},\cdots, S_{i_{j+1}}-S_{i_j})$ is the skip-free to the right random walk $(S_{-n})_{n \geq 0}$ conditioned on $\eta_1< \infty$, and stopped at $\eta_{1}$, where $\eta_1 = \min\{n>0: S_{-n}=1\}$.
But $(S_{-n})_{n \geq 0}$ drifts to $+\infty$.
So, \(\eta_1<\infty\) a.s..
Hence, $(0, S_{i_j-1}-S_{i_j},S_{i_j-2}-S_{i_j},\cdots, S_{i_{j-1}}-S_{i_j})$ has the same distribution as $(0,S_{-1},\cdots,S_{\eta_1})$.
Therefore, by Proposition~\ref{20230128172431}, for each $1 \leq j \leq n$, the descendant trees are independent $GW(\pi)$.

\noindent {\em Step 6:}
Observe that the distribution of $(S_n-S_{\tau})_{n \geq \tau}$ conditioned on $\tau< \infty$ is the same as that of $(S_n)_{n \geq 0}$ by the strong Markov property.
Therefore, 
\[\mathbb{P}[F(o) \text{ has a parent }|\{o \text{ has a parent}\}] = \mathbb{P}[o \text{ has a parent}] = m(\pi).\]
Since $(S_{-n})_{n \geq 0}$ drifts to $+ \infty$ a.s., it reaches $S_{R^2(0)+1}$ a.s..
Therefore, by part (1) of Lemma~\ref{lemma_offspring_count}, the distribution of $d_1(F^2(o))$ conditioned on the event $\{o \text{ has a grandparent}\}$ is equal to the distribution of $X_{R^2(0)-1}+1$ conditioned on $\{R^2(0)>R(0)\}$ (i.e., $o$ has a grandparent) which has the same distribution as that of $X_{R(0)-1}+1$ conditioned on $R(0)>0$ (i.e., \(X_{\tau-1}\) conditioned on \(\tau< \infty\)).
In particular, the distribution of $[D(F^2(o))\backslash D(F(o)),o]$ conditioned on $\{o \text{ has a grandparent}\}$ is the same as that of $[D(F(o))\backslash D(o),o]$ conditioned on $\{o \text{ has a parent}\}$.
By induction, it follows that $[D(F^n(o))\backslash D(F^{n-1}(o)),o]$ conditioned on $\{o \text{ has }n- \text{th ancestor}\}$ has the same distribution as that of $[D(F(o))\backslash D(o),o]$ conditioned on $\{o \text{ has a parent}\}$.
\end{proof}

It is now shown that the Typically rooted Galton-Watson Tree is the typical re-rooting (as in Def.~\ref{def_typical_reroot_joining}) of a Galton-Watson Tree.
Let \([\mathbf{T},o]\) be the Family Tree as in Theorem~\ref{20230213191656}, and let \(\{F(o)=o\}\) denote the event that \(o\) does not have a parent.
Let 
\begin{equation} \label{eq_cond_tree_no_parent}
    [\mathbf{T}',\mathbf{o}'] \overset{\mathcal{D}}{:=} [\mathbf{T},o]|\{F(o)=o\},
\end{equation}
be the Family Tree obtained from \([\mathbf{T},o]\) by conditioning on the event that \(o\) does not have a parent.
The first step consists in showing that the distribution of \([\mathbf{T}',\mathbf{o}']\) is \(GW(\pi)\).
Since \(m(\pi)=\mathbb{E}[X_0]+1<1\), the size of \(GW(\pi)\) is finite in mean.
Thus, the typical re-rooting of \(GW(\pi)\) is defined.
The next step consists in proving that \([\mathbf{T},o]\) is the typical re-rooting of \([\mathbf{T}',\mathbf{o}']\).
By varying the distribution of \(X_0\), it is shown that every Typically rooted Galton-Watson Tree is the typical re-rooting of a Galton-Watson Tree.

The following lemma uses the notation of Theorem~\ref{20230213191656}.

\begin{lemma}\label{lemma_conditioned_is_GW}
    Let \([\mathbf{T}',\mathbf{o}']\) be the Family Tree defined in Eq.~\ref{eq_cond_tree_no_parent}, with \([\mathbf{T},0]\) as the component of \(0\) in the record graph of \((\mathbb{Z},X)\), where \(X\) is the sequence as in Theorem~\ref{20230213191656}.
    Then, the distribution of \([\mathbf{T}',\mathbf{o}']\) is \(GW(\pi)\).
\end{lemma}
\begin{proof}
    Note that the events \(\{F(o)=o\}\) and \(\{S_n<0 \, \forall n>0\}\) are one and the same because the latter event is the same as \(\{R(0)=0\}\).
    The random variables \(\{S_n:n \leq 0\}\) are independent of the event \(\{S_n<0 \, \forall n>0\}\), and the descendant tree of \(0\) is a measurable function of \(\{S_n:n \leq 0\}\).
    This implies that the distribution of the descendant tree \([D(o),o]\) of \(o\) conditioned on the event \(\{F(o)=o\}\) is the same as that of \([D(o),o]\) (the unconditioned descendant tree of \(o\)).
    Therefore,
    \begin{equation*}
        [\mathbf{T}',\mathbf{o}']= [D(o),o]|\{F(o)=o\}  \overset{\mathcal{D}}{=} [D(o),o].
    \end{equation*}

    Since \(\mathbb{E}[X_0]<0\), by Proposition~\ref{20230128172431}, the distribution of \([D(o),o]\) is \(GW(\pi)\).
\end{proof}

\begin{proposition}\label{prop_tgwt_is_typical}
Let \(\alpha\) be a probability measure on \(\{0,1,2,\cdots\}\) such that its mean \(m(\alpha)<1\).
Then, the \(TGWT(\alpha)\) is the typical re-rooting of \(GW(\alpha)\).
\end{proposition}
\begin{proof}
    Take an i.i.d. sequence of random variables \(X' = (X_n')_{n \in \mathbb{Z}}\) such that their common distribution is given by \(\mathbb{P}[X'_0=j]=\alpha(j+1)\), for all \(j \in \{-1,0,1,2,\ldots\}\).
    Let \([\mathbf{T}, \mathbf{o}]\) be the component of \(0\) in the record graph of \((\mathbb{Z},X')\) and let \([\mathbf{T}',\mathbf{o}']\) be the Family Tree obtained from \([\mathbf{T},\mathbf{o}]\) by conditioning on the event that \(\mathbf{o}\) does not have a parent.
    By Theorem~\ref{20230213191656}, the distribution of \([\mathbf{T}, \mathbf{o}]\) is \(TGWT(\alpha)\), which is unimodular; and by Lemma~\ref{lemma_conditioned_is_GW}, the distribution of \([\mathbf{T}',\mathbf{o}']\) is \(GW(\alpha)\), which has finite mean size.
    It is shown below that \([\mathbf{T}, \mathbf{o}]\) is the typical re-rooting of \([\mathbf{T}',\mathbf{o}']\).

    For any measurable subset \(A\) of \(\mathcal{T}_*\), consider the function \(g_A:\mathcal{T}_{**} \to \mathbb{R}_{\geq 0}\) defined by \(g_A([T,u,v]) = \mathbf{1}_A([T,u]) \mathbf{1}\{F(v)=v\}\).
    Then, for any measurable set \(A\), one has,
    \begin{equation*}
        \mathbb{E}\left[\sum_{u \in V(\mathbf{T})}g_A([\mathbf{T},\mathbf{o},u])\right]= \mathbb{E}[\mathbf{1}_A([\mathbf{T},\mathbf{o}])] = \mathbb{P}[[\mathbf{T},\mathbf{o}] \in A],
    \end{equation*}
since there is exactly one vertex that does not have a parent.
Similarly,
\begin{align*}
    \mathbb{E}\left[\sum_{u \in V(\mathbf{T})}g_A([\mathbf{T},u,\mathbf{o}])\right] &= \mathbb{E}\left[\mathbf{1}\{F(\mathbf{o})=\mathbf{o}\} \sum_{u \in V(\mathbf{T})}\mathbf{1}_A([\mathbf{T},u])\right]\\
    &= \mathbb{E}\left[ \sum_{u \in V(\mathbf{T}')}\mathbf{1}_A([\mathbf{T}',u])\right] \mathbb{P}[F(\mathbf{o})=\mathbf{o}].
\end{align*}
The last step follows since \([\mathbf{T}',\mathbf{o}']\) is the Family Tree obtained from \([\mathbf{T},\mathbf{o}]\) by conditioning  on the event \(\{F(\mathbf{o})=\mathbf{o}\}\).

By the unimodularity of \([\mathbf{T},\mathbf{o}]\), one gets 
\begin{equation*}
    \mathbb{P}[[\mathbf{T},\mathbf{o}] \in A] = \mathbb{E}\left[ \sum_{u \in V(\mathbf{T}')}\mathbf{1}_A([\mathbf{T}',u])\right] \mathbb{P}[F(\mathbf{o})=\mathbf{o}].
\end{equation*}

By taking \(A=\mathcal{T}_*\), one has \( \mathbb{P}[F(\mathbf{o})=\mathbf{o}] \mathbb{E}[\#V(\mathbf{T})]=1\).
Thus, for any measurable set \(A\), 
\[\mathbb{P}[[\mathbf{T},\mathbf{o}] \in A] = \frac{1}{\mathbb{E}[\#V(\mathbf{T})]}\mathbb{E}\left[ \sum_{u \in V(\mathbf{T}')}\mathbf{1}_A([\mathbf{T}',u])\right].\]
\end{proof}

The next proposition gives a characterizing condition for a Typically rooted Galton-Watson Tree.
The proof of this proposition depends on the following lemma which says that every unimodular tree is characterized by the descendant tree of its root.

\begin{lemma}\label{20230303153138}
    Let $[\mathbf{T},\mathbf{o}]$ be a unimodular Family Tree.
    Then, $[\mathbf{T},\mathbf{o}]$ is completely characterized by $[D(\mathbf{o}),\mathbf{o}]$.
\end{lemma}
\begin{proof}
    Observe that $[D(F^n(\mathbf{o})),\mathbf{o}]$ converges weakly to $[\mathbf{T},\mathbf{o}]$ as $n \to \infty$.
    Indeed, $[D(F^n(\mathbf{o})),\mathbf{o}]_r \overset{\mathcal{D}}{=} [\mathbf{T},\mathbf{o}]_r$ for all $n>r$.

   For any measurable set $A$,
    \begin{align*}
        \mathbb{P}[[D(F^n(\mathbf{o})),\mathbf{o}] \in A] &= \mathbb{E}\left[\sum_{u \in V(\mathbf{T})} \mathbf{1}_A [D(u),\mathbf{o}] \mathbf{1}\{u = F^n(\mathbf{o})\} \right].
    \end{align*}
    By unimodularity, the equation on the right-hand side is equal to
       \[ \mathbb{E}\left[\sum_{u \in V(\mathbf{T})} \mathbf{1}_A [D(\mathbf{o}),u] \mathbf{1}\{\mathbf{o} = F^n(u)\} \right].\]
      Therefore,
     \[ \mathbb{P}[[D(F^n(\mathbf{o})),\mathbf{o}] \in A]= \mathbb{E}\left[\sum_{u \in D_n(\mathbf{o})} \mathbf{1}_A[D(\mathbf{o}),u]\right].\]
\end{proof}

Lemma~\ref{20230303153138} is closely related to \cite[Proposition 10]{aldousAsymptoticFringeDistributions1991}.
In the language of this work, the latter proposition proves that if the descendant tree \([D(\mathbf{o}),\mathbf{o}]\) of a rooted Family Tree \([\mathbf{T},\mathbf{o}]\) is a fringe distribution, then \([\mathbf{T},\mathbf{o}]\) is completely described by \([D(\mathbf{o}),\mathbf{o}]\).
The connection between \cite[Proposition 10]{aldousAsymptoticFringeDistributions1991} and Lemma~\ref{20230303153138} follows from the observation that for any unimodular tree \([\mathbf{T},\mathbf{o}]\) of class \(\mathcal{I}/\mathcal{I}\), its descendant tree \([D(\mathbf{o}),\mathbf{o}]\) is a fringe distribution.
However, Lemma~\ref{20230303153138} is also applicable to the unimodular Family trees of class \(\mathcal{I}/\mathcal{F}\).
For more details on this connection, see \cite[Bibliographical Comments (Section 6.3)]{baccelliEternalFamilyTrees2018a}.
 
The following proposition and its proof are analogous to the characterization of Eternal Galton-Watson Tree of \cite{baccelliEternalFamilyTrees2018a}.
However, the following proposition differs from the latter as its statement is about finite Family Trees, whereas the latter is about EFTs.

In the following proposition, only the random objects are denoted in bold letters.

For a (deterministic) Family Tree $[T,o]$, the non-descendant tree of the root $o$ is the Family Tree $[D^c(o),o]$, where $D^c(o)$ is the subtree induced by $(V(T)\backslash D(o))\cup \{o\}$.
For any vertex $u$ of $T$, let $c_j(u)$ denote the $j$-th child of $u$.

 \begin{proposition}[Characterization of $TGWT$]\label{20230305162953}
    A random \emph{finite} Family Tree $[\mathbf{T},\mathbf{o}]$ is a Typically rooted Galton-Watson Tree ($TGWT$) if and only if
    \begin{enumerate}
        \item it is unimodular, and
        \item the number of children of the root $d_1(\mathbf{o})$ is independent of the non-descendant tree of the root $\mathbf{o}$.
    \end{enumerate} 
 \end{proposition}
\begin{proof}
    Let $[\mathbf{T},\mathbf{o}]$ be a random Family Tree whose distribution is $TGWT(\pi)$, where $\pi$ is a probability measure on $\{0,1,2,3,\cdots\}$ with mean $m(\pi)<1$.
    Consider an i.i.d. sequence $X=(X_n)_{n \in \mathbb{Z}}$ of random variables whose common distribution is given by $\mathbb{P}[X_0=n] = \pi(n+1)$ for all $n \in \{-1,0,1,2,\cdots\}$, and let $[\mathbb{Z},0,X]$ be its network.
    Since $\mathbb{E}[X_0]<0$, by Theorem~\ref{20230213191656}, the connected component of $0$ in the $R$-graph of $[\mathbb{Z},0,X]$ is $TGWT(\pi)$.
    Therefore, $[\mathbf{T},\mathbf{o}]$ is unimodular by Lemma~\ref{lemma_f_graph_unimodular}.
    The second condition is satisfied by $TGWT$, which follows from its construction.

    It is now shown that if a random finite Family Tree $[\mathbf{T},\mathbf{o}]$ satisfies the above conditions $1$ and $2$, then it is a $TGWT$.

    Let $E$ denote the event that $o$ has a parent.
    For a positive integer $k$, let $A = (A';A_1,A_2,\cdots,A_k)$ be an event of the form
    \begin{equation}\label{eq_event_comp_child}
        d_1(o)=k, D^c(o) \in A', D(c_1(o)) \in A_1,\cdots, D(c_k(o)) \in A_k.
    \end{equation}
    By Lemma~\ref{20230303153138}, any unimodular Family Tree is characterized by the descendant tree of its root.
    So, it is sufficient to prove that
    \begin{equation} \label{eq:1_2.35}
        \mathbb{P}[A] = \mathbb{P}[d_1(\mathbf{o})=k] \mathbb{P}[D^c(\mathbf{o}) \in A'] \left(\prod_{i=1}^{k}\mathbb{P}[D(\mathbf{o}) \in A_i]\right),
    \end{equation}
    for any such event \(A\) of the form given by Eq (\ref{eq_event_comp_child}).
    Assume further that the events $A_1,A_2,\cdots,A_k$ depend only on the descendant tree up to the \(n\)-th generation of $o$, which is the union of \(D_1(o)\), \(D_2(o)\), \(\cdots\), \(D_n(o)\).
    It suffices to show Eq.~(\ref{eq:1_2.35}) for such events because of the local topology.

    The result is proved by induction on $n$.
    For $n=0$, one has $\mathbb{P}[A]= \mathbb{P}[d_1(\mathbf{o})=k]\mathbb{P}[D^c(\mathbf{o}) \in A']$ by condition $2$ of the hypothesis.
    So, assume that the result is true for $n-1$.   
    Let $n \geq 1$ and $1 \leq j \leq k$.
    For any Family Tree $[T,o]$, let $h_j([T,o]):= \mathbf{1}_A[T,F(o)] \mathbf{1}\{o = c_j(F(o))\}$.
    Then, for any $1 \leq j \leq k$,
  \begin{equation} \label{eq:2_2.35}
      \mathbf{1}_A[T,o] = \sum_{u \in D_1(o)} \mathbf{1}_A[T,F(u)] \mathbf{1}\{u=c_j(o)\} = \sum_{u \in D_1(o)}h_j([T,u]).
  \end{equation}
  Note that, since \(k \geq 1\), the event \(\{[T,o] \in A\}\) is a subset of the event \(\{D_1(o)\not = \emptyset\}\)
  Applying Eq.~(\ref{eq:2_2.35}) for $j=1$, one gets
  \begin{align}
    \mathbb{P}[A] &= \mathbb{E}\left[\sum_{u \in V(\mathbf{T})} \mathbf{1}_A[\mathbf{T},F(u)] \mathbf{1}\{u = c_1(\mathbf{o})\} \mathbf{1}\{u \in D_1(\mathbf{o})\}\right] \notag\\
    &= \mathbb{E}\left[\sum_{u \in V(\mathbf{T})}\mathbf{1}_A[\mathbf{T},F(\mathbf{o})] \mathbf{1}\{\mathbf{o} = c_1(u)\} \mathbf{1}\{\mathbf{o} \in D_1(u)\}\right] \label{eq:3_2.35}\\
    &= \mathbb{E}\left[\mathbf{1}_A[\mathbf{T},F(\mathbf{o})] \mathbf{1}\{ \mathbf{o} \in E\}  \mathbf{1}\{\mathbf{o} = c_1(F(\mathbf{o}))\} \right] \notag\\
    &= \mathbb{P}[\mathbf{o} \in E, \mathbf{o} = c_1(F(\mathbf{o})),[\mathbf{T},F(\mathbf{o})] \in A ]\label{eq:4_2.35}\\
    &= \mathbb{P}[D(\mathbf{o}) \in A_1, \mathbf{o} \in E, \mathbf{o} = c_1(F(\mathbf{o})),[\mathbf{T},F(\mathbf{o})] \in A(A';\mathcal{T}_*,A_2,\cdots,A_k) ] \notag.
  \end{align}
  In the above, Eq.~(\ref{eq:3_2.35}) follows from unimodularity.
  Note that $D(o) \in A_1$ depends on one generation less that of $[T,F(o)] \in A$.
  Therefore, by induction,
  \begin{align*}
    \mathbb{P}[A]
    &= \mathbb{P}[D(\mathbf{o}) \in A_1] \mathbb{P}[\mathbf{o} \in E, \mathbf{o} = c_1(F(\mathbf{o})),[\mathbf{T},F(\mathbf{o})] \in A(A';\mathcal{T}_*,A_2,\cdots,A_k)]\\
    &= \mathbb{P}[D(\mathbf{o}) \in A_1] \mathbb{P}[A(A';\mathcal{T}_*,A_2,\cdots,A_k)].
  \end{align*}
  The last equation is obtained by Eq.~(\ref{eq:4_2.35}).

  Applying Eq.(\ref{eq:2_2.35}) with $j=2$ to $A(A';\mathcal{T}_*,A_2,\cdots,A_k)$, one gets
  \begin{align*}
    \mathbb{P}[A&(A';\mathcal{T}_*,A_2,\cdots,A_k)]\\
     &= \mathbb{P}[D(\mathbf{o}) \in A_2]\mathbb{P}[\mathbf{o} \in E, \mathbf{o} = c_1(F(\mathbf{o})), [\mathbf{T},F(\mathbf{o})] \in A(A';\mathcal{T}_*,\mathcal{T}_*,\cdots,A_k)]\\
    &=\mathbb{P}[D(\mathbf{o}) \in A_2]\mathbb{P}[A(A';\mathcal{T}_*,\mathcal{T}_*,\cdots,A_k)].
  \end{align*}
 By iterating the same procedure for $j=3,4,\cdots$, one gets
 \begin{align*}
    \mathbb{P}[A] &= \left(\prod_{i=1}^{k} \mathbb{P}[D(\mathbf{o}) \in A_i]\right) \mathbb{P}[d_1(\mathbf{o})=k, D^c(\mathbf{o}) \in A']\\
    &= \left(\prod_{i=1}^{k} \mathbb{P}[D(\mathbf{o}) \in A_i]\right) \mathbb{P}[d_1(\mathbf{o})=k] \mathbb{P}[ D^c(\mathbf{o}) \in A'].
 \end{align*}
Thus, Eq.~(\ref{eq:1_2.35}) holds, which completes the proof. 
 \end{proof}

 \subsubsection{Zero mean}\label{subsec:r_graph_egwt}
 \begin{theorem} \label{theorem:R-graph_egwt}
    Let $X=(X_n)_{n \in \mathbb{Z}}$ be a sequence  of random variables satisfying the conditions of Eq.~\ref{eq_skip_free_condition}, $(\mathbb{Z},X)$ be its associated network, and $[\mathbf{T},0]$ be the connected component of $0$ in the record graph of $(\mathbb{Z},X)$.
    If $\mathbb{E}[X_0]=0$, then $[\mathbf{T},0]$ is the ordered Eternal Galton-Watson Tree $EGWT(\pi)$ with offspring distribution $\pi \stackrel{\mathcal{D}}{=}X_{0}+1$.
  \end{theorem}
  \begin{proof}
    It is first shown that \([\mathbf{T},0]\) satisfies the condition of Theorem~\ref{theorem_characterization_egwt}.
    Observe that $L_X(0)$ is a stopping time for the random walk with increments $(X_{-n})_{n\geq 1}$ and \(-\infty<L_X(0)\) a.s..
    By the strong Markov property, the random variables $Y=(X_{-n})_{n > -L_X(0)}$ are independent of  $Z=(X_{-n})_{n=1}^{n=L_X(0)}$.
    Let $(\mathbf{T}',0)$ denote the non-descendant tree of \(o\), i.e., the tree generated by $(\mathbf{T}\backslash D(0))\cup \{0\}$.
    By Lemma~\ref{lemma:descendants}, the subtree $(D(0),0)$ is a function of $Z$ and the subtree $(\mathbf{T}',0)$ is a function of $Y \cup (X_n)_{n \geq 0}$.
    So, the unordered tree $(D(0),0)$ is independent of the unordered tree $(\mathbf{T}',0)$.
    The second part of Lemma~\ref{lemma_offspring_count} implies that, conditionally on the event that $0$ has children, the order of a child is independent of $Y \cup (X_n)_{n \geq 0}$, whereas the order of any vertex in $(\mathbf{T}',0)$ is a function of $Y \cup (X_n)_{n \geq 0}$.
    Thus, the ordered subtree $(D(0),0)$ is independent of the ordered subtree $(\mathbf{T}',0)$.
    So, \([\mathbf{T},0]\) satisfies the sufficient condition of Theorem~\ref{theorem_characterization_egwt}.
    
    The first statement of Lemma~\ref{lemma_offspring_count} implies that $d_1(0) = X_{-1}+1 \overset{\mathcal{D}}{=}X_0+1$, and thus completes the proof.
  \end{proof}

  \subsubsection{Positive mean}

  Let $X=(X_n)_{n \in \mathbb{Z}}$ be an i.i.d. sequence that satisfies the conditions of Eq.~(\ref{eq_skip_free_condition}) with $0<\mathbb{E}[X_{0}]\leq \infty$, and $(\mathbb{Z},X)$ be its associated network.

Recall the following notations of Subsection~\ref{subsec_properties_skip_free}.
Let $\eta_{-1}$ be the hitting time at \(-1\) defined as in Eq.~(\ref{eq_eta_j}) for \(j=-1\), and $c = \mathbb{P}[\eta_{-1}< \infty]$ be the probability that the hitting time at \(-1\) is finite.
Let $[\mathbf{T},\mathbf{o}]$ be the connected component of \(0\) (with root $0 = \mathbf{o}$) in the record graph of $(\mathbb{Z},X)$.

Let $\tilde{\pi},\bar{\pi}$ be the probability distributions on \(\{0,1,2,\cdots\}\) given by:
\begin{align} \label{eq:pi_defn}
 \tilde{\pi}(k) &= c^{k-1} \mathbb{P}[X_{0}=k-1],\\ \nonumber
 \bar{\pi}(k)&= \mathbb{P}[X_{0}\geq k] c^{k}.
\end{align}

\begin{remark}
  The fact that \(\bar{\pi}\) is a probability distribution follows from Eq.~(\ref{eq_pi_bar_distribution}), in the proof of Theorem~\ref{r_graph_positive_drift_20230203174636}.
  Similarly, it follows from Eq.~(\ref{eq_pi_tilde_distribution}) that \(\tilde{\pi}\) is a probability distribution.
\end{remark}

Recall the construction bi-variate Eternal Kesten Tree (see Subsection~\ref{subsec_bi_variate_ekt}).
Let $[\mathbf{T}_0,\mathbf{o}_0]$ be a Family Tree whose distribution is given by the following: the offspring distribution of $\mathbf{o}_0$ is \(\bar{\pi}\).
The descendant trees of the children of $\mathbf{o}_0$ are independent Galton-Watson Trees with offspring distribution $\tilde{\pi}$, and they are independent of \(d_1(\mathbf{o}_0)\).
Let \(([\mathbf{T}_i,\mathbf{o}_i])_{i \in \mathbb{Z}}\) be an i.i.d. sequence of Family Trees.
Recall that the distribution of the Family Tree obtained from the typically re-rooted joining of \(([\mathbf{T}_i,\mathbf{o}_i])_{i \in \mathbb{Z}}\) is the unimodularised \(EKT(\bar{\pi},\tilde{\pi})\).
Since the mean \(m(\tilde{\pi})<1\) (by Lemma~\ref{20230116135048}), the typically re-rooted joining operation is well-defined for this sequence.
Recall the ECS order on the joining of \(([\mathbf{T}_i,\mathbf{o}_i])_{i \in \mathbb{Z}}\) (see Subsection~\ref{subsec_bi_variate_ekt}).

\begin{theorem}\label{r_graph_positive_drift_20230203174636}
   If  \(0<\mathbb{E}[X_{0}]\leq \infty\), then the distribution of $[\mathbf{T}, \mathbf{o}]$ is the unimodularised ECS ordered \(EKT(\bar{\pi},\tilde{\pi})\).
\end{theorem}

\begin{proof}
   Since the random walk has positive drift, by  Theorem~\ref{thm_phase_transition_stationary}, the unimodular EFT $[\mathbf{T},\mathbf{o}]$ is of class $\mathcal{I}/\mathcal{F}$.
   Let $\mathbf{o} \in \neswarrow$ denote the event that $\mathbf{o}$ belongs to the bi-infinite path of $\mathbf{T}$.
   Let $[\mathbf{T}',\mathbf{o}']$ be the conditioned Family Tree obtained from $[\mathbf{T},\mathbf{o}]$ by conditioning on the event $\mathbf{o} \in \neswarrow$, denoted as \([\mathbf{T},\mathbf{o}]|\mathbf{o} \in \neswarrow\).
   By Theorem~\ref{thm_eft_I_f_joining}, $[\mathbf{T}',\mathbf{o}']$ is the joining of some stationary sequence of Family Trees $([\mathbf{T}'_i,\mathbf{o}'_i])_{i \in \mathbb{Z}}$.
   The proof is complete once it is shown that this sequence is i.i.d. (shown in \emph{Step 1}), that the following distributional equality holds: $[\mathbf{T}'_1,\mathbf{o}'_1] \overset{\mathcal{D}}{=} [\mathbf{T}_0,\mathbf{o}_0]$ (shown in \emph{Step 2}), and that \([\mathbf{T}',\mathbf{o}']\) has ECS order --- \(\mathbf{o}_i'\) is the smallest child of \(\mathbf{o}_{i+1}'\) in \(\mathbf{T}'\), for all \(i \in \mathbb{Z}\) (shown in \emph{Step 3}).
   
   \noindent \emph{Step 1, $([\mathbf{T}'_i,\mathbf{o}'_i])_{i \in \mathbb{Z}}$ is an i.i.d. sequence:}
   Since the increments have positive mean, $S_n \to +\infty$ as $n \to + \infty$ and $S_n \to - \infty$ as $n \to -\infty$.
   By Lemma~\ref{lemma:descendants}, $\mathbf{o} \in \neswarrow$ if and only if \(L_X(0) = \infty\), and the latter condition holds  if and only if $S_n \leq 0, \forall n \leq -1$.

   For any $u \in \mathbb{Z}$, let $\mathfrak{S}(u)$ denote the (random) subtree of the descendants of $u$ (including \(u\)) in $\mathbf{T}$. 
   The distribution of $([\mathbf{T}'_i,\mathbf{o}'_i])_{i \geq 1}$ is the same as that of 
   \[([\mathfrak{S}(F^n(\mathbf{o}))\backslash \mathfrak{S}(F^{n-1}(\mathbf{o})),F^n(\mathbf{o})])_{n \geq 1}|\mathbf{o} \in \neswarrow\]
   (distribution obtained by conditioning on \(\mathbf{o} \in \neswarrow\)), where \(F\) is the parent vertex-shift on \(\mathbf{T}\) with the convention that $F^0(\mathbf{o})= \mathbf{o}$ (and \(F^n(\mathbf{o})\) is the ancestor of order \(n\) of \(\mathbf{o}\)).
   On the event $\mathbf{o} \in \neswarrow$, each subtree $\mathfrak{S}(F^n(\mathbf{o}))\backslash \mathfrak{S}(F^{n-1}(\mathbf{o}))$ is completely determined by the part of the random walk
   \[(S_{R^{n-1}(0)+1}-S_{R^{n-1}(0)}, \ldots, S_{R^{n}(0)}-S_{R^{n-1}(0)}),\]
   for each $n \geq 1$ (see Figure~\ref{fig_I_f_trajectory} for an illustration).
The distribution of the sequence 
   \[((S_{R^{n-1}(0)+1}-S_{R^{n-1}(0)}, \ldots, S_{R^{n}(0)}-S_{R^{n-1}(0)}))_{n \geq 1}\]conditioned on $\{S_n \leq 0: n \leq -1\}$ is the same as that of the unconditioned sequence (the same sequence without the condition) because the sequence is independent of the event on which it is conditioned.
   Moreover, it is an i.i.d. sequence by the strong Markov property of the random walk.
   This implies that $([\mathbf{T}'_i,\mathbf{o}'_i])_{i \geq 1}$ is an i.i.d. sequence.
   \begin{figure}[h]
    \begin{center}
          \includegraphics[scale=0.7]{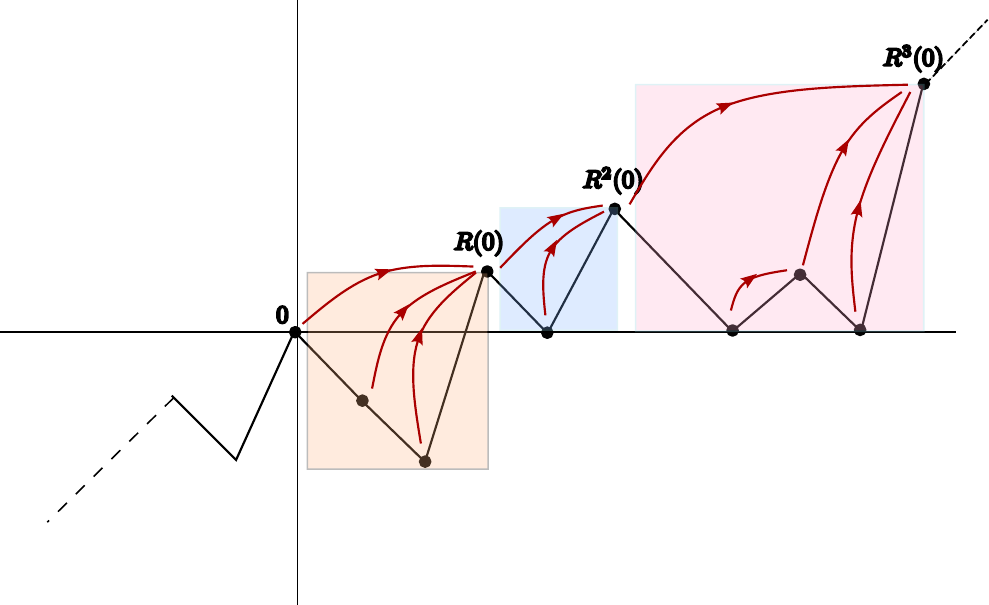}    
    \end{center}
        \caption{Illustration of the subtrees of ancestors of \(0\) on a trajectory conditioned on the event that \(S_n \leq 0\) for all \(n \leq 0\) when \(\mathbb{E}[X_0]>0\). Trajectories in the boxes correspond to the subtrees \(\mathbf{T}_1,\mathbf{T}_2, \mathbf{T}_3, \cdots\) (see Step 1 of Theorem~\ref{r_graph_positive_drift_20230203174636}).}  
        \label{fig_I_f_trajectory}
    \end{figure}

   Further, the stationarity of $([\mathbf{T}_i',\mathbf{o}_i'])_{i \in \mathbb{Z}}$ implies that $([\mathbf{T}_i',\mathbf{o}_i'])_{i \in \mathbb{Z}}$  is an i.i.d. sequence (the index set here is \(\mathbb{Z}\), whereas in the previous sentence it is \(i \geq 1\)).

   \noindent \emph{Step 2, \([\mathbf{T}'_1, \mathbf{o}'_1] \overset{\mathcal{D}}{=} [\mathbf{T}_0,\mathbf{o}_0]\):}
   It is first shown that the offspring distribution of $\mathbf{o}'_1$ is \(\bar{\pi}\).
   Let $\tau = \inf\{n>0: S_n \geq 0\}$ and \(\infty\) if \(S_n<0,\, \forall n>0\), and $X_{\tau-1}$, $S_{\tau}$ be the increment and location of the random walk at this time.
   Since the increments have positive mean, $\tau< \infty$ a.s., and therefore \(R(0)=\tau\) a.s. (if \(\tau=\infty\) then \(R(0)\) is defined to be \(0\)). 

   The equality of the following events holds:
   \[\{d_1(F(\mathbf{o}))=1\} = \{\tau = 1\} = \{X_{\tau-1} = S_{\tau}\} = \{S_1 \geq 0\}.\] 
   Therefore,
   \[\mathbb{P}[d_1(F(\mathbf{o}))=1, \mathbf{o} \in \neswarrow] = \sum_{j=0}^{\infty}(\mathbb{P}[X_0 = j] \mathbb{P}[ \mathbf{o} \in \neswarrow]) = \mathbb{P}[X_0 \geq 0] \mathbb{P}[\mathbf{o} \in \neswarrow].\]
   Consider now the event \(d_1(F(\mathbf{o}))=k\) with \(k>1\) (since \(0\) is the child of \(\tau\), the event \(\mathbb{P}[d_1(\tau) = 0]=0\)).
   Observe that, on the event $\mathbf{o} \in \neswarrow$, no negative integer is a child of $F(\mathbf{o})$.
   To see this, it is sufficient to show that \(L_X(0)=- \infty\). Because, the latter condition implies that all the non-negative integers are the descendants of \(0\) by Lemma~\ref{lemma:descendants}.
   So, none of them are the children of \(\tau\) (see Figure~\ref{fig_I_f_trajectory}).
   The latter condition is equivalent to \(S_n \leq 0,\, \forall n<0\), which is further equivalent to \(\mathbf{o} \in \neswarrow\).
   Thus, the observation follows. 
   This observation implies that \(0\) is the last child of \(\tau\).
   So, \(l_X(\tau) = 0\) and \(t_X(\tau) = S_{\tau} \geq 0\).
   Using Lemma~\ref{20230405112810}, one gets the relation \(d_1(\tau)=X_{\tau-1}+1-S_{\tau}\).
   Use this relation to get
   \[\mathbb{P}[d_1(F(\mathbf{o}))=k, \mathbf{o} \in  \neswarrow] = \sum_{j=k-1}^{\infty}\mathbb{P}[X_{\tau-1}=j, S_{\tau}= j-(k-1)] \mathbb{P}[\mathbf{o} \in \neswarrow].\]
   Then, apply Lemma~\ref{20230119141633} to the first term inside the sum to get 
   \[\mathbb{P}[X_{\tau-1}=j, S_{\tau}= j-(k-1)] = \mathbb{P}[X_0 = j]c^{k-1},  0<k-1 \leq j.\]
   Using this, one gets
   \begin{align*}
       \mathbb{P}[d_1(F(\mathbf{o}))=k, \mathbf{o} \in  \neswarrow] &= \sum_{j=k-1}^{\infty}\mathbb{P}[X_0=j]c^{k-1}\mathbb{P}[\mathbf{o} \in \neswarrow]\\
       &= \mathbb{P}[X_0 \geq k-1] c^{k-1} \mathbb{P}[\mathbf{o} \in \neswarrow].
   \end{align*}
   Thus, the offspring distribution of \(\mathbf{o}_1'\) in the tree \(T_1'\) is given by: for any $k \geq 1$,
   \begin{align} \label{eq_pi_bar_distribution}
    \mathbb{P}[d_1(\mathbf{o}_1',T_1')=k-1]=\mathbb{P}[d_1(\mathbf{o}_1',T') = k] &= \mathbf{P}[d_1(F(\mathbf{o}))=k|\mathbf{o} \in \neswarrow] \nonumber\\ 
    &= \mathbb{P}[X_0 \geq k-1]c^{k-1}\\
    &= \bar{\pi}(k-1), \nonumber
   \end{align}
   which proves the first result.

   Recall the construction of unimodularised bi-variate \(EKT(\alpha,\beta)\).
   Note that the offspring distributions \(\alpha, \beta\) need to have means \(m(\alpha)<\infty\) and \(m(\beta)<1\) to define the typically re-rooted joining of i.i.d. bi-variate \(GW(\alpha,\beta)\) Family Trees. 
   It is now shown that \(m(\bar{\pi})<\infty\) (since \(\alpha=\bar{\pi}\) here).
   This follows from the bound \(X_{\tau-1}-S_{\tau} \leq \tau\), which is a consequence of the skip-free to the left property.
   Since \(\mathbb{E}[X_0]>0\), the classical result on stopping time \cite{gutStoppedRandomWalks2009} implies \(\mathbb{E}[\tau]<\infty\).
   Thus, \(m(\bar{\pi}) = \mathbb{E}[X_{\tau-1} - S_{\tau}] \leq \mathbb{E}[\tau]<\infty\).
   The fact that \(m(\tilde{\pi})<1\) follows from Lemma~\ref{20230116135048}.

   It is now shown that, conditioned on $\{d_1(\mathbf{o}_1',T')=k\}$, the descendant subtrees of the children of $\mathbf{o}_1'$ are independent Galton-Watson Trees with offspring distribution $\tilde{\pi}$.

   On $\{d_1(F(\mathbf{o}))=k\}$, let $0=i_k<i_{k-1}< \cdots<i_1$ be the children of $F(\mathbf{o})$.
   Observe that, on the event $\{d_1(F(\mathbf{o}))=k\}$, for each $1 \leq j \leq k$, the part of the random walk $(0, S_{i_j-1}-S_{i_j},S_{i_j-2}-S_{i_j},\ldots, S_{i_{(j+1)}}-S_{i_j})$ is an excursion set because, by the third part of Lemma~\ref{lemma_offspring_count}, $S_{i_{(j+1)}}-S_{i_{j}}=1$ and $S_l \leq S_{i_{(j+1)}}$ for all $i_{(j+1)}<l \leq i_j$.
   Further, these excursion sets are mutually independent because the times $i_j$, for $1 \leq j \leq k$, are stopping times.
   Therefore, for each $j$ chosen above, $(0, S_{i_j-1}-S_{i_j},S_{i_j-2}-S_{i_j},\cdots, S_{i_{(j+1)}}-S_{i_j})$ obtained by conditioning on \(\{d_1(F(\mathbf{o}))=k\}\) is the conditioned skip-free to the right random walk of $(S_{-n})_{n \geq 0}$ conditioned on $\eta_{1}< \infty$ and stopped at $\eta_1$, where $\eta_1 = \min\{n >0: S_{-n}=1\}$ and \(-\infty\) if \(S_{-n}<1\) for all \(n>0\).
   By Lemma~\ref{20230113184856} (and applying the same to $(S_n)_{n \geq 0}$ and $\eta_{-1}$), the conditioned random walk is an unconditioned random walk $(\hat{S}_n)_{n \geq 0}$ whose increments $(\hat{X}_n)_{n \geq 1}$ have distribution $\mathbb{P}[\hat{X}_1=k] = \mathbb{P}[X_1=k]c^{k}$, for all $k \geq -1$.
   Moreover, by Lemma~\ref{20230116135048}, the random walk $(\hat{S}_n)_{n \geq 0}$ has negative drift.
   Therefore, by Proposition~\ref{20230128172431}, the descendant subtree of the children of \(\mathbf{o}_1'\) is a Galton-Watson Tree with offspring distribution $\tilde{\pi}$. 
   %In particular, the mean \(m(\tilde{\pi}) < 1\) by Lemma~\ref{20230116135048}.
   %So, \(GW(\tilde{\pi})\) is a.s. finite.

   \noindent {\em Step 3, \(\mathbf{o}'_i\) is the smallest child of \(\mathbf{o}'_{i+1}\),\(\forall i \in \mathbb{Z}\):}
Observe that an integer \(j\) belongs to the bi-infinite path of \(\mathbf{T}\) if and only if \(t(j) \geq 0\).
This follows because, for any \(j \in \mathbb{Z}\), its type \(t(j)\) is non-negative  if and only if \(y(n,j) \geq 0\) for all \(n<j\).
The latter condition is satisfied if and only if \(L(j)=-\infty\).
By Lemma~\ref{lemma:descendants}, \(L(j)=-\infty\) if and only if \(j\) has infinitely many descendants.

For any $i \in \mathbb{Z}$ with \(t(i) \geq 0\), let $l(i) :=l_X(i)= \max\{m<i: y(m,i)=t(i)\}$ and $l^n(i):=l^{n-1}(l(i))$ for all $n>1$.
By Lemma~\ref{smallest_child_20230404185835}, it follows that $l(i)$ is the smallest among the children of $i$.

 On the event $\{S_n\leq 0 \quad \forall n<0\}$ (which is equivalent to $\mathbf{o} \in \neswarrow$), one has $t(0)\geq 0$.
 Since $S_m \leq S_{l^n(0)}$ for all $m \leq l^n(0), n \geq 1$, one has $t(l^n(0))\geq 0$ for all $n \geq 1$.
 This implies that \(\mathbf{o}'_{-n} = l^n(0)\), for all \(n \geq 1\).
 So, by the above discussion, on the event $\mathbf{o} \in \neswarrow$, one gets that $R(l^n(0))=l^{n-1}(0)$ for all $n \geq 1$ (with the notation $l^0(0)=0$).
 Moreover, on the same event, the vertices of the bi-infinite $F$-path of $\mathbf{T}$ are $\{l^n(0): n \geq 1\} \cup \{0\} \cup \{R^m(0): m \geq 1\}$. 
 Since, on the event $\mathbf{o} \in \neswarrow$, for any $m \geq 1$, $S_{R^m(0)}-S_n > S_{R^m(0)}-S_{R^{m-1}(0)} $ for all $R^{m-1}(0) < n < R^m(0)$, one has $l(R^m(0))=R^{m-1}(0)$ (with the notation that $R^0(0)=0$).
 Therefore, \(R^{m}(0) = \mathbf{o}_m', \ \forall m \geq 1\); \(l^n(0)\) is the last child of \(l^{n-1}(0)\), for all \(n \geq 1\); and \(R^m(0)\) is the last child of \(R^{m+1}(0)\), for all \(m \geq 0\).
 Thus, $\mathbf{o}'_i$ has the smallest order among the children of $\mathbf{o}'_{i+1}$ for all $i \in \mathbb{Z}$.
 Since the order of $(\mathbf{o}'_i)_{i\in \mathbb{Z}}$ characterizes the order on $[\mathbf{T}', \mathbf{o}']$, this completes Step 3.

\end{proof}

% \begin{tcolorbox}[colback=orange!5!white, colframe = orange!10!white]
%     Modify the summary at the top of this section after writing the three subsections
%     Using h-transform, re-write some lemmas
% \end{tcolorbox}
\section{Record representable EFTs}

This section focuses on the following question:
Given a unimodular ordered EFT \([\mathbf{T},\mathbf{o}]\), does there exist a stationary sequence \(Y=(Y_n)_{n \in \mathbb{Z}}\) of random variables taking values in \(\{-1,0,1,2,\ldots\}\) such that \([\mathbf{T},\mathbf{o}]\) is the component of \(0\) in the record graph of the network \((\mathbb{Z},Y)\)?
If such a sequence exists, then \([\mathbf{T},\mathbf{o}]\) is considered to be {\em record representable}.
Theorem~\ref{thm_representation} provides sufficient conditions for \([\mathbf{T},\mathbf{o}]\) to be record representable.
In particular, it is shown that every unimodular ordered EFT of class \(\mathcal{I}/\mathcal{I}\) is record representable.
As for the unimodular ordered EFTs of class \(\mathcal{I}/\mathcal{F}\), it is shown that if they have unique succession lines, then they are record representable.

Theorem~\ref{thm_representation} is proved in the following way: first, it is shown that on any unimodular ordered EFT, the functions \(a\) and \(b\) (given in Def.~\ref{defn_a_b}) are well-defined.
In an ordered EFT, the function \(a\) maps a vertex to its immediate predecessor, whereas the function \(b\) maps a vertex to its immediate successor.
Next, it is shown that every unimodular ordered EFT has at most two succession lines.
A backward map \(\Phi_R\) is defined on the set of ordered Family Trees that have at most one bi-infinite \(F\)-path.
The image of the map \(\Phi_R\) is a sequence that encodes the number of offsprings along the succession line.
This map extends the coding of finite ordered trees defined in \cite{legallRandomTreesApplications2005b}.
It is shown that \(\Phi_R\) is bijective on the set of ordered EFTs that have unique succession lines.
The stationary sequence \(Y\) is obtained by taking the image of \([\mathbf{T},\mathbf{o}]\) under  \(\Phi_R\).
The stationarity of \(Y\) follows by the Point-stationarity Theorem~\ref{thm_point_stationarity}.

An immediate corollary of Theorem~\ref{thm_representation} is that if \(X=(X_n)_{n \in \mathbb{Z}}\) is a stationary sequence of random variables taking values in \(\mathbb{Z}\), \((\mathbb{Z},X)\) be its associated network, \(f\) is a vertex-shift on \((\mathbb{Z},X)\) and \([\mathbb{Z}^f_X(0),0]\) be the component of \(0\) in the \(f\)-graph of \((\mathbb{Z},X)\) such that \([\mathbb{Z}^f_X(0),0]\) has a unique succession line, then there exist a stationary sequence \(Y=(Y_n)_{n \in \mathbb{Z}}\) of random variables taking values in \(\{-1,0,1,2,\ldots\}\) such that the component \([\mathbb{Z}^R_Y(0),0]\) of $0$ in the record graph of the network \((\mathbb{Z},Y)\) has the same distribution as that of \([\mathbb{Z}^f_X(0),0]\).
In particular, if \([\mathbb{Z}^f_X(0),0]\) is of class \(\mathcal{I}/\mathcal{I}\), then such a stationary sequence \(Y\) exists.
The vertex-shift \(f\) on \((\mathbb{Z},X)\) is said to be {\em record representable} if \([\mathbb{Z}^f_X(0),0]\) is record representable.

\begin{theorem}\label{thm_representation}
    Let \([\mathbf{T},\mathbf{o}]\) be a unimodular ordered EFT such that \(\mathbf{T}\) has a unique succession line.
    Then, there exists a stationary sequence \(Y=(Y_n)_{n \in \mathbb{Z}}\) of random variables taking values in \(\{-1,0,1,2,\cdots\}\) such that \([\mathbf{T},\mathbf{o}] \overset{\mathcal{D}}{=}[\mathbb{Z}^R_Y(0),0]\), where \(R\) is the record vertex-shift and \(\mathbb{Z}^R_Y(0)\) is the component of \(0\) in the record graph \(\mathbb{Z}^R_Y\) of the network \((\mathbb{Z},Y)\).
    In particular, if \([\mathbf{T},\mathbf{o}]\) is of class \(\mathcal{I}/\mathcal{I}\), then such a stationary sequence \(Y\) exists.
\end{theorem}

\begin{corollary}
    Let \(X=(X_n)_{n \in \mathbb{Z}}\) be a stationary sequence of integer-valued random variables and \([\mathbb{Z},0,X]\) be its associated network.
    Let \(f\) be a vertex-shift on \((\mathbb{Z},X)\) such that \(0\) has a.s. infinitely many ancestors in the \(f\)-graph \(\mathbb{Z}^f_X\) and the component \(\mathbb{Z}^f_X(0)\) of \(0\) in \(\mathbb{Z}^f_X\) has a unique succession line.
    Then, there exists a stationary sequence \(Y=(Y_n)_{n \in \mathbb{Z}}\) of random variables taking values in \(\{-1,0,1,2,\cdots\}\) such that \([\mathbb{Z}^f_X(0),0] \overset{\mathcal{D}}{=}[\mathbb{Z}^R_Y(0),0]\), where \(R\) is the record vertex-shift and \(\mathbb{Z}^R_Y(0)\) is the component of \(0\) in the record graph \(\mathbb{Z}^R_Y\) of the network \((\mathbb{Z},Y)\).
\end{corollary}
\begin{proof}
    Since \(X\) is stationary, the network \([\mathbb{Z}^f_X(0),0]\) is unimodular.
    Apply Theorem~\ref{thm_representation} to get the result.
\end{proof}

Theorem~\ref{thm_representation} is proved after the following lemmas.

Let \([\mathbf{T},\mathbf{o}]\) be a unimodular ordered EFT.
Denote the RLS (total) order obtained on the vertices of \(\mathbf{T}\) by \(\prec\) (see Section \ref{subsec_RLS}).
The set of minimal vertices of \(\mathbf{T}\) is a covariant subset.
It can be either empty or a singleton.
Since \(\mathbf{T}\) has infinitely many vertices, by No Infinite/ Finite Inclusion (Lemma~\ref{lemma_no_infinite_finite}), a.s., \(\mathbf{T}\) does not have any minimal vertices.
Therefore, the map \(b:V(\mathbf{T})\to V(\mathbf{T})\) given by \(b(u) = \max\{v: v \prec u\}\) (if \(\max\) exists) is a.s. well-defined.
The map \(b\) is injective on \(V(\mathbf{T})\) since \(T\) is totally ordered.
Since \([\mathbf{T},\mathbf{o}]\) is unimodular and \(b\) is injective, it follows from Proposition~\ref{prop_injective_bijective} that \(b\) is bijective a.s..
But \(a(b(u)) = u\), for all \(u \in V(\mathbf{T})\), where \(a(v) = \min\{w:w \succ v\}\).
Since the range of \(b\) is \(V(\mathbf{T})\), the map \(a\) is well-defined on \(V(\mathbf{T})\).
Thus, the following lemma holds:

\begin{lemma}\label{lemma_a_b_wellDefined}
    Let \([\mathbf{T},\mathbf{o}]\) be a unimodular ordered EFT.
    Then, the maps \(a:V(\mathbf{T}) \to V(\mathbf{T})\), \(b:V(\mathbf{T}) \to V(\mathbf{T})\) are a.s. well-defined.
\end{lemma}

Let \(T\) be an ordered EFT.
For any vertex \(u \in V(T)\), let \(U((T,u)) = (u_n)_{n \in \mathbb{Z}}\) be the succession line passing through \(u\) (see Def. \ref{defn_succession_line}).
The set \(\{u_n:n \in \mathbb{Z}\}\) is called the succession line of \(u\).
It is easy to check that if two succession lines intersect then they are one and the same.

Let \([\mathbf{T},\mathbf{o}]\) be a unimodular ordered EFT.
Lemma~\ref{lemma_a_b_wellDefined} implies that \(U((\mathbf{T},u))\) has distinct entries for all \(u \in V(\mathbf{T})\).
In Lemma~\ref{lemma_unique_succession_line}, it is shown that a every unimodular ordered EFT has at most two succession lines.
Note that by the classification theorem, \([\mathbf{T},\mathbf{o}]\) is either of class \(\mathcal{I}/\mathcal{I}\) or of class \(\mathcal{I}/\mathcal{F}\).

\begin{lemma}\label{lemma_unique_succession_line}
    Let \([\mathbf{T},\mathbf{o}]\) be a unimodular ordered EFT.
    \begin{itemize}
        \item If \([\mathbf{T},\mathbf{o}]\) is of class \(\mathcal{I}/\mathcal{I}\), then \(\mathbf{T}\) has a unique succession line and it contains all the vertices of \(\mathbf{T}\).
        \item If \([\mathbf{T},\mathbf{o}]\) is of class \(\mathcal{I}/\mathcal{F}\), then \(\mathbf{T}\) has at most two succession lines; together they contain all the vertices of \(\mathbf{T}\).
    \end{itemize}
\end{lemma}

\begin{proof}

    If \([\mathbf{T},\mathbf{o}]\) is of class \(\mathcal{I}/\mathcal{I}\), then every vertex of \(\mathbf{T}\) has finitely many descendants.
    Let \(v \in V(\mathbf{T})\).
    Consider the smallest common ancestor \(w\) of \(v\) and \(\mathbf{o}\).
    The descendant tree of \(w\) is finite and it contains both \(v\) and \(\mathbf{o}\).
    Thus, the succession lines of \(v\) and \(o\) intersect, which implies that they are one and the same.

    Assume that \([\mathbf{T},\mathbf{o}]\) is of class \(\mathcal{I}/\mathcal{F}\).
    Let \(\neswarrow\) denote the unique bi-infinite path of \(\mathbf{T}\).
    Consider the set
  \begin{equation}\label{eq_W}
      W=\{u \in V(\mathbf{T}): u \prec v \ \forall v \in \neswarrow\}.
  \end{equation}
Denote its complement \(V(\mathbf{T})\backslash W\) by \(W'\).
Firstly, it is shown that \(W'\) has a unique succession line and all the vertices of \(W'\) belong to this succession line.
See Figure~\ref{fig_i_f_succession_lines} for an illustration.

If \(u \in W'\), then there exists a vertex \(v \in \neswarrow\) such that \(u \succ v\) (one may safely assume it because if \(u \in \neswarrow\) then u has a child \(v' \in \neswarrow\) and \(u \succ v' \)).
The vertex \(a(u) = \min\{w: w \succ u\}\) also satisfies \(a(u) \succ v\) since \(a(u) \succ u\), implying that \(a(u) \in W'\).
The vertex \(b(u)= \max\{w: w \prec u\}\) satisfies \(b(u) \succeq v\) since \(v \in \{w:w \prec u\}\), implying that \(b(u) \in W'\).
Therefore, for every vertex \(u \in W'\), the succession line of \(u\) is contained in  \(W'\).
For any two vertices \(u,v\) in \(W'\), there are finitely many vertices in between them, i.e., the set \(\{w:u \prec w \prec v\}\) is finite (assuming that \(u \prec v\)).
Therefore, the succession line of \(u\) and \(v\) are one and the same.
Thus, \(W'\) has a unique succession line and this succession line contains all the vertices of \(W'\).

It is now shown that if \(W\) is non-empty, then it has a unique succession line and all the vertices of \(W\) belong to this succession line.
Let \(u \in W\).
The vertex \(a(u)= \min\{w: w \succ u\}\) also belongs to \(W\) because \(\neswarrow \subset \{w:w \succ u\}\) and \(a(u) \not \in \neswarrow\).
The vertex \(b(u)= \max\{w: w \prec u\}\) also belongs to \(W\) since \(b(u)\prec u\) and \(u \in W\).
Therefore, the succession line of \(u\) is contained in \(W\).
For any two vertices \(u,v\) in \(W\), there are finitely many vertices between them.
Thus, \(W\) has a unique succession line (if \(W\) is non-empty).
This completes the proof.
\end{proof}

\begin{figure}[htbp]
    % gives better spacing than \begin{center}...\end{center}
       \includegraphics[scale=0.85]{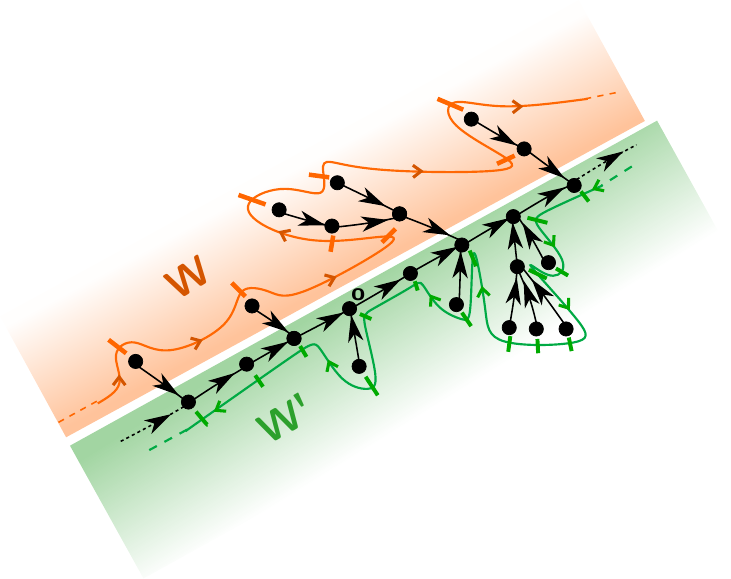}
       \caption{An Illustration of the \(\mathcal{I}/\mathcal{F}\) case in the proof of Lemma~\ref{lemma_unique_succession_line}.
       The set \(W\) and its complement \(W'\) are highlighted. Each one of these sets contains a unique succession line passing through all of its vertices. The arrow direction of the succession line indicates the images under the map \(b\), i.e., the immediate successor of the previous vertex. In this figure, the root \(\mathbf{o}\) belongs to the succession line containing the bi-infinite path.}
       \label{fig_i_f_succession_lines}
     \end{figure}

\begin{remark}
    It follows from the proof of Lemma~\ref{lemma_unique_succession_line} that if \([\mathbf{T},\mathbf{o}]\) is unimodular ordered EFT of class \(\mathcal{I}/\mathcal{F}\) that has a unique  succession line, then it has the ECS (Every Child Succeeding) order (see Subsection \ref{subsec_bi_variate_ekt} for the definition).
\end{remark}

\begin{example}[Unimodular EFT with two succession lines]
  Consider the unimodularised bi-variate Eternal Kesten Tree with offspring distributions \(\alpha,\beta\), as described in Subsection~\ref{subsec_bi_variate_ekt}, but with a uniform order imposed on the vertices of the bi-infinite path instead of the ECS order.
  The construction of this tree is as follows: start with the typically re-rooted joining \([\mathbf{T},\mathbf{o}]\) of i.i.d. sequence of bi-variate ordered \(GW(\alpha, \beta)\), where the offspring distributions \(\alpha\) and \(\beta\) have means \(m(\alpha)<\infty\) and \(m(\beta)<1\).
  Let \((\mathbf{o}'_n)_{n \in \mathbb{Z}}\) denote the vertices of the bi-infinite path of \([\mathbf{T}',\mathbf{o}']\), where \((\mathbf{o}'_n)\) is the child of \((\mathbf{o}'_{n+1})\), for all \(n \in \mathbb{Z}\), and  \(\mathbf{o}_0'\) is the ancestor of the root \(\mathbf{o}'\).
  For each \(n \in \mathbb{Z}\), assign a uniform order to the children of \(\mathbf{o}_n'\).
  In particular, given that \(\mathbf{o}_{n+1}'\) has \(k\) children, the order of \(\mathbf{o}_n'\) is uniformly distributed on \({1,2,\ldots,k}\).
  The resulting unimodularised ordered bi-variate \(EKT(\alpha,\beta)\) \([\mathbf{T},\mathbf{o}]\) has infinitely many vertices in the set \(W\) (Eq.~\ref{eq_W}) with positive probability.
  According to Lemma~\ref{lemma_unique_succession_line}, \([\mathbf{T},\mathbf{o}]\) has two succession lines with positive probability.
\end{example}

The following lemma is used to associate a sequence to an ordered finite tree.

\begin{lemma} \label{lemma:legall_sums}
    Let $T$ be a finite ordered tree and $v_1 \prec v_2 \prec \cdots \prec v_n$ be the vertices of $T$ ordered according to the RLS order. Then,
    \begin{enumerate}
      \item $\sum_{i=1}^n (d_1(v_i)-1) = -1$.
      \item $\sum_{i = 1}^k (d_1(v_i)-1) < 0$,  for all $1<k \leq n$.
    \end{enumerate}
  \end{lemma}
  
  \begin{proof}
  The first equation follows from the fact that for a finite tree $T$,
  \begin{equation*}
    \sum_{u \in V(T)} d_1(u)= \#V(T) - 1.
  \end{equation*}
  
  The second equation is now proved by induction on $k$.
  If $k=1$ then $v_1$ is a leaf and hence $d_1(v_1)-1 = -1$.
  Now, assume that the statement is true for $1 \leq m \leq k$, where \(k>1\).
  Observe that by the RLS order, either $v_{k+1}$ is a leaf or $v_{k+1}$ is the parent of $v_k$. 
  In the latter case, there exists $1 \leq j \leq k$ such that $v_j<v_{j+1} < \cdots < v_k$ (written according to the RLS order) are all the descendants of $v_{k+1}$.
  If $v_{k+1}$ is a leaf then, by induction, 
  \begin{align*}
    \sum_{i=1}^{k+1} (d_1(v_i)-1) &= \sum_{i=1}^k (d_1(v_i)-1) + d_1(v_{k+1})-1 <-1.
  \end{align*}
  
  The only other possibility is that $v_{k+1}$ is the parent of $v_k$.
   Then, 
  \begin{align*}
    \sum_{i=1}^{k+1} (d_1(v_i)-1) &= \sum_{i = 1}^{j-1}(d_1(v_i) - 1) +  \sum_{i=j}^{k+1}(d_1(v_i)-1) \leq 0 +( -1).
  \end{align*}
  The last step follows from the fact that the first sum on the right-hand side is at most $0$ (by induction), while the second sum is on the descendant tree of $v_{k+1}$ which is equal to $-1$ by the first result of the lemma.
  \end{proof}

Consider the following subset
\begin{equation}
    \widetilde{\mathcal{T}}_*:=\{[T,o]: T \text{ is an ordered EFT with a unique succession line}\}
\end{equation}
of EFTs.
It follows from Lemma~\ref{lemma_a_b_wellDefined} and Lemma~\ref{lemma_unique_succession_line} that for all \([T,o] \in \widetilde{\mathcal{T}}_*\), the succession line \(U((T,o)) = (u_n)_{n \in \mathbb{Z}}\) passing through \(o\) with \(u_0=o\) is a bijection from \(\mathbb{Z}\) to \(V(T)\) that maps \(n\) to \(u_n\), for all \(n \in \mathbb{Z}\).

\begin{definition}[Backward map]
  Define the backward map \(\Phi_R\) on the set \( \mathcal{T}'_*\) of Family Trees that have at most one bi-infinite \(F\)-path in the following way.
For all \([T,o] \in \mathcal{T}'_*\), \(\Phi_R([T,o]) := [\mathbb{Z},0,y]\), where \(y = (y_n)_{n \in \mathbb{Z}}\) is the sequence given by \(y_n = d_1(u_{n+1})-1\), for all \(n \in \mathbb{Z}\) and \((u_n)_{n \in \mathbb{Z}} = U((T,o))\) is the succession line passing through \(o\).
\end{definition}

The sequence \(y=(y_n)_{n \in \mathbb{Z}}\) takes values in \(\{-1,0,1,2,\cdots\}^{\mathbb{Z}}\).
It is shown that \(\Psi_R \circ \Phi_R = I\) on \(\widetilde{\mathcal{T}}_*\), where \(\Psi_R([\mathbb{Z},0,y]) = [\mathbb{Z}^R_y(0),0]\) is the component of \(0\) in the record graph of the network \((\mathbb{Z},y)\).
The proof relies on the property that every elder sibling of any vertex (if they exist), of a Family Tree with a unique succession line, has finitely many descendants.

\begin{lemma}\label{lemma_encoded_sequence_same_tree}
    For all \([T,o] \in \widetilde{\mathcal{T}}_*\), \(\Psi_R \circ \Phi_R([T,o]) = [T,o]\).
\end{lemma}
\begin{proof}
    Let \((u_n)_{n \in \mathbb{Z}} = U((T,o))\) be the succession line of \((T,o)\) and $[\mathbb{Z},0,y] = \Phi_R([T,o])$, where $y=(y_n)_{n \in \mathbb{Z}}$.

    It is now proved that the bijective map \(\alpha:\mathbb{Z} \to V(T)\) defined by \(\alpha(n)= u_n\), for all \(n \in \mathbb{Z}\), induces a rooted network isomorphism \(\alpha\) from \([\mathbb{Z}^R_y(0),0]=\Psi_R([\mathbb{Z},0,y])\) to \((T,o)\).
    Note that, if \(u_j\) is a parent of \(u_i\) in \(T\) for some integers \(i\) and \(j\) then \(j>i\), which follows because \((u_n)_{ \in \mathbb{Z}}\) is the succession line of \((T,o)\).
    In view of this, to prove that \(\alpha\) induces a rooted network isomorphism, it is enough to show that for all $i \in \mathbb{Z}$, if \(u_j\) is the parent of \(u_i\) in \(T\) for some  \(j>i\), then \(j=R_y(i)\).
    It is enough to show the latter because (assuming the latter implication) if \(j = R_y(i)\) and \(u_k\) is a parent of \(u_i\) for some \(k>i\), then \(R_y(i)=k\), which implies that \(j=k\), and \(u_k = u_j\).
    
   Let \(i \in \mathbb{Z}\) and \(u_j\) be the parent of \(u_i\) in \(T\) for some integer $j>i$.
    If $j=i+1$, then $x_i = d_1(u_{i+1})-1 \geq 0$, because $u_i$ is a child of $u_{i+1}$.
    Hence, $R_y(i)=i+1$.
  
   So, assume that $j>i+1$.
   This implies that either $u_{i+1}$ is a leaf and a descendant of a sibling of $u_i$ or \(u_{i+1}\) is a sibling of \(u_i\).
   Every sibling \(v\) of \(u_i\) such that \(v \succ u_i\) has finitely many descendants.
   This follows because if there exists a sibling \(v \succ u_i\) of \(u_i\) that has infinitely many descendants, then \(u_i \in W\) (see Eq.~(\ref{eq_W})).
   But \(W\) is empty since \(T\) has a unique succession line.

    Let $u_{i_1} \prec u_{i_2} \prec \cdots  \prec u_{i_n}$ be the elder siblings of $u_i$ (i.e., \(u_i \prec u_{i_1}\)), and $T_1,T_2,\cdots,T_n$ be their descendant trees respectively.
   For $i<k<j$, let $l(k)$ be the smallest element of the set $\{1,2,\cdots,n\}$ such that $\{u_{i+1},u_{i+2},\cdots,u_k\} \subset V(T_1) \cup V(T_2) \cup \cdots V(T_{l(k)})$, and $u_{k_1}$ be the smallest vertex of $T_{l(k)}$.
  
   Then, the sum $z(i,k) = \sum_{m=i}^{k-1} y_m= \sum_{m=i+1}^{k}(d_1(u_m)-1)$ can be written as 
   \begin{align*}
     z(i,k) &= \left(\sum_{u \in V(T_1)}d_1(u)-1 \right)+ \cdots + \left(\sum_{u \in V(T_{l(k)-1})}d_1(u) - 1 \right) + \sum_{m=k_1}^k (d_1(u_m)-1)\\
     &< (-1) + \cdots + (-1) + \sum_{m=k_1}^k (d_1(u_m)-1)< 0,
   \end{align*}
   with the notation that \(V(T_0)\) is an empty set.
   The last steps follow from the first and the second statements of Lemma~\ref{lemma:legall_sums}. 
  
    On the other hand, the sum $z(i,j) = \sum_{m=i+1}^j (d_1(u_m)-1)$ can be written as
    \begin{align*}
      z(i,j) &= \left(\sum_{u \in V(T_1)}d_1(u)-1 \right)+ \cdots + \left(\sum_{u \in V(T_n)}d_1(u) - 1 \right) + d_1(u_j)-1\\
      &= -n + d_1(u_j) - 1 \geq 0.
    \end{align*}
    The above steps follow from the first part of Lemma~\ref{lemma:legall_sums}, and by the assumption that $u_j$ has at least $n+1$ children, namely $u_i,u_{i_1},\cdots,u_{i_n}$.
    Thus, $R_x(i)=j$, completing the proof.
\end{proof}

\begin{proof}[Proof of Theorem~\ref{thm_representation}]
    Let \([\mathbb{Z},0,Y] = \Phi_R([\mathbf{T},\mathbf{o}])\).
    Since \(\mathbf{T}\) has a unique succession line, it follows that a.s., \(\Psi_R \circ \Phi_R([\mathbf{T},\mathbf{o}]) = [\mathbf{T},\mathbf{o}]\) by Lemma~\ref{lemma_encoded_sequence_same_tree}.
    Therefore, a.s., \([\mathbb{Z}^R_Y(0),0] = [\mathbf{T},\mathbf{o}]\).

    It is now shown that the sequence \(Y\) is stationary.
    Let \((u_n)_{n \in \mathbb{Z}}=U((\mathbf{T},\mathbf{o}))\) be the succession line with \(u_0=\mathbf{o}\).
    Since \(a\) and \(b\) are well-defined on a unimodular EFT (by Lemma~\ref{lemma_a_b_wellDefined}), it implies that the map \(u_n \mapsto u_{n+i}, \forall n \in \mathbb{Z}\)  is bijective for any \(i \in \mathbb{Z}\).
    Therefore, the map \([\mathbf{T},\mathbf{o}] \mapsto [\mathbf{T},u_i]\) is measure-preserving, for any \(i \in \mathbb{Z}\), by Point-stationarity Theorem~\ref{thm_point_stationarity}.
    Since, \([\mathbb{Z},0,T_iY] = \Phi_R([\mathbf{T},u_i])\), the map \((Y_n)_{n \in \mathbb{Z}} \mapsto (Y_{n+i})_{n \in \mathbb{Z}}\) is measure-preserving for any \(i \in \mathbb{Z}\).
    Thus, \(Y\) is stationary.
\end{proof}

\subsection{Examples of record representable vertex-shifts}

\begin{example}[Strict record vertex-shift on an i.i.d. sequence (class \(\mathcal{I}/\mathcal{I}\))]\normalfont \label{example_strict_record}
    An illustration of this example can be seen in Figure~\ref{fig_example_strict_record}.
    The strict record vertex-shift \(SR\) is defined on the set of networks of the form \([\mathbb{Z},0,x]\), where \(x=(x_n)_{n \in \mathbb{Z}}\) is a sequence of integers.
    It is defined by 
    \[SR(i) = \begin{cases*}
        \inf\{n>i:\sum_{k=i}^{n-1}x_k>0\} \text{ if the infimum exists,}\\
        i \text{ otherwise,}
    \end{cases*}\]
    for all \(i \in \mathbb{Z}\).
    Let \(X=(X_n)_{n \in \mathbb{Z}}\) be an i.i.d. sequence of random variables taking values in \(\mathbb{Z}\) such that their common mean exists.
    Let \(\Psi_{SR}\) be the map that takes any network of the form \([\mathbb{Z},0,x]\) and maps to the component of \(0\) in the \(SR\)-graph of \((\mathbb{Z},x)\).
    One can show that the component of \(0\) in the \(SR\)-graph is of class \(\mathcal{I}/\mathcal{I}\) if \(\mathbb{E}[X_0]=0\) and \(\mathbb{E}[X_0^2]>0\), because \(\limsup_{n \to \infty}S_{-n} = \infty\); which implies that the descendant tree of \(0\) is a.s. finite, and because \(0\) has infinitely many ancestors; which follows from the Chung-Fuchs recurrence theorem.
    By Theorem~\ref{thm_representation}, there exists a stationary sequence \(Y=(Y_n)_{n \in \mathbb{Z}}\) of random variables taking values in \(\{-1,0,1,2,\cdots\}\) such that \([\mathbb{Z}^{SR}_X(0),0] \overset{\mathcal{D}}{=} [\mathbb{Z}_Y^R(0),0]\).
    \begin{figure}[htbp]
     % gives better spacing than \begin{center}...\end{center}
        \includegraphics[scale=0.60]{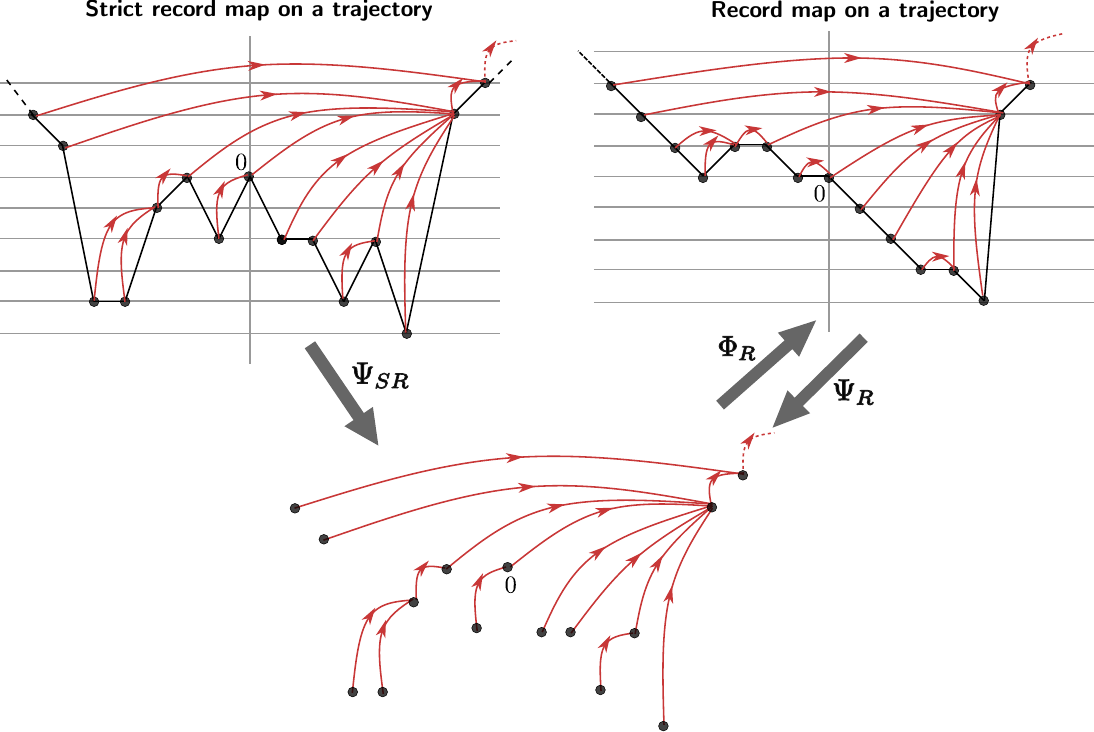}
        \caption{An Illustration of Example \ref{example_strict_record}.
        The top left figure shows the trajectory of a realization of the i.i.d. sequence \(X=(X_n)_{n \in \mathbb{Z}}\), where \(X_n \in \mathbb{Z}, \forall n\).
        The  \(SR\)-graph (strict record graph) is drawn in red on the trajectory of the top left figure.
        The component of \(0\) in the \(SR\)-graph is drawn in the bottom figure.
        The top right figure shows the trajectory of the sequence \(Y=(Y_n)_{n \in \mathbb{Z}}\) which is obtained by applying the backward map \(\Phi_R\) to the rooted tree of the bottom figure.
        The component of \(0\) in the record graph of the trajectory in the top right figure is the same as the component of \(0\) in the \(SR\)-graph.}
        \label{fig_example_strict_record}
      \end{figure}
\end{example}

\begin{example}[Climbing point vertex-shift that depends on the future (class \(\mathcal{I}/\mathcal{F}\))]\normalfont \label{example_climbing_vs}
    An illustration of this example is shown in Figure~\ref{fig_example_climbing_vertex-shift}.
    Define a vertex-shift $C$ (C for climbing point) based on an example taken from \cite[Example 1(a)]{fossStochasticSequencesRegenerative2013a}.
    Let \(X=(X_n)_{n \in \mathbb{Z}}\) be an i.i.d. sequence of random variables taking values in \(\{-1,0,1\}\) with probabilities given by: \(\mathbb{P}[X_0=1]=p, \mathbb{P}[X_0=-1]=q\) and \(\mathbb{P}[X_0 =  0]=1-p-q\), where \(0<q<p\) and \(p+q<1\).
    Define the vertex-shift on the networks of the form \([\mathbb{Z},0,x]\), where \(x=(x_n)_{n \in \mathbb{Z}}\) is a sequence with \(x_n \in \{-1,0,1\}\ \forall n\) by
    \[C(i) = \begin{cases}
        \inf\{n>i:\sum_{k=n}^{m}x_k \geq 0, \forall m>n\} \text{ if the infimum exists},\\
        i \text{ otherwise,}
    \end{cases}\]
    for all \(i \in \mathbb{Z}\).
    Let \(\Psi_C\) denote the map that takes the networks of the form \([\mathbb{Z},0,x]\) and maps to the connected component of \(0\) in the \(C\)-graph of \((\mathbb{Z},x)\).
    Note that \(C(i)=k\) for some \(k>i\) if and only if \(k\) is the smallest integer larger than \(i \) such that \(S_{k} = \min\{S_n:n>i\}\).
    Since \(\mathbb{P}[X_0]=p-q>0\), the random walk \((S_n)_{n \geq 0}\) drifts to \(+\infty\), where \(S_0=0, S_n = \sum_{k=0}^{n-1}X_k\).
    Therefore, the event \(C(i)>i\) for all \(i \in \mathbb{Z}\) occurs with probability \(1\).
    Since, \(S_{-n}=\sum_{k=-n}^{-1}-X_k\) converges to \(-\infty\) as \( n \to \infty\), the random variable \(\max\{S_{-n}:n \geq 0\}\) is finite a.s.
    Thus, some ancestor of \(0\) in the \(C\)-graph has infinitely many descendants.
    So, the component \([\mathbb{Z}_X^C(0),0]\) of \(0\) in the \(C\)-graph is of class \(\mathcal{I}/\mathcal{F}\).
    One can show that the smallest child of every vertex which lies on the bi-infinite path of \(\mathbb{Z}^C_X(0)\) belongs to the bi-infinite path.
    Therefore, \([\mathbb{Z}_X^C(0),0]\) is of class \(\mathcal{I}/\mathcal{F}\) and has a unique succession line (since it has ECS order).
    \begin{figure}[htbp]
           \includegraphics[scale=0.5]{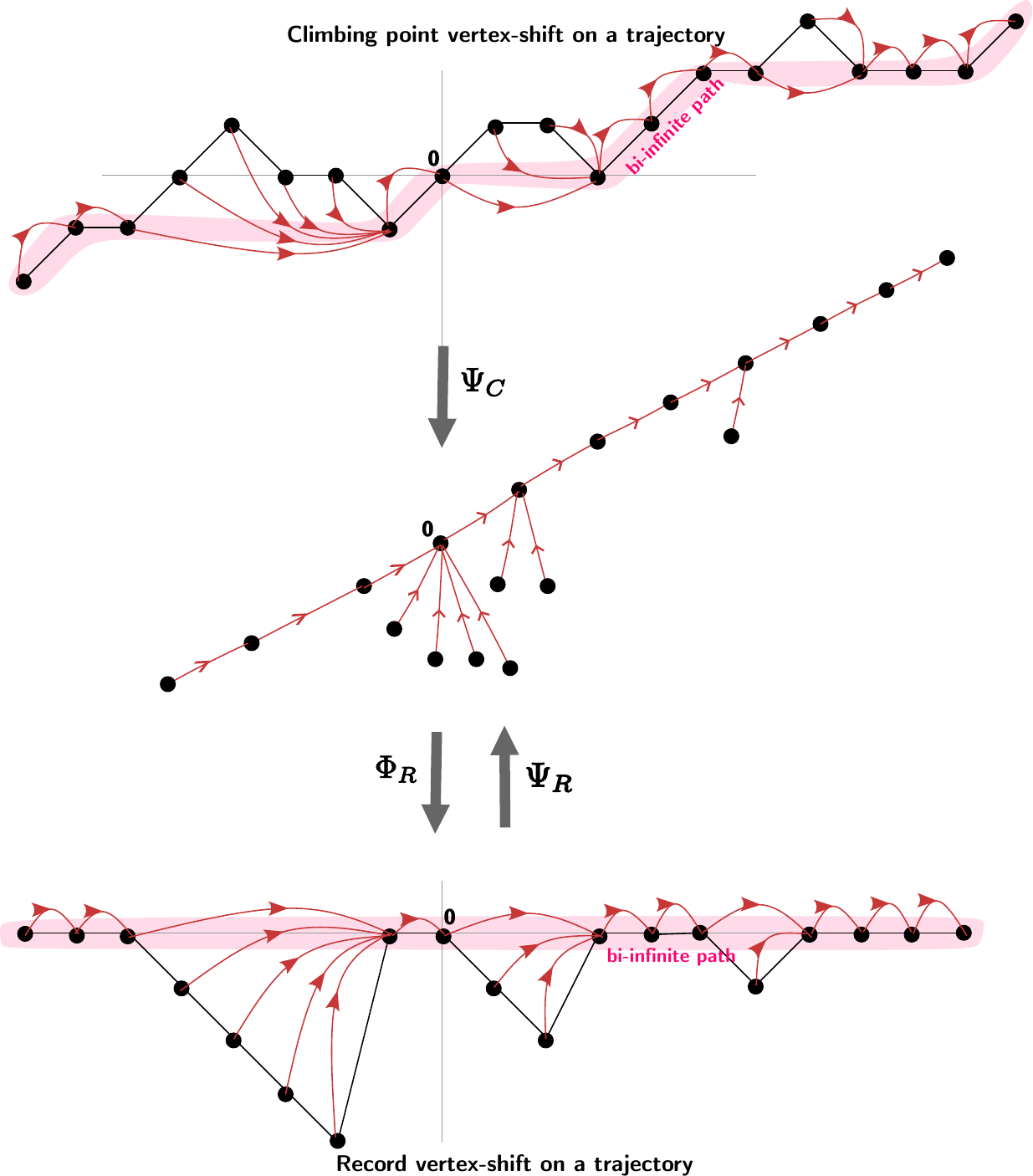}
           \caption{An Illustration of Example \ref{example_climbing_vs}.
           The top figure shows the trajectory of a realization of the i.i.d. sequence \(X=(X_n)_{n \in \mathbb{Z}}\), where \(X_n \in \mathbb{Z}, \forall n\).
           The  \(C\)-graph (\(C\) for climbing point vertex-shift) is drawn in red on the trajectory of the top figure.
           The vertices on the bi-infinite path are highlighted.
           The component of \(0\) in the \(C\)-graph is drawn in the middle figure.
           The bottom figure shows the trajectory of the sequence \(Y=(Y_n)_{n \in \mathbb{Z}}\) which is obtained by applying the backward map \(\Phi_R\) to the rooted tree shown in the middle figure.
           The component of \(0\) in the record graph of the trajectory in the bottom figure is the same as the component of \(0\) in the \(C\)-graph.}
           \label{fig_example_climbing_vertex-shift}
         \end{figure}
\end{example}

%%%%%%%%%%%%%%%%%%%%%%%%%%%%%%%%%%%%%%%%%%%%%%
%% Single Appendix:                         %%
%%%%%%%%%%%%%%%%%%%%%%%%%%%%%%%%%%%%%%%%%%%%%%
%\begin{appendix}
%\section*{???}%% if no title is needed, leave empty \section*{}.
%\end{appendix}
%%%%%%%%%%%%%%%%%%%%%%%%%%%%%%%%%%%%%%%%%%%%%%
%% Multiple Appendixes:                     %%
%%%%%%%%%%%%%%%%%%%%%%%%%%%%%%%%%%%%%%%%%%%%%%
%\begin{appendix}
%\section{???}
%
%\section{???}
%
%\end{appendix}

%%%%%%%%%%%%%%%%%%%%%%%%%%%%%%%%%%%%%%%%%%%%%%
%% Support information, if any,             %%
%% should be provided in the                %%
%% Acknowledgements section.                %%
%%%%%%%%%%%%%%%%%%%%%%%%%%%%%%%%%%%%%%%%%%%%%%
\begin{acks}[Acknowledgments]
  Most of the material contained in this paper was previously posted (in modified form) in the PhD thesis \cite{roychoudhuryRecordsStationaryProcesses2023} of the second author.
\end{acks}
%%%%%%%%%%%%%%%%%%%%%%%%%%%%%%%%%%%%%%%%%%%%%%
%% Funding information, if any,             %%
%% should be provided in the                %%
%% funding section.                         %%
%%%%%%%%%%%%%%%%%%%%%%%%%%%%%%%%%%%%%%%%%%%%%%
\begin{funding}
 This research work was funded by the European Union (ERC NEMO 788851).
\end{funding}

\end{document}